\documentclass{amsart}
\usepackage{amssymb,amsfonts, color,epsf}
\usepackage{appendix}
\usepackage [cmtip,arrow]{xy}
\xyoption{all}
\usepackage {pb-diagram,pb-xy}
\usepackage{enumerate}
\usepackage{accents}
\usepackage[noautoscale]{youngtab}

\usepackage{subcaption}

\ifx\pdfpageheight\undefined
   \usepackage[dvips,colorlinks=true,linkcolor=blue,citecolor=red,%
      urlcolor=green]{hyperref}
   \usepackage[dvips]{graphicx}
   \makeatletter
   \edef\Gin@extensions{\Gin@extensions,.mps}
   \makeatother
\else
 \usepackage[pdftex]{graphicx}
   \usepackage[bookmarksopen=false,pdftex=true,breaklinks=true,%
      backref=page,pagebackref=true,plainpages=false,%
      hyperindex=true,pdfstartview=FitH,colorlinks=true,%
      pdfpagelabels=true,colorlinks=true,linkcolor=blue,%
      citecolor=red,urlcolor=green,hypertexnames=false%
      ]%
   {hyperref}
\fi
\usepackage{bm}
\usepackage{caption}

\usepackage{tikz-cd}
\makeatletter
\tikzset{
  column sep/.code=\def\pgfmatrixcolumnsep{\pgf@matrix@xscale*(#1)},
  row sep/.code   =\def\pgfmatrixrowsep{\pgf@matrix@yscale*(#1)},
  matrix xscale/.code=%
    \pgfmathsetmacro\pgf@matrix@xscale{\pgf@matrix@xscale*(#1)},
  matrix yscale/.code=%
    \pgfmathsetmacro\pgf@matrix@yscale{\pgf@matrix@yscale*(#1)},
  matrix scale/.style={/tikz/matrix xscale={#1},/tikz/matrix yscale={#1}}}
\def\pgf@matrix@xscale{1}
\def\pgf@matrix@yscale{1}
\makeatother

\usepackage{amsmath,amssymb,amsfonts}
\newtheorem{theorem}{Theorem}
\newtheorem{lemma}{Lemma}
\newtheorem{corollary}{Corollary}
\newtheorem{proposition}{Proposition}

\newtheorem*{theorem*}{Theorem}
\newtheorem{claim}{Claim}
\theoremstyle{definition}
\newtheorem{definition}{Definition}
\newtheorem{example}{Example}

\newtheorem{notation}{Notation}

\usepackage{algorithm}% http://ctan.org/pkg/algorithms
\usepackage{algpseudocode}% http://ctan.org/pkg/algorithmicx

\usepackage{tikz}

\algnewcommand\algorithmicinput{\textbf{Input:}}
\algnewcommand\INPUT{\item[\algorithmicinput]}
\algnewcommand\algorithmicoutput{\textbf{Output:}}
\algnewcommand\OUTPUT{\item[\algorithmicoutput]}

\algnewcommand\algorithmicproc{\textbf{Procedure:}}
\algnewcommand\PROCEDURE{\item[\algorithmicproc]}

\newlength{\continueindent}
\usepackage{etoolbox}
\makeatletter
\newcommand*{\ALG@customparshape}{\parshape 2 \leftmargin \linewidth \dimexpr\ALG@tlm+\continueindent\relax \dimexpr\linewidth+\leftmargin-\ALG@tlm-\continueindent\relax}
\apptocmd{\ALG@beginblock}{\ALG@customparshape}{}{\errmessage{failed to patch}}
\makeatother
\theoremstyle{remark}
\newtheorem{remark}{Remark}

\definecolor{DarkBlue}{rgb}{0,0.1,0.55}

\numberwithin{equation}{section}

\newcommand {\hide}[1]{}
 \newcommand {\sign} {\mbox{\bf sign}}

\newcommand {\junk}[1]{}

\newcommand {\R} {\mathrm{R}}

\newcommand {\D}     {\mbox{\rm D}}

\newcommand {\Sphere}{\mbox{${\bf S}$}}     % Sphere
\newcommand {\Z}  {\mathbb{Z}}
 
\newcommand {\Q}         {\mathbb{Q}}

\newcommand{\Ww}    {\mathbf{W}}
\newcommand{\Ss}    {\mathbf{S}}

\newcommand{\F}{\mathbb{F}}

\newcommand {\ZZ} {{\rm Z}}
\newcommand {\RR} {{\mathcal R}}

\newcommand {\la}   {{\langle}}
\newcommand {\ra}   {{\rangle}}
\newcommand {\eps} {{\varepsilon}}

\newcommand{\card}{\mathrm{card}}
\newcommand{\rank}{\mathrm{rank}}

\def\addots{\mathinner{\mkern1mu
\raise1pt\vbox{\kern7pt\hbox{.}}
\mkern2mu\raise4pt\hbox{.}\mkern2mu
\raise7pt\hbox{.}\mkern1mu}}

\newcommand{\HH}  {\mbox{\rm H}}

\newcommand{\x}{\mathbf{x}}

\newcommand{\y}{\mathbf{y}}

\newcommand{\z}{\mathbf{z}}

\newcommand{\Par}{\mathrm{Par}}

\newcommand{\Ind}{\mathrm{Ind}}

\newcommand{\GL}{\mathrm{GL}}

\newcommand{\length}{\mathrm{length}}

\newcommand{\gdom}{{\;\underline{\triangleright}\;}}

\newcommand{\ind}{\mathrm{ind}}

\newcommand{\mult}{\mathrm{mult}}

\newcommand{\Comp}{\mathrm{Comp}}
\newcommand{\CompMax}{\mathrm{CompMax}}
\newcommand{\CompMin}{\mathrm{CompMin}}

\newcommand{\Weyl}{\mathcal{W}}
\newcommand{\w}{\mathbf{w}}

\newcommand{\Coxeter}{\mathrm{Cox}}
\newcommand{\Bl}{\mathrm{Bl}}

\begin{document}
\title[Cohomology of symmetric semi-algebraic sets]
{
Vandermonde varieties, mirrored spaces, and the cohomology of symmetric semi-algebraic sets
}

%    Information for first author
\author{Saugata Basu}
%    Address of record for the research reported here
\address{Department of Mathematics,
Purdue University, West Lafayette, IN 47906, U.S.A.}
%    Current address
\email{sbasu@math.purdue.edu}

\author{Cordian Riener}
\address{Department of Mathematics and Statistics, 
UiT The Arctic University of Norway,  
9037 Troms\o{}, Norway}
\email{cordian.riener@uit.no}

\date{\textbf{\today}}
\thanks{
Basu was  partially supported by NSF grants
CCF-1618918, DMS-1620271 and CCF-1910441.
Riener was supported by TFS grant 17\_matte\_CR. \\
Communicated by Peter Bürgisser.}

\begin{abstract}
Let $\R$ be a real closed field.
We prove that for each fixed $\ell, d \geq 0$, there exists
 an algorithm that takes as input a quantifier-free first order formula
 $\Phi$
 with atoms $P=0, P > 0, P < 0 \text{ with } P \in  \mathcal{P} \subset \D[X_1,\ldots,X_k]^{\mathfrak{S}_k}_{\leq d}$,
 where $\D$ is an ordered domain contained in $\R$,
 and computes the ranks of the first $\ell+1$ cohomology groups,
 of the symmetric semi-algebraic set defined by $\Phi$.
 The complexity of this algorithm (measured by the number of arithmetic operations in $\D$)
is bounded by a \emph{polynomial}  in $k$ and $\card(\mathcal{P})$
 (for fixed $d$ and $\ell$). This result contrasts with
 the $\mathbf{PSPACE}$-hardness of the problem of computing just  
 the zero-th Betti number (i.e. 
 the number of semi-algebraically connected components) 
in the general 
case for $d \geq 2$  (taking the ordered domain $\D$ to be equal to $\mathbb{Z}$).

The above algorithmic result is built on  new representation theoretic results on the cohomology 
of symmetric semi-algebraic sets. 
We prove that the Specht modules corresponding to partitions having long lengths cannot occur 
in the isotypic decompositions of 
low dimensional cohomology modules of closed semi-algebraic sets defined 
by symmetric polynomials having small degrees.  This result generalizes prior results obtained by the authors giving restrictions  on such partitions in terms of  their ranks,   
and is the key technical tool in the design of the algorithm mentioned in the previous paragraph. 
\end{abstract}
\maketitle

\tableofcontents

\section{Introduction and Main Results}
\label{sec:intro}
Throughout the paper we fix a real closed field, which we will denote by  $\R$ (there is no harm in assuming $\R = \mathbb{R}$). We assume familiarity
with the basic notions of semi-algebraic geometry \cite{BCR,BPRbook2} -- especially, definitions
of semi-algebraic sets, their homology and cohomology groups and  main properties. 

We will use the following notation.
\begin{notation}[Betti numbers]
  Let $S \subset\R^{k}$ be any  semi-algebraic set.
  We denote by $b_{i} (S) =  \dim_{\Q}  \HH^{i} (S,\Q)$
  (here and everywhere else in this paper without further mention we only consider cohomology with rational coefficients and we will denote $\HH^i(S)=  \HH^i(S,\Q)$).
 It is worth noting  that the precise  definition of  the cohomology groups 
 $\HH^i(S)$  requires some care if the semi-algebraic set $S$ is defined over an arbitrary (possibly non-archimedean) real closed field. For details we refer to \cite[Chapter 7,  Section 5]{BPRbook2}. 
 \end{notation}

\subsection{Background and main algorithmic result}
\label{subsec:background}
The algorithmic problem of computing Betti numbers 
of arbitrary semi-algebraic subsets of $\R^k$  is a central and 
extremely well-studied
problem in algorithmic semi-algebraic geometry.  It has many ramifications, ranging from applications in 
the theory of computational complexity where it plays the role of `generalized counting' in real models of computation (see \cite{Burgisser-Cucker-04, BZ09}),
to robot motion planning  where the problem of computing the 
zero-th Betti number  (that is the number of  connected components)  of the free space of a robot
which can be modeled as a semi-algebraic set, is a central problem \cite{SS,Canny87}).

It is well-known that
the Betti numbers of semi-algebraic subsets  of $\R^k$ satisfy a singly exponential (in $k$) upper bound (see for example
\cite[Theorem 7.38]{BPRbook2}).
The singly exponential dependence on $k$ of the bound is moreover unavoidable 
as shown by the following (key) example.

\begin{example}[Key example]
\label{eg:basic1}
Let 
\begin{equation}
\label{eqn:eg:basic1}
F_{k,d,\eps} = \sum_{i=1}^k \left( \prod_{j=0}^{d-1}(X_i - j)^2\right) - \eps.
\end{equation}
Then, for $0 < \eps \ll 1$, the set of real zeros, $V_{d,k,\eps}$ of $F_{k,d,\eps}$ in $\R^k$ consists of $d^k$ semi-algebraically connected components --
each of which is semi-algebraically homeomorphic to a small sphere.
Thus,
\[
b_0(V_{d,k,\eps})  = b_{k-1}(V_{d,k,\eps}) = d^k,
\]
and both grow exponentially  in $k$ (for fixed $d$). 
\end{example}

A  common belief 
in algorithmic semi-algebraic geometry is that topological invariants satisfying
a certain bound should in fact be computable by algorithms with complexity bounded by roughly the
same estimate. From this point of view one expects that there should exist algorithms for computing the 
Betti numbers of semi-algebraic sets with complexity bounded singly exponentially. 
Indeed, algorithms for computing the zero-th Betti number (i.e. the number of semi-algebraically connected components)
of semi-algebraic sets have been investigated in depth, and nearly optimal algorithms are known for this problem 
\cite{BPR99,BMF2014}.
An algorithm with singly exponential complexity is known for
computing the first Betti number of semi-algebraic sets is given in \cite{BPRbettione}, and then extended to the first
 $\ell$ (for any fixed $\ell$) Betti numbers in \cite{Bas05-first}.
The Euler-Poincar\'e characteristic, which is the alternating sum of the Betti numbers, is easier to compute,
and a singly exponential algorithm for computing it is known \cite{Basu1,BPR-euler-poincare}. \\

While many advances have been made in recent years \cite{Bas05-first,Bas05-topbetti,BPRbettione,BCL}
the best algorithm for computing \emph{all}  the Betti numbers of any given semi-algebraic set $S \subset \R^k$ 
still has doubly exponential (in $k$) complexity, even in the case where the degrees of the defining polynomials
are assumed to be bounded by a constant ($\geq 2$) \cite{SS}. 
The existence of algorithms with \emph{singly exponential complexity}  for computing all the Betti numbers of 
a given semi-algebraic set is considered to be a major open question in algorithmic semi-algebraic geometry
(see the survey \cite{Basu-survey}). \\

One important reason why the problem of designing an algorithm for computing the Betti numbers of semi-algebraic sets 
with singly exponential complexity  is open,  is that while the Betti numbers of 
semi-algebraic sets are bounded by a singly exponential function, the best known algorithm
for obtaining semi-algebraic triangulation has doubly exponential complexity \cite{SS}. \\

\begin{remark}[Other models]
\label{rem:BCL}
We remark here that  by the word `algorithm'  in the previous paragraphs we are referring only to algorithms that work correctly for all inputs and whose complexity is uniformly bounded i.e. bounded in terms of the degrees
and the number of input polynomials and independent of the actual coefficients of the polynomials. In contrast to this,  there has been very exciting recent work where the authors have given algorithms  with singly exponential complexity for computing all the 
Betti numbers of semi-algebraic sets \cite{BCL2019, BCT2020.1, BCT2020.2}. However, the complexities of these 
algorithms depend in addition to the degrees and the number of polynomials,  also on the `condition number' of the input. The condition number can be infinite if the given input is ill-conditioned. Thus, such algorithms will fail to produce any result on certain inputs. 
In this paper we will be concerned with exact algorithms that work for all possible inputs.
\end{remark}

From the point of view of lower bounds,  the problem of  
computing even  the number of connected components (i.e. the zero-th Betti number) of general (not necessarily symmetric)
semi-algebraic sets defined by polynomials of degrees bounded by any constant $d \geq 2$ is a $\textbf{PSPACE}$-hard
problem  \cite{Reif79},
and thus unlikely to have algorithms with polynomially bounded complexity. \\

In what follows, we will consider the algorithmic problem of computing the Betti numbers of semi-algebraic
sets in the presence of an additional important property -- namely symmetry.

\subsubsection{Brief History}
\label{subsubsec:brief-history}
The study of efficient algorithms for computing topological invariants of symmetric semi-algebraic sets 
has a shorter history than of such algorithms for arbitrary semi-algebraic set. 
Using the so called `degree principle' proved by Timofte \cite{Timofte03} and Riener \cite{Riener}, one can design an algorithm for deciding emptiness of symmetric algebraic sets in $\R^k$ defined by symmetric polynomials of degree $d$, having complexity $k^{O(d)}$ (i.e.  polynomial in $s,k$ for fixed $d$).
The algorithmic questions of computing the \emph{equivariant Betti numbers} 
(i.e.  the dimensions of $\HH_{\mathfrak{S}_k}^*(S)$ -- see the end of the paragraph for definition),
and also the Euler-Poincar\'e characteristics of symmetric semi-algebraic sets $S \subset \R^k$
were considered by the authors of the current paper. 
In \cite{BC-selecta}, an algorithm with polynomially bounded complexity (polynomial in $k$  and the number of polynomials used in the definition of $S$, for fixed $d$) 
was described for computing all the equivariant Betti numbers of a closed symmetric semi-algebraic set $S \subset \R^k$ defined by a formula involving at most $s$ symmetric polynomials of degree bounded by $d$. Since we consider cohomology with rational coefficients and because $\mathfrak{S}_k$ is a finite group, there is an isomorphism $\HH^*(S/\mathfrak{S}_k) \cong \HH_{\mathfrak{S}_k}^*(S)$, and hence this amounts to computing the Betti numbers of the quotient.
In \cite{BC-conm}, an algorithm with polynomially bounded complexity (better than that of the algorithm mentioned above) was given for computing the equivariant as well as the ordinary Euler-Poincar\'e characteristics
of symmetric semi-algebraic sets.

Before continuing further we introduce some useful notation. \\
\subsubsection{Notation}
\label{subsec:basic-notation-definition}
\begin{notation}[Zeros]
\label{not:zeros}
  For $P \in \R [X_{1} , \ldots ,X_{k}]$,  we denote by $\ZZ(P, \R^{k})$  the set of zeros of $P$ in
  $\R^{k}$. More generally, for any finite set
  $\mathcal{P} \subset \R [ X_{1} , \ldots ,X_{k} ]$, we denote by $\ZZ
  (\mathcal{P}, \R^{k})$ the set of common zeros of $\mathcal{P}$ in
  $\R^{k}$.  
\end{notation}

\begin{notation}[Realizations, $\mathcal{P}$- and $\mathcal{P}$-closed semi-algebraic sets]
  \label{not:sign-condition} 
  For any finite family of polynomials $\mathcal{P}
  \subset \R [ X_{1} , \ldots ,X_{k} ]$, we call an element $\sigma \in \{
  0,1,-1 \}^{\mathcal{P}}$, a \emph{sign condition} on $\mathcal{P}$. For
  any semi-algebraic set $Z \subset \R^{k}$, and a sign condition $\sigma \in
  \{ 0,1,-1 \}^{\mathcal{P}}$, we denote by $\RR (\sigma ,Z)$ the
  semi-algebraic set defined by $$\{ \mathbf{x} \in Z \mid \sign (P (
  \mathbf{x})) = \sigma (P)  ,P \in \mathcal{P} \},$$ and call it the
  \emph{realization} of $\sigma$ on $Z$. 
 
  More generally, we call any
  Boolean formula $\Phi$ with atoms, $P =0, P < 0, P>0, \text{ for } P \in \mathcal{P}$, to
  be a \emph{$\mathcal{P}$-formula}. We call the realization of $\Phi$,
  namely the semi-algebraic set
  \begin{eqnarray*}
    \RR \left(\Phi \right) & := & \left\{ \mathbf{x} \in \R^{k} \mid
    \Phi (\mathbf{x}) \right\}
  \end{eqnarray*}
  a \emph{$\mathcal{P}$-semi-algebraic set}. 
  
  Finally, we call a Boolean
  formula without negations, and with atoms $P \geq 0, P\leq 0$,  where $P\in \mathcal{P}$, to be a 
  \emph{$\mathcal{P}$-closed formula}, and we call
  the realization, $\RR \left(\Phi \right)$, a \emph{$\mathcal{P}$-closed
  semi-algebraic set}.
  
\end{notation}

 \begin{notation}[Symmetric polynomials of bounded degrees]
\label{not:ring-of-symmetric}
For all $d,k \geq 0$,
we will denote by $\R[X_1,\ldots,X_k]_{\leq d}^{\mathfrak{S}_k}$ the subspace of the polynomial 
ring $\R[X_1,\ldots,X_k]$ consisting of symmetric polynomials of degree at most $d$.
\end{notation}

\begin{definition}[Symmetric semi-algebraic sets]
\label{def:symmetric-sa}
We say that a semi-algebraic $S \subset \R^k$ is \emph{symmetric}
if it is stable under the standard action of the symmetric group $\mathfrak{S}_k$ permuting coordinates.
\end{definition}

Since we will discuss complexities of various algorithms we also make precise the notion
of complexity that we are going to use.

\begin{definition}[Definition of complexity]
\label{def:complexity}
In our algorithms we will usually take as input polynomials with coefficients belonging to an ordered domain (say $\D$).
By complexity of an algorithm we will mean the number of arithmetic operations and comparisons in the domain $\D$.
Since $\Z$ is always a subring of $\D$, this will include operations involving integers. If $\D = \mathbb{R}$, then
the complexity of our algorithm will agree with the  Blum-Shub-Smale notion of real number complexity \cite{BSS}.
In case, $\D = \Z$, then we are able to deduce the bit-complexity of our algorithms in terms of the bit-sizes of the coefficients
of the input polynomials, and this will agree with the classical (Turing) notion of complexity.
\end{definition} 

We are now in a position to state our main algorithmic result.
\subsubsection{Main Algorithmic result}
\label{subsubsec:algorithmic}
\begin{theorem}
\label{thm:alg}
Let $\D$ be an ordered domain contained in a real closed field $\R$, and let $\ell, d \geq 0$.
There exists an algorithm with takes as input a finite set 
$\mathcal{P} \subset \D[X_1,\ldots, X_k]^{\mathfrak{S}_k}_{\leq d}$, and a $\mathcal{P}$-formula $\Phi$,
and computes
the tuple of integers
\[
(b_0(\RR(\Phi)), \ldots, b_\ell(\RR(\Phi))).
\] 

The complexity of the algorithm, measured by the number of arithmetic operations in 
$\D$,  is bounded by
$(s k d)^{2^{O(d+\ell)}}$.

If $\D = \Z$, and the bit-sizes of the coefficients of the input is bounded by $\tau$, then the bit-complexity
of the algorithm is bounded by 
\[
(\tau s k d)^{2^{O(d+\ell)}}.
\]
\end{theorem}

\begin{remark}[Polynomiality]
\label{rem:complexity}
Note that the complexity of the algorithm in Theorem~\ref{thm:alg}  
is bounded by a polynomial in $s$ and $k$ for every fixed $\ell,d$.
\end{remark}

Note that as mentioned previously,
the analogous algorithmic problem of computing Betti numbers of general
(not necessarily symmetric)  semi-algebraic sets defined by polynomials of degree bounded by any fixed constant $d$ is a
$\mathbf{PSPACE}$-hard problem for $d \geq 2$ (with the coefficients of the input polynomials belonging to $\mathbb{Z}$), 
and thus unlikely to admit algorithms with polynomially bounded complexity. \\

\subsubsection{New ideas}
Several new ideas (compared to previous algorithms for computing Betti numbers of semi-algebraic sets) 
appear in the design of the algorithm cited in Theorem~\ref{thm:alg}. \\

We begin by replacing the given set by a closed and bounded one defined by symmetric polynomials 
satisfying the same degree bound as the input polynomials, and whose cohomology
groups are isomorphic to those of the given set up to dimension $\ell$.  
The key new idea is to utilize the $\mathfrak{S}_k$-module
structure of the cohomology groups of this new closed and bounded semi-algebraic set. This reduces the problem of 
computing the dimensions of the cohomology groups of the original set, 
to that of computing the multiplicities of the various
Specht modules appearing in the cohomology groups (up to dimension $\ell$) of the new set.
The sought after Betti numbers can then be recovered from these multiplicities. \\

In order to compute the multiplicities of the various Specht modules, we leverage certain techniques originating in the study of cohomology groups of  mirrored spaces \cite{Davis-book}. These  techniques form the basis of the proofs of
our representation theoretic results (Theorems~\ref{thm:main1} and \ref{thm:vandermonde}).
On the algorithmic front they help us in two ways. 
Firstly, (in small dimensions) it
guarantees that only a polynomially bounded many of the multiplicities to be computed can be non-zero, and this 
restricts the set of partitions that enters into the computation.
Secondly,
it allows us to obtain a \emph{dimension reduction}, reducing the problem of computing the multiplicities 
for any given closed and bounded  semi-algebraic set  defined in terms of symmetric polynomials of 
degrees bounded by $d$, to the problem of computing the Betti numbers of pairs of semi-algebraic subsets,
which are not symmetric any more but contained in a much smaller ($O(d+\ell)$) dimensional space. For the latter
problem it suffices to use the standard algorithms mentioned previously. 
We refer the reader to Section \ref{subsec:alg:outline} for a more detailed outline.

\subsection{Representation-theoretic results}
A key step in the proof of Theorem~\ref{thm:alg} as outlined above is the computation of the multiplicities
of the Specht modules in the cohomology modules of the given semi-algebraic sets. For this we need to consider
the isotypic decomposition of cohomology modules.
In the next three subsections (namely, Sections~\ref{subsubsec:isotypic}, \ref{subsubsec:orbitspace} and \ref{subsubsec:general}) we provide some background and survey prior results, and state the new results in Section~\ref{subsubsec:rep-theory-results}.
Finally, in Section~\ref{subsubsec:Vandermonde} we state a representation-theoretic result about the cohomology of a class of very well-studied symmetric varieties (namely Vandermonde varieties) which plays a key role in the proofs of the main theorems of this paper. This result could also be of independent interest.

\subsubsection{Isotypic decomposition of cohomology modules}
\label{subsubsec:isotypic}
Despite the worst case exponential behavior of the  Betti numbers of symmetric varieties, there is one handle we have
on them that makes their behavior tame, at least when the degrees of the defining polynomials are held fixed.
The action of the symmetric group $\mathfrak{S}_k$ on symmetric semi-algebraic sets
$S \subset \R^k$ induces an action on the cohomology spaces $\HH^*(S)$, giving $\HH^*(S)$ the structure of 
\emph{a finite dimensional $\mathfrak{S}_k$-module} (see Definition~\ref{def:G-module} in Section~\ref{subsec:primer}). 
General facts from group representation theory (see Appendix~\ref{sec:Appendix}) 
then tell us that
the $\mathfrak{S}_k$-module $\HH^*(S)$ admits a canonically defined  \emph{isotypic decomposition} into a direct sum
of sub-$\mathfrak{S}_k$-submodules, each of which is a multiple of certain irreducible $\mathfrak{S}_k$-modules (see Theorem~\ref{thm:isotypic}, Section~\ref{subsec:primer}).
The irreducible $\mathfrak{S}_k$-modules are well studied, and they are in bijection with the finite set of partitions
of the number $k$ -- the module corresponding to the partition $\lambda \vdash k$ will be denoted by $\mathbb{S}^\lambda$ in what follows, and is called the \emph{Specht} module corresponding to $\lambda$
(see  Definition~\ref{def:Specht} in Section~\ref{subsec:Specht} for the precise definition of these modules).

Thus the isotypic decomposition of $\HH^*(S)$ gives a direct sum decomposition
\[
\HH^i(S) \cong_{\mathfrak{S}_k} \bigoplus_{\lambda \vdash k} m_{i,\lambda}(S)  \mathbb{S}^\lambda,
\]
the non-negative integer $m_{i,\lambda}(S)$ is called the \emph{multiplicity} of $\mathbb{S}^\lambda$ in $\HH^i(S)$.

The dimension
of the Specht module $\mathbb{S}^\lambda$, has a simple expression
\[
\dim_\mathbb{Q} \mathbb{S}^\lambda =  \frac{k!}{\mbox{product of the hook lengths of the boxes in the Young diagram of 
$\lambda$}}
\] 
which is sometimes called the \emph{hook length formula}. These dimensions could be exponentially big even for relatively simple partitions (say the partition $(k/2,k/2)$ for even $k$).  
Thus,  knowing the multiplicities $m_{i,\lambda}(S), \lambda \vdash k$,  allows one to compute the dimension of $\HH^i(S)$,
and thus the Betti numbers of $S$. However, note that the number of partitions of $k$ is exponentially large (due to a result of  Erd\H{o}s  and Lehner \cite{Erdos-Lehner}). Thus, this method is at best of exponential complexity, 
unless we can restrict a priori  the 
number of partitions to consider (i.e. those that are allowed to appear in the isotypic decomposition of the 
cohomology modules of symmetric semi-algebraic sets that we are considering). \\

In order to compute the multiplicities efficiently, we prove new  quantitative results on the 
representations of the symmetric group that can appear as cohomology modules of the symmetric
semi-algebraic sets under consideration. These results might be of independent interest.
In order to relate the new results with prior work and to put them in context,
we first survey some known results in the next two sections.

\subsubsection{Partitions of length one and cohomology of the orbit space}
\label{subsubsec:orbitspace}
The partition $(k) \vdash k$ having length one plays a special role. The corresponding Specht-module
$\mathbb{S}^{(k)}$ is the one dimensional \emph{trivial representation}  of $\mathfrak{S}_k$ (we also denote it by $1_{\mathfrak{S}_k}$), and the isotypic
component of $\HH^i(S)$ corresponding to  the partition $(k)$ is thus isomorphic to the fixed part
$\HH^i(S)^{\mathfrak{S}_k}$ of $\HH^i(S)$, which in turn is isomorphic to 
$\HH^i(S/\mathfrak{S}_k)$ (see \cite{BC-imrn} for details and subtleties regarding these isomorphisms).
We obtain that the multiplicity of $\mathbb{S}^{(k)}$ in the cohomology module $\HH^i(S)$ gives the
$i$-th Betti number, $b_i(S/\mathfrak{S}_k)$. Thus, the problem of computing the 
dimension of the
cohomology of the quotient
$S/\mathfrak{S}_k$ 
(or equivalently the space of orbits)
is a special case of computing a multiplicity of a particular Specht-module in 
$\HH^*(S)$. 
We examine this case closely in the next subsection. \\

It is clear that even in the presence of symmetry the Betti numbers of semi-algebraic sets can be exponentially 
large (cf. Example~\ref{eg:basic1}). 
However, if in Example~\ref{eg:basic1} we set $\eps = 0$, and 
consider the \emph{orbits}  of the action of the symmetric group $\mathfrak{S}_k$ on  the real algebraic set 
$V_{d,k} = V_{d,k,0}$ defined in Example~\ref{eg:basic1}, then
the number of orbits of this action equals the zero-th Betti number of the quotient $V_{d,k}/\mathfrak{S}_k$. 
(Note that for any  symmetric semi-algebraic set $S\subset\R^k$ the corresponding  orbit space  $S/\mathfrak{S}_k$ can be constructed as the image of a polynomial map and thus  is again semi-algebraic  \cite{brocker1998symmetric,Procesi-Schwarz}). \\

It is not too difficult to see that the orbit of a point $\x=(x_1,\ldots,x_k) \in V_{d,k}$ is determined by the tuple $\lambda(\x) = (\lambda_1,\ldots,\lambda_{d})$, where
$\lambda_i = \card(\{j \mid x_j = i\})$. 

Thus, the number of orbits of $V_{d,k}$, and thus the sum of the Betti numbers of the quotient $V_{d,k}/\mathfrak{S}_k$ equals
$\binom{k + d -1}{d-1}$, which satisfies the inequalities
\begin{equation*}
c_d \cdot k^{d-1} \leq \binom{k +d -1}{d-1}  \leq C_d \cdot k^{d-1},
\end{equation*}
where $c_d, C_d$ are constants that depend only on $d$. 

Note that 
\[
V_{d,k} = \ZZ(F_{d,k,0},\R^k),
\]
and
$F_{d,k,0} \in \R[X_1,\ldots,X_k]^{\mathfrak{S}_k}_{\leq 2d}$ (cf. Eqn. \eqref{eqn:eg:basic1}).
Moreover, notice that unlike the Betti numbers of $V_{d,k}$ itself,
the Betti numbers of the quotient, $V_{d,k}/\mathfrak{S}_k$, are bounded by a \emph{polynomial} in $k$ 
(for fixed $d$), and moreover the
degree of this polynomial is $d-1$. \\

In fact, the following general theorem is proved in ~\cite[Theorem 6]{BC-selecta} of which the phenomenon exhibited above is a particular case.

\begin{theorem}\cite{BC-selecta}
 \label{thm:bound}
 Let $S \subset \R^k$ be a $\mathcal{P}$-closed semi-algebraic set, where 
 \[
 \mathcal{P} \subset \R[X_1,\ldots,X_k]^{\mathfrak{S}_k}_{\leq d}, 
 \]
 $\card(\mathcal{P}) = s$ and 
 $d> 1$.
  Then,
 \begin{eqnarray}
 b(S/\mathfrak{S}_k) &=&   d^{O(d)} s^d k^{\lfloor d/2 \rfloor -1} \mbox{ if }  1 < d  \ll s,k.
 \end{eqnarray}
 \end{theorem}

The following theorem which also appears in \cite[Theorem 10]{BC-selecta} indicates  that the orbit-space case is markedly different from the general (non-symmetric) case from the point of view of algorithmic complexity as well.

\begin{theorem}\cite{BC-selecta}
\label{thm:algorithm}
For every fixed $d \geq 0$, there exists an algorithm that takes as input a 
$\mathcal{P}$-closed formula $\Phi$, where $\mathcal{P} \subset \R[X_1,\ldots,X_k]^{\mathfrak{S}_k}_{\leq d}$, and outputs 
$b_i(S/\mathfrak{S}_k), 0 \leq i < d$, where $S = \RR(\Phi,\R^k)$. The complexity of this algorithm  is bounded by 
$(\card(\mathcal{P}) k d)^{2^{O(d)}}$.
 \end{theorem}

Notice that for fixed $d$ the complexity of the algorithm in Theorem \ref{thm:algorithm} is polynomial in $\card(\mathcal{P})$ and $k$. Taken together, Theorems~\ref{thm:bound} and \ref{thm:algorithm} show a 
\emph{dramatic reduction of complexity} -- both topological and algorithmic -- 
when passing from a symmetric variety to its orbit space. \\

\subsubsection{General partitions}
\label{subsubsec:general}
We now return to the study of the cohomology of a symmetric semi-algebraic set $S$ itself -- rather than its quotient.
Before proceeding further it is useful to go back to our key example (Example~\ref{eg:basic1}).

\begin{example}[Key example continued with $d=2$]
\label{eg:basic2}
We set the degree $d=2$ and $\eps =0$ in the  polynomial $F_{k,d,\eps}$ in Example~\ref{eg:basic1}, and denote
$
F_k = F_{k,2,0} = \sum_{i=1}^k X_i^2(X_i -1)^2,
$
and 
$V_k = V_{k,2,0} = \ZZ(F_k,\R^k)$.

We now describe the isotypic decomposition of $\HH^*(V_k)$. The details of this computation
appear in \cite{BC-imrn} and are omitted here.
In dimension $0$ we get:
\begin{eqnarray}
\label{eqn:eg:basic:1'}
\HH^0(V_k)  &\cong_{\mathfrak{S}_k} &  \bigoplus_{\substack{\mu \vdash k\\ \ell(\mu) \leq 2}} m_\mu \mathbb{S}^\mu,
\end{eqnarray}
where
\begin{eqnarray}
\label{eqn:eg:basic:2'}
m_\mu &=&2\mu_1 -k +1 \label{eqn:even-and-odd} \\
		&=& \mu_1 - \mu_2 +1 \\
            &\leq & k+1. \nonumber
\end{eqnarray}

Notice that for $\mu = (\mu_1,\mu_2) \vdash k$, by the hook-length formula
we have,
\begin{eqnarray}
\label{eqn:hook-length}
\dim \; \mathbb{S}^\mu &=& \frac{k! \; (\mu_1 - \mu_2+1)}{(\mu_1+1)!\mu_2!}.
\end{eqnarray}
Note that  since $\dim \HH^0(V_k,\F) =  2^k$, we obtain as a consequence  (from \eqref{eqn:even-and-odd} and 
 \eqref{eqn:hook-length}) the slightly non-obvious  identity
 \begin{eqnarray}
 \label{eqn:eg:basic:3'}
 2^k &=&
 \sum_{\substack{
 \mu_1 \geq \mu_2\geq 0\\ \mu_1+\mu_2 =k} }(\mu_1 - \mu_2 +1)\cdot \left( \frac{k!(\mu_1 - \mu_2 +1)}{(\mu_1+1)!\mu_2!}\right).
\end{eqnarray}
Notice that Eqns. \eqref{eqn:eg:basic:1'}, \eqref{eqn:eg:basic:2'}, and \eqref{eqn:eg:basic:3'} illustrate
the phenomenon of how an exponentially large dimensional cohomology group is built out of a relatively small (i.e. polynomially bounded) number of pieces -- each of which is a multiple (with polynomially bounded multiplicity) of certain Specht modules. 
\end{example}

The decomposition of the cohomology modules of a closed semi-algebraic set $S \subset \R^k$
defined by symmetric polynomials having degrees at most $d$  into isotypic components was studied in \cite{BC-imrn}, where several results were proved.
The first important result was a severe restriction on the partitions that are allowed to appear in the isotypic 
decomposition of the cohomology -- which cuts down the possibilities for the allowed partitions  \emph{from exponential
to polynomial} (for fixed $d$).
More precisely, it is shown in \cite{BC-imrn}
that with the same hypothesis as Theorem \ref{thm:bound}, 
\begin{equation}
\label{eqn:BC-imrn}
m_{i,\lambda}(S) \neq 0 \Rightarrow \rank(\lambda) < 2d,
\end{equation}
where 
$\rank(\lambda)$ is the size of the largest square (also referred to as the  `Durfee square' of the partition) that can fit inside the Young diagram (cf. Definition~\ref{def:young-diagram} in Section~\ref{subsec:Specht}) of the partition $\lambda$. 
For 
every fixed $d$, the number of partitions $\lambda$ of $k$ satisfying the condition  $\rank(\lambda) < 2d$ is \emph{polynomially bounded}  in $k$ (unlike the total number of partitions which grows exponentially). \\

The second key result obtained in \cite{BC-imrn}  is a \emph{polynomial bound}  (again for fixed $d$)
 on the multiplicities $m_{i,\lambda}(S)$ occurring in the isotypic decomposition of $\HH^i(S)$.  
Taken together -- the polynomiality of the number of allowed partitions, and the polynomiality of their multiplicities -- gives rise to the hope (via the `common belief'  alluded to before),
that the Betti numbers of symmetric semi-algebraic sets defined by symmetric polynomials of degrees bounded by a constant, could be computed with polynomially bounded complexity. \\

\subsubsection{New representation-theoretic results}
\label{subsubsec:rep-theory-results}
We now describe the new representation theoretic results that makes it possible to partially realize the `hope'  expressed above.
We obtain restrictions on the Specht modules, $\mathbb{S}^\lambda, \lambda \vdash k$, that are allowed to appear depending on $d$ and $k$, as well as the dimension (or the degree)  of the cohomology group under consideration.
These restrictions are of two kinds. 
Firstly, we prove that when $d$ is fixed, the Specht modules corresponding to partitions having long lengths cannot occur 
in the isotypic decompositions of 
small dimensional cohomology modules of 
semi-algebraic sets defined by symmetric polynomials of degrees bounded by $d$.
Secondly, we prove that the Specht modules corresponding to partitions having short lengths cannot occur 
in the isotypic decompositions of the high dimensional cohomology modules of 
semi-algebraic sets defined by symmetric polynomials of degrees bounded by $d$.

\begin{notation}
\label{not:par}
Recall that for any symmetric semi-algebraic subset $S \subset \R^k$ and $i \geq 0$,
we denote by $m_{i,\lambda}(S)$ the multiplicity of $\mathbb{S}^\lambda$ in the isotypic decomposition of $\HH^i(S)$, i.e., 
$m_{i,\lambda}(S)=\mult_{\mathbb{S}^\lambda}(\HH^i(S))$.
We will denote
\[
\Par_i(S) = \{\lambda \vdash k \mid  m_{i,\lambda}(S) \neq 0 \}.
\]
\end{notation}

We prove the following theorem.
The notation used in the theorems in this section is mostly standard; 
but readers unfamiliar with them should consult 
Appendix~\ref{sec:Appendix}.

\begin{theorem}
\label{thm:main1}
Let $d,k \in \Z_{>0}$
$d \geq 2$,
and
$S \subset \R^k$ be a $\mathcal{P}$-closed semi-algebraic set with 
$\mathcal{P} \subset \R[X_1,\ldots,X_k]^{\mathfrak{S}_k}_{\leq d}$.
Then, for  all $\lambda \vdash k$:
\begin{enumerate}[(a)]
\item
\label{itemlabel:thm:main1:a}
\begin{equation*}
m_{i,\lambda}(S)=0
\mbox{ for } i \leq \length(\lambda)- 2d + 1,
\end{equation*}
or equivalently,
\begin{equation*}
\max_{\lambda \in \Par_i(S)} \length(\lambda) <   i+2d-1;
\end{equation*}
\item
\label{itemlabel:thm:main1:b}
\begin{equation*}
m_{i,\lambda}(S)=0
\mbox{ for } i \geq k - \length(^{t}\lambda) +d+1,
\end{equation*}
or equivalently,
\begin{equation*}
\max_{\lambda \in \Par_i(S)}  \length(^{t}\lambda) <  k - i +d+1.
\end{equation*}
\end{enumerate}
\end{theorem}

Part~\eqref{itemlabel:thm:main1:a} of Theorem \ref{thm:main1} can be read as saying that for any fixed $i \geq 0$, 
and $S \subset \R^k$ a $\mathcal{P}$-semi-algebraic set with 
$\mathcal{P} \subset \R[X_1,\ldots,X_k]^{\mathfrak{S}_k}_{\leq d}$,

\[
\max_{\lambda \in \Par_i(S)} \length(\lambda) < i + 2d -1 = O(d).
\]

Similarly,
Part~\eqref{itemlabel:thm:main1:b} of Theorem \ref{thm:main1} can be read as saying that
\[
\max_{\lambda \in \Par_{k-i}(S))} \length(^{t}\lambda) < i + d + 1 = O(d).
\]

The following analysis of the cohomology modules of the key example (Example~\ref{eg:basic1})
shows that up to a multiplicative constant the bounds stated in Theorem~\ref{thm:main1}  on $\length(\lambda)$ and 
$\length(^{t}\lambda)$ for $\lambda \in \Par_i(S)$  are tight.

\begin{example}[Key example continued]
\label{eg:basic3}
For $d,k \in \Z_{>0}$,  and $0 < \eps \ll 1$,
consider the real algebraic set $V_{d,k,\eps}$ defined in Example~\ref{eg:basic1}.
Recall that for $0<\eps \ll 1$, $V_{d,k,\eps}$ consists of 
$d^k$ disjoint topological spheres, each sphere infinitesimally close (as a function of $\eps$) 
to one of the $d^k$ points
$\{0,\ldots,d-1\}^k \subset \R^k$.

Thus, 
 for $0 < \eps \ll 1$,
$\dim_\Q(\HH^0(V_{d,k,\eps})) = \dim_\Q(\HH^{k-1}(V_{d,k,\eps})) = d^k$, and 
and $\HH^i(V_{d,k,\eps}) = 0, i \neq 0,k-1$.
We now describe the isotypic decomposition of $\HH^i(V_{k,d,\eps})$ for $0 <\eps \ll 1$, and $i=0, k-1$. \\

In what follows, for $\lambda = (\lambda_1,\ldots,\lambda_m) \in \Z_{> 0}^m,
\sum_{i=1}^m \lambda_i = k$, we  denote by $\widetilde{\lambda}$ the partition of $k$ obtained by
permuting the $\lambda_i$'s so that they are in non-increasing order. \\

It is shown in \cite{BC-imrn} that
\begin{equation}
\label{eqn:eg:basic:5}
\HH^0(V_{k,d,\eps})  \cong_{\mathfrak{S}_k} \bigoplus_{\substack{
\lambda = (\lambda_1,\ldots,\lambda_d) \in \Z_{\geq 0}^d\\
\sum_{i=1}^d \lambda_i = k}} \left(\mathbb{S}^{\widetilde{\lambda}} \oplus \bigoplus_{\mu \gdom \widetilde{\lambda}, \mu \neq \widetilde{\lambda}} K(\mu,\widetilde{\lambda})\ \mathbb{S}^{\mu}\right),
\end{equation}
where $\gdom$ denotes the partial order often referred to as the \emph{dominance order} on the set of partitions of $k$, 
and $K(\cdot,\cdot)$ are the \emph{Kostka numbers} (see \cite{Ceccherini-book} for definitions).

It is clear from  \eqref{eqn:eg:basic:5} that there exists $\lambda \vdash k$ with $\length(\lambda) = d$, such that
 \[
 m_{0,\lambda}(V_{k,d,\eps})>0
 \] 
 which  shows that  the restriction,
 $\length(\lambda) = O(d)$
 (in the case $i=0$) in Part~\eqref {itemlabel:thm:main1:a} of Theorem \ref{thm:main1} is tight up to a multiplicative factor. \\

It follows from the $\mathfrak{S}_k$-equivariant Poincar\'e duality 
(see for example \cite[Theorem 3.23]{BC-imrn}), that
\begin{equation}
\label{eqn:eg:basic:6}
\HH^{k-1}(V_{k,d,\eps})  \cong_{\mathfrak{S}_k} \bigoplus_{\substack{
\lambda = (\lambda_1,\ldots,\lambda_d) \in \Z_{\geq 0}^d\\
\sum_{i=1}^d \lambda_i = k}} \left(\mathbb{S}^{^{t}\widetilde{\lambda}} \oplus \bigoplus_{\mu \gdom \widetilde{\lambda}, \mu \neq \widetilde{\lambda}} K(\mu,\widetilde{\lambda})\ \mathbb{S}^{^{t}\mu}\right).
\end{equation}

This shows that there exists $\lambda \vdash k$ with $\length(^{t}\lambda) = d$, such that
 \[
 m_{k-1,\lambda}(V_{k,d,\eps})>0
 \] 
So the restriction,
 $\length(^{t} \lambda) = O(d)$
 (in the case $i=0$) in the Part~\eqref{itemlabel:thm:main1:b}  of Theorem \ref{thm:main1} is also tight up to a multiplicative factor.
\end{example}

\subsubsection{Role played by Vandermonde varieties}
\label{subsubsec:Vandermonde}
The proof of Theorem~\ref{thm:main1} stated in the previous section depends crucially on a similar restriction theorem for a class of symmetric semi-algebraic sets which are particularly simple 
to define -- namely, \emph{Vandermonde varieties}.
Vandermonde varieties have been studied widely in a series of papers by Arnold \cite{Arnold}, Giventhal \cite{Giventhal}, Kostov \cite{Kostov} amongst others,
mainly from a topological point of view. The representation-theoretic results we prove in this paper 
on their cohomology modules are new and might be
of independent interest.
The restrictions on the $\mathfrak{S}_k$-module structure for 
Vandermonde varieties, produce  via an application of an argument involving  the (equivariant)  Leray
spectral sequence,  similar (slightly looser) restrictions  on the cohomology modules of arbitrary symmetric semi-algebraic sets defined by
quantifier-free formula involving qualities and inequalities of symmetric polynomials of degrees bounded by $d \leq k$  
(cf. Theorem~\ref{thm:main1}). \\

The intersections of the level sets of the  first $d$ (weighted) Newton power sums in $\R^k$ for some $d \leq k$ have been called  
Vandermonde varieties by Arnold  \cite{Arnold} and Giventhal \cite{Giventhal},   
who studied their topological properties in detail. 
When the 
weights are all equal the Vandermonde varieties are also symmetric with respect to the standard action (by permuting 
coordinates) of the symmetric group $\mathfrak{S}_k$, and thus the cohomology groups of the Vandermonde varieties acquire the structure of  finite dimensional $\mathfrak{S}_k$-modules.

\begin{remark}
\label{rem:generators}
If one replaces in the definition of Vandermonde varieties, 
the Newton power sums with any other set of generators of the ring of $\mathfrak{S}_k$-invariant polynomials (for example the elementary symmetric polynomials), the intersection of the level sets of the generators of degree at most $d$ give  the same class of real varieties. 
Indeed, Vandermonde varieties can be defined as level sets of the first $d$ generators of the
invariant ring of any finite reflection group, and many results and techniques introduced in the current paper extend to more general reflection groups. However, the case of the symmetric group is the most important from the point of view of applications, and we restrict ourselves to this special case in this paper.
\end{remark}

In their foundational work on the topic, Arnold \cite{Arnold}, Giventhal \cite{Giventhal} and Kostov \cite{Kostov},
proved that the intersection of a symmetric Vandermonde variety with the Weyl chamber in $\R^k$, defined by the inequalities 
$X_1 \leq \cdots \leq X_k$ is contractible if non-empty, which in turn implies that the quotient space of a symmetric Vandermonde variety is contractible if non-empty.  \\

As a first step towards proving Theorem~\ref{thm:main1} we  study the $\mathfrak{S}_k$-module structure of the cohomology
groups of  symmetric Vandermonde varieties  themselves  (not just their quotient space).
We prove the following theorem.

\begin{theorem}
\label{thm:vandermonde}
Let $d,k \in \Z_{>0}, d \geq 2$, 
$\y =(y_1,\ldots,y_d) \in \R^d$, and 
let $V_{d,\y}^{(k)}$ denote the Vandermonde variety defined by
$p_1^{(k)} = y_1, \ldots, p_d^{(k)} = y_d$, where $p_j^{(k)} = \sum_{i=1}^{k} X_i^j$.
Then, for  all $\lambda \vdash k$:
\begin{enumerate}[(a)]
\item
\label{itemlabel:thm:vandermonde:a}
\begin{equation*}
m_{i,\lambda}(V_{d,\y}^{(k)})=0,
\mbox{ for } i \leq \length(\lambda)- 2d + 1,
\end{equation*}
or equivalently,

\begin{equation}
\label{eqn:thm:vandermonde:a}
\max_{\lambda \in \Par_i(V^{(k)}_{d,\y})} \length(\lambda) < i+2d-1;
\end{equation}
\item
\label{itemlabel:thm:vandermonde:b}
\begin{equation*}
m_{i,\lambda}(V_{d,\y}^{(k)})=0
\mbox{ for } i \geq k - \length(^{t}\lambda)+1,
\end{equation*}
or equivalently,
\begin{equation*}
\max_{\lambda \in \Par_i(V^{(k)}_{d,\y})} \length(^{t}\lambda) <  k - i +1.
\end{equation*}
\end{enumerate}
\end{theorem}

\begin{remark}[Cases $d=1,2$]
\label{rem:d=1}
The case $d=1$ is omitted in Theorem \ref{thm:vandermonde}. Indeed, Part \eqref{itemlabel:thm:vandermonde:a} is not true as stated in the case $d=1$. In this case,
$V_{d,\y}^{(k)}$ is the hyperplane defined by the equation 
\[
\sum_{i=1}^k X_i = y_1,
\] 
and is $\mathfrak{S}_k$-equivariantly contractible to the point $\frac{1}{k}\cdot(y_1,\ldots,y_1)$.
Hence
\begin{eqnarray*}
\HH^i(V_{d,\y}^{(k)}) &\cong_{\mathfrak{S}_k}& \mathbb{S}^{(k)}, \mbox{ if $i=0$}, \\
& \cong_{\mathfrak{S}_k} & 0, \mbox{ otherwise}
\end{eqnarray*}
(recall that the Specht module $\mathbb{S}^\lambda$ for $\lambda$ equal to the trivial partition $(k)$
is isomorphic to the one-dimensional trivial representation).  
It follows that for $i=0$, 
\[
m_{i,\lambda}(V_{d,\y}^{(k)})=1\neq 0,
\] 
but 
\[
\length((k)) = 1 \not<  i +2 d - 1 =  0 + 2 -1 = 1,
\]
which violates \eqref{eqn:thm:vandermonde:a}.
 
On the other hand, the case $d=2$ already indicates that the bounds in
Theorem \ref{thm:vandermonde} is sharp.

If $d=2$ and $k \geq 3$, the Vandermonde variety $V^{(k)}_{d,\y}$  is the defined by the equation 
\[
\sum_{i=1}^k X_i = y_1, \sum_{i=1}^{k} X_i^2 = y_2,
\] 
and can be empty, a point,  or semi-algebraically homeomorphic to a sphere of dimension $k-2$
(depending on 
whether
$y_1^2 - k y_2$ is $> 0, = 0$, or $< 0$,
 respectively).
In the last case 
(i.e. when $y_1^2 - k y_2 < 0$):

\begin{eqnarray}
\nonumber
\HH^i(V_{2,\y}^{(k)}) &\cong_{\mathfrak{S}_k}& \mathbb{S}^{(k)}, \mbox{ if $i=0$}, \\
\label{eqn:d=2}
\HH^i(V_{2,\y}^{(k)}) &\cong_{\mathfrak{S}_k}& \mathbb{S}^{1^k}, \mbox{ if $i=k-2$}, \\
\nonumber
& \cong_{\mathfrak{S}_k} & 0, \mbox{ otherwise}
\end{eqnarray}
(see Subsection~\ref{eg:V-k-2} below for a proof).

It follows that for $i=k-2, k \geq 3$ and $y_2 >0$, 
\[
m_{i,\lambda}(V_{d,\y}^{(k)})= 1 \neq 0 \Rightarrow 1^k \in \Par_{k-2}(V^{k}_{2,\y}),
\] 
and
\[
\max_{\lambda \in \Par_{k-2}(V^{(k)}_{2,\y})}  \length(\lambda) = \length(1^k) = k <  k-2  +2\cdot 2 - 1 =  k+1.
\]
\end{remark}

\subsubsection{Improvements over prior work}
Theorems~\ref{thm:main1} and \ref{thm:vandermonde} are improvements
over prior results in \cite{BC-imrn} (Theorem 2.5, Part (1))  having similar flavor  in several different ways.\\

Firstly,
the restrictions (cf. \eqref{eqn:BC-imrn})   on partitions given in \cite[Theorem 2.5]{BC-imrn} are in terms of  upper bounds on their \emph{ranks}  
rather than their lengths. 
While the length of a partition is an upper bound on its rank,
a partition having small rank can be arbitrarily long.  
For example, the partition $1^k:=(1,\ldots,1)$ has rank $1$, but its length
is clearly the maximum possible, namely $k$.  \\

Secondly, the restrictions  in \cite[Theorem 2.5]{BC-imrn} 
do not take into consideration the dimension (or the degree) of the cohomology groups under consideration.
In contrast,
the restrictions on  the partitions $\lambda$ given in Theorems~ \ref{thm:main1} and \ref{thm:vandermonde}  in the 
current paper,
do depend in a strong manner on the dimension  (or the degree) of the cohomology group.
As a result in small dimensions, we obtain that only the partitions with a small length can appear unlike the restrictions
obtained in \cite{BC-imrn}, where there were no non-trivial restriction on the length. The restriction on the length
is a key ingredient in the algorithmic result obtained in this paper. \\

The results of the current paper depend on:
\begin{enumerate}[(a)]
\item
results from the cohomological study of mirrored spaces due to
Davis \cite{Davis} and Solomon \cite{Solomon1968}, 
\item
fundamental results on Vandermonde varieties due to Arnold \cite{Arnold}, Giventhal \cite{Giventhal}
and Kostov \cite{Kostov}, and
\item
a careful topological analysis of certain regular cell complexes that arise in the process
of combining these results.
\end{enumerate}

In contrast, the proofs of the results in \cite{BC-imrn} are based essentially on equivariant Morse theory which plays no role in the current paper.
The reader who is curious about the interplay of  results coming from different areas  
and how they combine together in the study of Vandermonde varieties,  can skip forward to 
Examples~\ref{eg:V-k-2} and \ref{eg:V-4-3} where the examples of Vandermonde varieties of
degree $2$ in $\R^k$, $k\geq 3$, and that of degree $3$ in $\R^4$ are worked out in full detail. \\

The rest of the paper is dedicated to the proofs  of Theorems~\ref{thm:alg},  \ref{thm:main1},  and \ref{thm:vandermonde}.
In Section~\ref{sec:reptheory-coda}, 
we prove a few preliminary results on the Solomon decomposition of the cohomology groups of
mirrored spaces that play an important role in the rest of the paper. We introduce all necessary background material referring the reader to Appendix~\ref{sec:Appendix} for 
the more basic material on representation theory of finite groups and of the symmetric groups in particular
that we utilize.
In Section \ref{sec:outline+example} we give outlines of  the proofs of Theorems~\ref{thm:main1} and \ref{thm:vandermonde},  
and also describe
two important examples illustrating the main steps.
In Section \ref{sec:proofs1}, we give the proofs of Theorems~\ref{thm:main1} and \ref{thm:vandermonde}.
In Section \ref{sec:proofs2} we give the proof of Theorem~\ref{thm:alg} after introducing the 
necessary preliminary results.

\section{Solomon modules and mirrored spaces}
\label{sec:reptheory-coda}
This section is divided into two subsections.
In the first subsection  (Subsection~\ref{subsec:Solomon}) we discuss the representation theory of the symmetric groups by viewing them  as examples of finite Coxeter groups  
drawing on the work of Solomon \cite{Solomon1968}. In particular, we show how to obtain the isotypic decomposition of the  Solomon modules (which are certain representations of symmetric groups that we define in this section), and prove certain quantitative statements about them that are key to the proofs of the main theorems of the paper. These results (namely, Propositions~\ref{prop:Solomon-to-Specht} and \ref{prop:Solomon} and Corollary~\ref{cor:Solomon-to-Specht}) are the only results from this section that are used later in the paper. \\

In the second subsection (Subsection~\ref{subsec:mirror}) we introduce mirrored spaces and discuss a key theorem
(cf. Theorem~\ref{thm:Davis}) giving a formula for the cohomology of a mirrored space in
terms of certain Solomon modules. This theorem plays a central role in the proof of Theorem~\ref{thm:vandermonde}.

\subsection{Symmetric groups as Coxeter groups and
properties of Solomon modules}
\label{subsec:Solomon}
Recall that a Coxeter pair $(\Ww,\Ss)$, consists of  a group $\Ww$ and a set of generators, $\Ss = \{ s_i \mid i \in I\}$, of $\Ww$  each having order $2$, 
and numbers $(m_{i,j})_{i,j \in I}$ such that $(s_i s_j)^{m_{ij}} = e$. \\

Our main example of a  Coxeter groups will be the symmetric group $\mathfrak{S}_k$ 
considered as a Coxeter group with 
the set of Coxeter generators, $\Coxeter(k) = \{ s_i = (i,i+1) \mid 1 \leq i \leq k-1 \}$
(here $(i,i+1)$ denotes the permutation of $(1,\ldots,k)$ which exchanges $i$ and $i+1$
keeping all other elements fixed). \\

We will need the notion of \emph{length} of an element of a Coxeter group.

\begin{notation}[Length of an element of $\Ww$]
\label{not:lenth-of-word}
Given Coxeter pair $(\Ww,\Ss)$, with $\Ss = \{ s_i \mid i \in I\}$,
and an element $w = s_{i_1}\cdots s_{i_m}  \in \Ww$, we call $m$ to be the \emph{length of $w$}
(denoted $\ell(w)$), if $m$ is minimal amongst all such expressions for $w$. 
\end{notation}

\begin{example}
\label{eg:length-of-words}
If $(\Ww,\Ss)  = (\mathfrak{S}_3,\Coxeter(3))$, the lengths of the various elements of $\mathfrak{S}_3$
viewed as permutations are displayed below.
$$
\displaylines{
\ell(123) = 0, \cr
\ell(132) = \ell(213) = 1, \cr
\ell(231) = \ell(312)   = 2, \cr
\ell(321) = 3.
}
$$
\end{example}

Following the same notation as in \cite{Davis-book},
for $J \subset \Coxeter(k)$, we denote by $\mathfrak{S}_k^J$ the subgroup of $\mathfrak{S}_k$ generated by $J$,
and let 
\[
A^J = \Q[\mathfrak{S}_k^J].
\]
We will write $N_J = \card(\mathfrak{S}_k^J)$.
For $J \subset \Coxeter(k)$, let 
\begin{eqnarray}
\label{eqn:def:xi}
\xi_J
&=&  
N_J^{-1} \sum_{w \in \mathfrak{S}_k^J} w, \\
\label{eqn:def:eta}
\eta_J
 &=&  
N_J^{-1} 
\sum_{w \in \mathfrak{S}_k^J}  (-1)^{\ell(w)} w.
\end{eqnarray}

For $P,Q \subset \Coxeter(k), P \cap Q = \emptyset$, we denote 
(following \cite{Solomon1968})

\begin{equation}
\label{eqn:def-of-Psi}
\Psi_{P,Q}= A^{P \cup Q}  \xi_P \eta_Q.
\end{equation}

\subsubsection{Algebras, tensor products and representations}
Let $\Ww$ be a group and $A = \Q[\Ww]$ be the group algebra of $\Ww$. A left ideal $I \subset A$ is then
a (left) $\Ww$-module.
Now let $\Ww',\Ww''$ be two Coxeter groups, and $A' = \Q[\Ww'], A'' = \Q[\Ww'']$ be their group algebras.
Then, the tensor product $A' \otimes_\Q A''$ is again an algebra, where the multiplication is defined by
$(a' \otimes a'')\cdot (b' \otimes b'') = a'a'' \otimes b'b''$. Moreover, $A' \otimes_\Q A''$ is isomorphic
as an $\Q$-algebra to $A = \Q[\Ww' \times \Ww'']$, where the isomorphism is given by
\[
w' \otimes w'' \mapsto (w',w''), w' \in \Ww', w'' \in \Ww''.
\]
If $\Ww',\Ww''$ are subgroups of $\Ww$, such that $\Ww$ is the (internal) direct product of $\Ww',\Ww''$, then the isomorphism,
\begin{equation}
\label{eqn:isomorphism}
A' \otimes_\Q A'' \rightarrow A
\end{equation}
 is given by $w' \otimes w'' \mapsto w'w''$. \\

Finally, if $I'$ is a left ideal of $A'$, and $I''$  a left ideal of $A''$, then $I' \otimes_\Q I''$ is a 
left ideal of the algebra $A' \otimes_\Q A''$. If we denote by $\Psi'$ (resp. $\Psi''$) the $\Ww'$-representation
(resp. $\Ww''$-representation) corresponding to $I'$ (resp. $I''$), then we will denote by
$\Psi' \boxtimes \Psi''$ the $(W' \times W'')$-representation corresponding to $I' \otimes_\Q I''$.
We will need later the following proposition.

\begin{proposition}
\label{prop:Solomon0}
Let $k > 0$, and $1 \leq q \leq k-1$. Let $P',Q' \subset \{s_1,\ldots,s_{q-1} \}$, $P'',Q'' \subset \{s_{q+1},\ldots,s_{k-1}\}$
such that  
$P' \cap Q' = P'' \cap Q'' = \emptyset$, and 
\begin{eqnarray*}
P' \cup Q' &=&  \{s_1,\ldots,s_{q-1}\},\\
P'' \cup Q'' &=& \{s_{q+1},\ldots,s_{k-1}\}.
\end{eqnarray*}
Then, 
\begin{equation}
\label{eqn:prop:Solomon}
\Psi_{P' \cup P'',Q' \cup Q''} \cong_{\mathfrak{S}_q \times \mathfrak{S}_{k-q}}  \Psi_{P',Q'} \boxtimes \Psi_{P'',Q''}.
\end{equation}
\end{proposition} 

\begin{proof}
Let $J' = P' \cup Q' = \{s_1,\ldots,s_{q-1}\}$, $J'' = P'' \cup Q'' = \{s_{q+1},\ldots,s_{k-1}\}$, and $J = J \cup J' = 
\Coxeter(k) - \{s_q\}$.
Observe first that the elements of $\mathfrak{S}_k^{J'}$ commute with the elements of $\mathfrak{S}_k^{J''}$,
$\mathfrak{S}_k^J = \mathfrak{S}_k^{J'}\mathfrak{S}_k^{J''}$, and $\mathfrak{S}_k^{J'} \cap \mathfrak{S}_k^{J''} = 
\{e\}$.
Hence it follows that 
$\mathfrak{S}_k^{J}$ is isomorphic to the  direct product of the subgroups
 $\mathfrak{S}_k^{J'}$ and $\mathfrak{S}_k^{J''}$. In particular, every element $w \in \mathfrak{S}_k^{J}$ can be written uniquely as 
 \[
 w= w' w''
 \] 
 with $w' \in \mathfrak{S}_k^{J'}$ and 
 $w'' \in \mathfrak{S}_k^{J''}$.
 Moreover, 
 \[
 \ell(w) = \ell(w')+ \ell(w'').
 \]
 
It follows from \eqref{eqn:def-of-Psi} that
$ \Psi_{P',Q'}$  (resp.  $\Psi_{P'',Q''}$) 
is the $\mathfrak{S}_k^{J'}$-representation (resp. $\mathfrak{S}_{k}^{J''}$-representation)
corresponding to the left ideal  $I' = A^{J'} \xi_{P'}\eta_{Q'}$ of $A^{J'}$  (resp. $I'' = A^{J''} \xi_{P''}\eta_{Q''}$ of
$A^{J''}$). \\

Moreover, there is an isomorphism of $\Q$-algebras  (see \eqref{eqn:isomorphism}) 
$\phi_{J',J'}:A^{J'} \otimes_\Q A^{J''} \rightarrow A^J$,
defined by $w' \otimes w'' \mapsto w'w''$.
It suffices to prove that $\phi_{J',J''}$ carries the left ideal $I' \otimes_\Q I''$ of $A^{J'} \otimes_\Q A^{J''}$ surjectively to the left  ideal $I = A^J \xi_{P' \cup P'',Q' \cup Q''}$ of $A^J$. \\
 
Since, $I  = A^{J} \xi_{P' \cup P''}\eta_{Q' \cup Q''}$ is spanned by
the elements $w \xi_{P' \cup P''}\eta_{Q' \cup Q''}, w \in \mathfrak{S}_k^{J}$ it suffices to prove that 
\[
w \xi_{P' \cup P''}\eta_{Q' \cup Q''} \in \phi_{J',J''}(A^{J'} \xi_{P'}\eta_{Q'} \otimes_\Q A^{J''} \xi_{P''}\eta_{Q''})
\] 
for every $w \in \mathfrak{S}_k^{J}$. \\

 Using the fact that every element $w \in \mathfrak{S}_k^{J' \cup J''}$ can be written uniquely as $w= w' \cdot w''$ with $w' \in \mathfrak{S}_k^{J'}$ and 
 $w'' \in \mathfrak{S}_k^{J''}$,
 with
 \[
 \ell(w) = \ell(w')+ \ell(w''),
 \]
and \eqref{eqn:def:xi} and \eqref{eqn:def:eta}, we have
\begin{eqnarray*}
\frac{N_{P'}N_{P''}}{N_{P'\cup P''}}  \xi_{P'}\xi_{P''} &=&
\xi_{P' \cup P''}, \\
\frac{N_{Q'}N_{Q''}}{N_{Q'\cup Q''}}  \xi_{Q'}\xi_{Q''} &=&
\xi_{Q' \cup Q''}.
\end{eqnarray*}
Hence
\begin{equation}
\label{eqn:prop:Solomon:PQ}
 \xi_{P' \cup P''}\eta_{Q' \cup Q''} =
\frac{N_{P'}N_{P''}N_{Q'}N_{Q''}}{N_{P'\cup P''}N_{Q' \cup Q''}}  
\xi_{P'}\xi_{P''}\eta_{Q'}\eta_{Q''}.\\
\end{equation}

Now $w$ can be written (uniquely)  as $w= w' w''$ with $w' \in \mathfrak{S}_k^{J'}$ and 
 $w'' \in \mathfrak{S}_k^{J''}$, and hence
 \begin{eqnarray*}
 &&w \xi_{P' \cup P''}\eta_{Q' \cup Q''}\\
 &=& w'w'' \xi_{P' \cup P''}\eta_{Q' \cup Q''} \\
&=& \frac{N_{P'}N_{P''}N_{Q'}N_{Q''}}{N_{P'\cup P''}N_{Q' \cup Q''}}  
w'w''\xi_{P'}\xi_{P''}\eta_{Q'}\eta_{Q''}
\mbox{ (using \eqref{eqn:prop:Solomon:PQ})}\\
&=& 
\frac{N_{P'}N_{P''}N_{Q'}N_{Q''}}{N_{P'\cup P''}N_{Q' \cup Q''}}  
w'w''\xi_{P'}\eta_{Q'}\xi_{P''}\eta_{Q''} 
 \mbox{ (elements of  $A^{J'}$ and $A^{J''}$ commute)} \\
&=&
\phi_{J',J''}\left(  \frac{N_{P'}N_{P''}N_{Q'}N_{Q''}}{N_{P'\cup P''}N_{Q' \cup Q''}}     w'\xi_{P'}\eta_{Q'} \otimes w''\xi_{P''}\eta_{Q''}\right).
 \end{eqnarray*}
This finishes the proof.
\end{proof}

\begin{notation}[Solomon modules]
\label{not:T}
For ease of notation we will denote the representation $\Psi_{\Coxeter(k) - T,T}^{(k)}$ by
$\Psi_T^{(k)}$. We will call $\Psi_T^{(k)}$ the \emph{Solomon module indexed by $T$}.
\end{notation}

\begin{remark}
\label{rem:Solomon-to-Specht1}
The Solomon modules  $\Psi_T^{(k)}$ may be understood as analogs of Specht modules (cf. Definition~\ref{def:Specht}), but defined in terms of MacMahon's tableau \cite[Vol 1, Chapter 1, Sect IV, 129.]{MacMahon} 
rather than Young's tableau (cf. Definition~\ref{def:Young-tableau})
where the role of partitions is replaced by that of compositions (cf. Notation~\ref{not:Partition1}).
Unlike the Specht modules, the representations $\Psi_T^{(k)}$ need not be irreducible
(see Example~\ref{eg:Solomon-to-Specht}).
But we are able to obtain a necessary condition for  a Specht module to appear with positive multiplicity in 
$\Psi_T^{(k)}$ using a recursive formula due to Solomon \cite[Corollary 3.2]{Solomon1968} (cf. Proposition~\ref{prop:Solomon} below).
\end{remark}

\begin{remark}
As remarked above the representations $\Psi_T^{(k)}$  need not be irreducible in general.
However, it is easy to see from \eqref{eqn:def-of-Psi}, Notation~\ref{not:T} and 
Definition~\ref{def:Specht},  that in the following two special cases, they are indeed irreducible.
\begin{eqnarray}
\label{eqn:eg:1}
\Psi_{\emptyset}^{(k)} & \cong_{\mathfrak{S}_k}& \mathbb{S}^{(k)}  \cong_{\mathfrak{S}_k} 1_{\mathfrak{S}_k},\\
\label{eqn:eg:2}
\Psi_{\Coxeter(k)}^{(k)} &\cong_{\mathfrak{S}_k}& \mathbb{S}^{(1^k)} \cong_{\mathfrak{S}_k} \mathbf{sign}_k.
\end{eqnarray}

Another easy consequence of \eqref{eqn:def-of-Psi} is
\begin{eqnarray}
\label{eqn:transpose}
\Psi_{\Coxeter(k) - T}^{(k)} &\cong_{\mathfrak{S}_k}&  \Psi_{T}^{(k)} \otimes \mathbf{sign}_k.
\end{eqnarray}
\end{remark}

\subsubsection{Relation between Solomon modules and Specht modules}
We next prove a recursive formula for computing the multiplicities of Specht modules in the Solomon modules (Proposition~\ref{prop:Solomon-to-Specht} and Corollary~\ref{cor:Solomon-to-Specht}).
We also prove a condition  (in terms of $k$ and the cardinality of $T$) on partitions $\lambda$ which needs to be satisfied for
$
\mult_{\mathbb{S}^\lambda}(\Psi_T^{(k)}) > 0
$
to hold (Proposition~\ref{prop:Solomon}).

\begin{proposition}
\label{prop:Solomon-to-Specht}
Let $k \geq 1$, 
$T \subset \Coxeter(k)$, and 
\[
q =  \max \{i \;\mid\; s_i \in T \}.
\]
Then,
\[
\ind_{\mathfrak{S}_q \times \mathfrak{S}_{k-q}
}^{\mathfrak{S}_k} \left( \Psi_{T -\{s_q\}}^{(q)} \boxtimes 1_{\mathfrak{S}_{k-q}}\right) \cong_{\mathfrak{S}_k}
 \Psi_{T -\{s_q\}}^{(k)} \oplus \Psi_{T}^{(k)}.
 \]
\end{proposition}

\begin{proof}
Let 
\begin{eqnarray*}
Q' &=& T -\{s_q\}, \\
Q'' &=& \emptyset, \\
P' &=& \{s_1,\ldots,s_{q-1}\} -  T, \\
P''&=& \{s_{q+1},\ldots,s_{k-1}\}.
\end{eqnarray*}

Notice that 
\begin{eqnarray*}
\mathfrak{S}_k^{P' \cup Q'} &\cong& \mathfrak{S}_{q}, \\
\mathfrak{S}_k^{P'' \cup Q''} &\cong& \mathfrak{S}_{k-q}, \\
\mathfrak{S}_k^{P' \cup Q' \cup Q''} &\cong& \mathfrak{S}_q \times \mathfrak{S}_{k-q}.
\end{eqnarray*}

\begin{claim}
\label{claim:proof:prop:Solomon:1}
\begin{equation}
\label{eqn:proof:prop:Solomon:1}
\Psi_{P' \cup P'',Q'}    \cong_{\mathfrak{S}_q \times \mathfrak{S}_{k-q}}  
\Psi_{Q'}^{(q)} \boxtimes 1_{\mathfrak{S}_{k-q}}.
\end{equation}
\end{claim}

\begin{proof}[Proof of Claim~\ref{claim:proof:prop:Solomon:1}]
Observe that  it follows from the definitions of $P',P'',Q', Q''$ that
\[
\Psi_{P' \cup P'',Q'}  = \Psi_{ P' \cup P'', Q' \cup Q''},
\] 
and 
\[
 \Psi_{P',Q'} \boxtimes \Psi_{P'',Q''} = \Psi_{P',Q'} \boxtimes \Psi_{\{s_{q+1},\ldots,s_{k-1}\},\emptyset}.
\]
Now,
\[
\Psi_{P' \cup P'', Q' \cup Q'', }  \cong_{\mathfrak{S}_q \times \mathfrak{S}_{k-q}}  \Psi_{P',Q'} \boxtimes \Psi_{P'',Q''}
\]
using Proposition~\ref{prop:Solomon0}.
Finally, from the fact that $P' \cup Q' = \{s_1,\ldots.s_{q-1}\}$, $Q'' = \emptyset$, and $P'' = \{s_{q+1},\ldots,s_{k-1}\}$,
we have
\[
\Psi_{P',Q'} \cong_{\mathfrak{S}_{q}} \Psi_{Q'}^{(q)},
\]
and
\[
\Psi_{ \{s_{q+1},\ldots,s_{k-1}\},\emptyset} \cong_{\mathfrak{S}_{k-q}} 1_{\mathfrak{S}_{k-q}}.
\]
This finishes the proof of the claim.
\end{proof}

\begin{claim}
\label{claim:proof:prop:Solomon:2}
\begin{equation}
\label{eqn:proof:prop:Solomon:2}
\ind_{\mathfrak{S}_q \times \mathfrak{S}_{k-q}
}^{\mathfrak{S}_k} \Psi_{P' \cup P'',Q'} \cong_{\mathfrak{S}_k} 
 \Psi_{Q'}^{(k)} \oplus \Psi_{T}^{(k)}.
\end{equation}
\end{claim}

\begin{proof}[Proof of Claim~\ref{claim:proof:prop:Solomon:2}]
Observe that 
\[
\Psi_{P' \cup P'' \cup\{s_q\},Q'} \oplus \Psi_{P' \cup P'',T} = \Psi_{Q'}^{(k)} \oplus \Psi_{T}^{(k)}.
\]
It follows directly from \cite[Corollarly 3.2]{Solomon1968} that
\[
\ind_{\mathfrak{S}_q \times \mathfrak{S}_{k-q}
}^{\mathfrak{S}_k} \Psi_{P' \cup P'',Q'} \cong_{\mathfrak{S}_k} 
\Psi_{P' \cup P'' \cup\{s_q\},Q'} \oplus \Psi_{P' \cup P'',T}
\]
which completes the proof of the claim.
\end{proof}

The proposition now follows directly from Claims~\ref{claim:proof:prop:Solomon:1} and \ref{claim:proof:prop:Solomon:2}.
\end{proof}

The following corollary of Proposition~\ref{prop:Solomon-to-Specht} will be useful in designing an algorithm
for computing isotypic decomposition of the Solomon modules $\Psi_T^{(k)}$.

\begin{corollary}
\label{cor:Solomon-to-Specht}
Let $k \geq 1$, 
$T \subset \Coxeter(k)$, and 
\[
q =  \max \{i \;\mid\; s_i \in T \}.
\]
Then,
for any $\lambda \vdash k$, 

\begin{equation}
\label{eqn:cor:Solomon-to-Specht}
\mult_{\mathbb{S}^\lambda}(\Psi_{T}^{(k)}) = \mult_{\mathbb{S}^\lambda}\left( \ind_{\mathfrak{S}_q \times \mathfrak{S}_{k-q}}^{\mathfrak{S}_k} \left( \Psi_{T -\{s_q\}}^{(q)}\boxtimes 1_{\mathfrak{S}_{k-q}}\right) \right) - \mult_{\mathbb{S}^\lambda}( \Psi_{T -\{s_q\}}^{(k)} ).
\end{equation}
\end{corollary}

\begin{proof}
Follows directly from Proposition~\ref{prop:Solomon-to-Specht} and  Schur's Lemma (Lemma~\ref{lem:Schur} in the Appendix).
\end{proof}

Before proceeding further we recall a classical formula -- namely Pieri's rule.

\begin{notation}
\label{not:Pieri}
For $0 \leq q \leq k$, and 
$\mu =(\mu_1,\ldots,\mu_m) \vdash q$, we denote by $S(\mu,k)$ the set consisting of partitions
either of the form 
$\lambda = (\lambda_1,\ldots, \lambda_m) \vdash k$
satisfying:
\begin{equation}
\label{eqn:Pieri1}
\lambda_1 \geq \mu_1 \geq \lambda_2 \geq \mu_2\geq \cdots \geq \lambda_{m} \geq 
\mu_{m}, 
\end{equation}
and 
\begin{equation}
\label{eqn:Pieri2}
\sum_{i=1}^m (\lambda_i - \mu_i) = k-q.
\end{equation}
or of the form 
$\lambda = (\lambda_1,\ldots, \lambda_m,\lambda_{m+1}) \vdash k$
satisfying:
\begin{equation}
\label{eqn:Pieri1'}
\lambda_1 \geq \mu_1 \geq \lambda_2 \geq \mu_2\geq \cdots \geq \lambda_{m} \geq 
\mu_{m}\geq \lambda_{m+1} > 0, 
\end{equation}
and 
\begin{equation}
\label{eqn:Pieri2'}
\lambda_{m+1} + \sum_{i=1}^m (\lambda_i - \mu_i) = k-q.
\end{equation}

In other words $\lambda \in S(\mu,k)$, if and only if $\lambda \vdash k$ and the Young diagram corresponding to $\lambda$ is obtained from that of $\mu$ by adding $k-q$ boxes, such that no two boxes are added in the same column.  
\end{notation}

\begin{example}
\label{eg:Pieri}
For example,
\[
S((2,1),4) = \{(3,1), (2,2), (2,1,1)\}.
\]
\end{example}

The significance of the set $S(\mu,k)$ is encapsulated in the following lemma. 
With the same notation as in Notation~\ref{not:Pieri}:
\begin{lemma}[Pieri's rule]
\label{lem:Pieri}
\begin{enumerate}[(a)]
\item
\label{itemlabel:lem:Pieri:a}
\[
\Ind_{\mathfrak{S}_{p} \times \mathfrak{S}_{k-p}}^{\mathfrak{S}_k} \left(\mathbb{S}^{\mu} \boxtimes 1_{\mathfrak{S}_{k-p}}\right) \cong_{\mathfrak{S}_k}  \bigoplus_{\lambda \in S(\mu,k)} \mathbb{S}^\lambda.
\]
\item
\label{itemlabel:lem:Pieri:b}
For each $\lambda \in S(\mu,k)$, $\length(\mu) \leq \length(\lambda) \leq\length(\mu)+1$. 
\end{enumerate}
\end{lemma}

\begin{proof}
Part~\eqref{itemlabel:lem:Pieri:a} is just Pieri's rule (see for instance \cite[Page 109]{Manivel}).
Part~\eqref{itemlabel:lem:Pieri:b} is obvious from definition of $S(\mu,k)$ (cf. Notation~\ref{not:Pieri}).
\end{proof}

The following lemma in conjunction with Lemma~\ref{lem:Pieri} will be used in the complexity analysis of 
Algorithm~\ref{alg:mult}.
\begin{lemma}
\label{lem:Pieri-quantitative}
Let $k \ge 1$, $0 \leq q \leq k$, and 
$\mu \vdash q$.
Then,
\[
\card(S(\mu,k)) \leq (k - \mu_1+1)(\mu_1 - \mu_2+1) \cdots (\mu_{m-1} - \mu_m+1)(\mu_m+1) \leq k^{\length(\mu)+1}.
\]
\end{lemma}

\begin{proof}
Obvious from 
Eqns. \eqref{eqn:Pieri1}, \eqref{eqn:Pieri2}, \eqref{eqn:Pieri1'} and \eqref{eqn:Pieri2'}.
\end{proof}

\begin{remark}
\label{rem:Solomon-to-Specht2}
Corollary~\ref{cor:Solomon-to-Specht} gives us an inductive method (using double induction on $k$ and $\card(T)$)  for obtaining the isotypic decomposition of the Solomon modules $\Psi^{(k)}_T$,
since the Solomon modules that appear on the right hand side of \eqref{eqn:cor:Solomon-to-Specht}
are either of a strictly smaller symmetric group since
$q < k$, 
or the Solomon module of $\mathfrak{S}_k$ but with respect to a smaller set of Coxeter elements
(since $\card(T - \{s_q\}) = \card(T) -1 < \card(T)$). Moreover, the isotypic decomposition of the
representation $\ind_{\mathfrak{S}_q \times \mathfrak{S}_{k-q}
}^{\mathfrak{S}_k} \left( \Psi_{T -\{s_q\}}^{(q)}\boxtimes 1_{\mathfrak{S}_{k-q}}\right)$ can be computed from
that of $ \Psi_{T -\{s_q\}}^{(q)}$ using Part~\eqref{itemlabel:lem:Pieri:a} of  Lemma~\ref{lem:Pieri} (Pieri's rule). \\

For the base cases notice that $\Psi^{(k)}_T$ is isomorphic to the trivial representation, $1_{\mathfrak{S}_k} \cong_{\mathfrak{S}_k} \mathbb{S}^{(k)}$ if $T =\emptyset$,
and for $k=1$, the $\Psi^{(1)}_T$ is again the trivial representation (the only $T$ that can appear is the empty set).\\

This algorithm for computing the isotypic decomposition of $\Psi_T^{(k)}$ using the inductive method sketched above is formally described  in Algorithm~\ref{alg:mult} in Section~\ref{sec:proofs2}, where we analyze the complexity of this
algorithm as well. We illustrate the method here by giving an example.

\begin{example}
\label{eg:Solomon-to-Specht}
Let $k=4$ and $T = \{s_2\}$. We will use Proposition~\ref{prop:Solomon-to-Specht} to obtain the isotypic decomposition of $\Psi^{(4)}_T$. 
In this example $q = 2$. So applying Proposition~\ref{prop:Solomon-to-Specht} we obtain
\begin{equation}
\label{eqn:eg:Solomon-to-Specht:1}
\ind_{\mathfrak{S}_2 \times \mathfrak{S}_2}^{\mathfrak{S}_4} \left(\Psi^{(2)}_{\emptyset} \boxtimes 1_{\mathfrak{S}_2}\right) \cong_{\mathfrak{S}_4} \Psi_{\emptyset}^{(4)} \oplus \Psi^{(4)}_T.
\end{equation}

Now (using \eqref{eqn:eg:1})
\begin{eqnarray*}
\Psi_{\emptyset}^{(2)} &\cong_{\mathfrak{S}_2}& \mathbb{S}^{(2)},\\
\Psi_{\emptyset}^{(4)} &\cong_{\mathfrak{S}_4}& \mathbb{S}^{(4)}.
\end{eqnarray*}

Using Part \eqref{itemlabel:lem:Pieri:a} of Lemma~\ref{lem:Pieri} we get
\begin{eqnarray}
\nonumber
\ind_{\mathfrak{S}_2 \times \mathfrak{S}_2}^{\mathfrak{S}_4} \left(\Psi^{(2)}_{\emptyset} \boxtimes 1_{\mathfrak{S}_2}\right) &\cong_{\mathfrak{S}_4}& \ind_{\mathfrak{S}_2 \times \mathfrak{S}_2}^{\mathfrak{S}_4} \left(\mathbb{S}^{(2)} \boxtimes 1_{\mathfrak{S}_2}\right) \\
\label{eqn:eg:Solomon-to-Specht:2}
&\cong_{\mathfrak{S}_4}&
\mathbb{S}^{(4)} \oplus \mathbb{S}^{(3,1)} \oplus \mathbb{S}^{(2,2)}.
\end{eqnarray}

In conjunction, \eqref{eqn:eg:Solomon-to-Specht:1} and \eqref{eqn:eg:Solomon-to-Specht:2} implies
\[
\mathbb{S}^{(4)} \oplus \mathbb{S}^{(3,1)} \oplus \mathbb{S}^{(2,2)} \cong_{\mathfrak{S}_4} \mathbb{S}^{(4)}
\oplus \Psi^{(4)}_T,
\]
whence
\[
\Psi^{(4)}_T \cong_{\mathfrak{S}_4} \mathbb{S}^{(3,1)} \oplus \mathbb{S}^{(2,2)}.
\]
Note that this example also illustrates  the fact that the Solomon modules need not be irreducible.
\end{example}
\end{remark}

Another important consequence of Proposition~\ref{prop:Solomon-to-Specht} that will be important for us
is a bound (in terms of the cardinality of $T$ alone) 
on the lengths of the partitions corresponding to the Specht modules that can appear in the isotypic
decomposition of $\Psi_T^{(k)}$. 
We deduce such a bound in the following proposition.

\begin{proposition}
\label{prop:Solomon}
Let $k \geq 1$, 
$T \subset \Coxeter(k)$. Then, for $\lambda \vdash k$, 
\begin{equation*}
\mult_{\mathbb{S}^{\lambda}}(\Psi^{(k)}_{T}) =  0 \mbox{ if }
\length(\lambda) > \card(T)+1 \mbox{ or  if } \length(^{t}\lambda) >  k - \card(T). 
\end{equation*}
\end{proposition}

\begin{remark}
\label{rem:Solomon}
Note that the bound in Proposition~\ref{prop:Solomon} above 
is the best possible (cf. Example~\ref{eg:Solomon-to-Specht}).
\end{remark}

\begin{proof}[Proof of Proposition~\ref{prop:Solomon}]
We first prove that 
\begin{equation}
\label{eqn:prop:Solomon:0}
\mult_{\mathbb{S}^{\lambda}}(\Psi^{(k)}_{T}) \neq  0 \Rightarrow 
\length(\lambda) \leq \card(T)+1.
\end{equation}
The proof is by a double induction on $k$, and on $t=\card(T)$.
Clearly, \eqref{eqn:prop:Solomon:0}  holds for $k=1$ and for all $T$. Also,
if  $T = \emptyset$ (i.e.  $t=0$)
\[
\Psi^{(k)}_{\emptyset} \cong \mathbb{S}^{(k)},
\]
and \eqref{eqn:prop:Solomon:0}  holds for all $k \geq 1$.

Now suppose that the proposition is true for all $k' < k$, and for given $k$ for all  $t' < t$ and suppose that $t >0$. 

Observe that for $\mu \vdash q$, using the fact that $q < k$ and the induction hypothesis we get that
\begin{equation}
\label{eqn:proof:prop:Solomon:4}
\mult_{\mathbb{S}^\mu} (\Psi^{(q)}_{Q'}) \neq 0  \Rightarrow
\length(\mu) \leq \card(Q') +1 = (\card(T)-1) +1 = \card(T). 
\end{equation}

Using Part~\eqref{itemlabel:lem:Pieri:a} of 
Lemma~\ref{lem:Pieri}
for any $\mu \vdash q$, 
\begin{equation}
\label{eqn:proof:prop:Solomon:4.4}
\ind_{\mathfrak{S}_{q} \times \mathfrak{S}_{k-q}}^{\mathfrak{S}_k}\left(\mathbb{S}^\mu \boxtimes 1_{\mathfrak{S}_{k-q}}\right) \cong 
\bigoplus_{\lambda \in S(\mu,k)} \mathbb{S}^\lambda \mbox{ (cf. Notation~\ref{not:Pieri})}.
\end{equation}
Also by  Part~\eqref{itemlabel:lem:Pieri:b} of 
Lemma~\ref{lem:Pieri}  
\begin{equation}
\label{eqn:proof:prop:Solomon:4.5}
\lambda \in S(\mu,k) \Rightarrow \length(\lambda) \leq  \length(\mu) +1.
\end{equation}

It follows from \eqref{eqn:proof:prop:Solomon:4},  
 \eqref{eqn:proof:prop:Solomon:4.4}
and  \eqref{eqn:proof:prop:Solomon:4.5}, that  
for $\lambda \vdash k$, 
\begin{equation}
\label{eqn:proof:prop:Solomon:5}
\mult_{\mathbb{S}^\lambda} \left(\ind_{\mathfrak{S}_q \times \mathfrak{S}_{k-q}}^{\mathfrak{S}_k}  
\left( \Psi^{(q)}_{Q'} \boxtimes 1_{\mathfrak{S}_{k-q}}\right)\right)
 \neq 0 \Rightarrow \length(\lambda) \leq \card(T) +1.
\end{equation}

The claim in \eqref{eqn:prop:Solomon:0} now follows from \eqref{eqn:proof:prop:Solomon:5}, 
Proposition~\ref{prop:Solomon-to-Specht} and Schur's Lemma (Lemma~\ref{lem:Schur} in the Appendix). 
This finishes the inductive proof of \eqref{eqn:prop:Solomon:0}. \\

We now prove 
\begin{equation}
\label{eqn:prop:Solomon:4}
\mult_{\mathbb{S}^{\lambda}}(\Psi^{(k)}_{T}) \neq  0 \Rightarrow 
\length(^{t}\lambda) \leq k-\card(T).
\end{equation}

First observe that using \eqref{eqn:transpose}
\begin{eqnarray*}
\mathbb{S}^{^{t}\lambda} &\cong& \mathbb{S}^\lambda \otimes \mathbb{S}^{1^k}, \\
\Psi^{(k)}_{\Coxeter(k)-T} &\cong & \Psi^{(k)}_{T}  \otimes \mathbb{S}^{1^k}.
\end{eqnarray*}

It follows that
\begin{eqnarray*}
\mult_{\mathbb{S}^{\lambda}}(\Psi^{(k)}_{T}) \neq  0 & \Leftrightarrow&
\mult_{\mathbb{S}^{^t\lambda}}(\Psi^{(k)}_{\Coxeter(k)-T}) \neq 0 \\
&\Rightarrow&
\length(^{t}\lambda) \leq  \card(\Coxeter(k)-T)+ 1 \mbox{ using \eqref{eqn:prop:Solomon:0}} \\
&\Rightarrow&
\length(^{t}\lambda) \leq  k - \card(T). \\
\end{eqnarray*}
\end{proof}

We now introduce a geometric construction (that of a mirrored space) which will play 
an important role later.

\subsection{Mirrored spaces and Weyl chambers}
\label{subsec:mirror}
We first recall a definition  from \cite{Davis-book}.

\begin{definition}[Mirrored space]
\label{def:mirrored-space}
Given a Coxeter pair  $(\Ww,\Ss)$  (i.e. $\Ww$ is a Coxeter group and $\Ss$ a set of reflections generating $\Ww$)
a space $Z$ with a family of closed 
subspaces $(Z_s)_{s\in \Ss}$ 
is called a \emph{mirror structure} on $Z$ \cite[Chapter 5.1]{Davis-book},
and $Z$ along with the collection $(Z_s)_{s \in \Ss}$ is called a \emph{mirrored space}  over $\Ss$.
\end{definition}

Given a mirrored space  $Z, (Z_s)_{s \in \Ss}$ over $\Ss$,
there is a classical construction (called `The Basic Construction'  in \cite[Chapter 5]{Davis-book})
of a space $\mathcal{U}(\Ww,Z)$ with a $\Ww$-action 
which we define as follows.

\begin{definition}[The Basic Construction \cite{Koszul, Tits, Vinberg, Davis}]
\label{def:U}
We define
\begin{equation}
\label{eqn:U}
\mathcal{U}(\Ww,Z) = \Ww \times Z/\sim
\end{equation}
where the topology on $\Ww \times Z$ is the product topology, with $\Ww$ given the discrete topology,
and the equivalence relation $\sim$ is defined by  
\[
(w_1,\x) \sim (w_2,\y) \Leftrightarrow \x= \y \mbox { and } w_1^{-1}w_2 \in \Ww^{\Ss(\x)},
\]    
with 
\[
\Ss(\x) = \{s \in \Ss \; \mid \; \x \in Z_s\},
\]  and $\Ww^{\Ss(\x)}$ the subgroup of $\Ww$ generated by $\Ss(\x)$. 

The group $\Ww$ acts on $\mathcal{U}(\Ww,Z)$ by $w_1 \cdot [(w_2,\z)] =  [(w_1w_2,\z)]$ (where
$[(w,\z)]$ denotes the equivalence class of $(w,\z) \in \Ww \times Z$ under the relation $\sim$).
\end{definition}

For a  mirrored space $Z$ over $\Ss$, the cohomology groups, $\HH^*(\mathcal{U}(\Ww,Z))$,  gets a structure of
a $\Ww$-module from the $\Ww$-action on $\mathcal{U}(\Ww,Z)$, and $\HH^*(\mathcal{U}(\Ww,Z))$.
The cohomology groups of $\mathcal{U}(\Ww,Z)$ are studied in \cite{Davis-book} in the case
where $Z$ is a finite CW-complex, however in this paper we are concerned with mirrored spaces which are semi-algebraic.

\subsubsection{Semi-algebraic mirrored spaces}

\begin{definition}
\label{def:mirrored-space-sa}
We will call a mirrored space $Z, (Z_s)_{s \in \Ss}$ over $\Ss$, to be a semi-algebraic mirrored space
over $\Ss$, if $Z$ and each $Z_s, s \in \Ss$ are semi-algebraic sets.
\end{definition}

\begin{remark}
\label{rem:semi-algebraic}
First observe that for a finite group $\Ww$, and a semi-algebraic set $Z$, $\Ww \times Z$ is again a semi-algebraic set. Moreover, if $Z$ is closed and bounded, so is $\Ww \times Z$, and the quotient
$\mathcal{U}(\Ww,Z)$ is also semi-algebraic, since  the quotient of a semi-algebraic set by a proper 
semi-algebraic equivalence relation is semi-algebraic (\cite[page 166]{Dries}). 

Note also that  every closed and bounded semi-algebraic set is semi-algebraically homeomorphic
to the  geometric realization over $\R$ of a finite simplicial complex  (see for example 
\cite[Chapter 5]{BPRbook2}).
More generally, if $Z, (Z_s)_{s \in \Ss}$ is a semi-algebraic mirrored space, with $Z, Z_s, s\in \Ss$ closed and bounded,
then there exists a finite simplicial complex $K$ and subcomplexes $K_s \subset K, s\in \Ss$, and a semi-algebraic
homeomorphism $h:Z \rightarrow |K|$, which restricts to  homeomorphisms $Z_s \rightarrow |K_s|, s \in \Ss$.  

Moreover, for any subset $T \subset \Ss$,  the cohomology groups of $Z$ (resp. pairs $(Z,\bigcup_{s \in T} Z_s)$) are isomorphic to the simplicial cohomology groups of the 
simplicial complex $K$ (resp. pairs $(K,\bigcup_{s \in T} K_s)$) (see \cite[Chapter 6]{BPRbook2}).
\end{remark}

In view of Remark~\ref{rem:semi-algebraic} the following theorem stated in \cite{Davis-book} 
for finite CW-complexes remain true for semi-algebraic mirrored space $(Z,(Z_s)_{s\in \Ss})$ with $Z,Z_s, s \in \Ss$  closed and bounded.  We state the theorem in the special case where
$(\Ww,\Ss) = (\mathfrak{S}_k,\Coxeter(k))$ which is the only case of interest to us in this paper.

\begin{theorem}\cite[Theorem 15.4.3]{Davis-book}
\label{thm:Davis0}
Let $(\Ww,\Ss) = (\mathfrak{S}_k,\Coxeter(k))$, and $Z,Z_s, s \in \Ss$ a semi-algebraic mirrored space over $\Ss$,  and  $Z,Z_s, s\in \Ss$ closed and bounded.
Then,
\[
\HH_*(\mathcal{U}(\Ww,Z)) \cong_{\mathfrak{S}_k}  \bigoplus_{T \subset \Ss} \HH_*(Z, Z^T) \otimes  \Psi^{(k)}_{T},
\]
where for each $T \subset \Ss$,
\[ 
Z^T = \bigcup_{s \in T} Z_s.
\]
\end{theorem}

\subsubsection{Weyl chambers}
The semi-algebraic mirrored spaces that we will be interested in are of a special type. In order to introduce them we first need a few more definitions.

\begin{notation}
\label{not:Weyl0}
We denote by $\Weyl^{(k)} \subset \R^k$  the cone defined by $X_1 \leq X_2 \leq \cdots \leq X_k$, and 
by $\Weyl^{(k),o}$ the interior of $\Weyl^{(k)}$ (i.e. the cone defined by $X_1 < X_2 < \cdots < X_k$).
\end{notation}

\begin{notation}
\label{not:composition}
For $k \in \Z_{\geq 0}$, we denote by $\Comp(k)$ the set of integer tuples 
\[
\lambda= (\lambda_1,\ldots,\lambda_\ell), \lambda_i > 0, |\lambda| := \sum_{i=1}^{\ell} \lambda_i = k.
\] 
\end{notation}

\begin{definition}
\label{def:composition-order}
For $k \in \Z_{\geq 0}$, and $\lambda = (\lambda_1,\ldots,\lambda_\ell) \in \Comp(k)$,
we denote by $\Weyl_{\lambda}$ the subset of $\Weyl^{(k)}$ defined by,
\[
X_1 = \cdots = X_{\lambda_1} \leq X_{\lambda_1+1} = \cdots = X_{\lambda_1+\lambda_2} \leq \cdots \leq X_{\lambda_1+\cdots+\lambda_{\ell-1}+1} = \cdots = X_k,
\]
and 
denote by $\Weyl_{\lambda}^o$ the subset of $\Weyl^{(k)}$ defined by
\[
X_1 = \cdots = X_{\lambda_1} < X_{\lambda_1+1} = \cdots = X_{\lambda_1+\lambda_2} < \cdots < X_{\lambda_1+\cdots+\lambda_{\ell-1}+1} = \cdots = X_k.
\]

We denote by $L_\lambda$ the subspace defined by 
\[
X_1 = \cdots = X_{\lambda_1},  X_{\lambda_1+1} = \cdots = X_{\lambda_1+\lambda_2}, \cdots, X_{\lambda_1+\cdots+\lambda_{\ell-1}+1} = \cdots = X_k,
\]
which is the linear hull of $\Weyl_\lambda$.
\end{definition}

\begin{notation}
\label{not:T-to-Comp}
For $s = (i,i+1) \in \Coxeter(k)$, we denote by $\Weyl^{(k)}_s$ the face of $\Weyl^{(k)}$ defined by 
$X_{i} = X_{i+1}$. More generally, for $T \subset \Coxeter(k)$, we denote:
\begin{eqnarray*}
\Weyl^{(k)}_T &=& \bigcap_{s \in T} \Weyl^{(k)}_s, \\ 
\Weyl^{(k,T)} &=& \bigcup_{s \in T} \Weyl^{(k)}_s.
\end{eqnarray*}

We also define $\lambda(T) \in \Comp(k)$ implicitly by the equation
\begin{equation}
\label{eqn:implicit-lambda}
\Weyl_{\lambda(T)} = \Weyl^{(k)}_T.
\end{equation}
\end{notation}

\begin{notation}
Finally, for any semi-algebraic set $Z \subset \Weyl^{(k)}$, $T \subset \Coxeter(k)$, we set
\label{not:Z-k-T}
\begin{eqnarray*}
Z^T &=& Z \cap \Weyl^{(k,T)}, \\
Z_{T} &=& Z \cap \Weyl^{(k)}_T .
\end{eqnarray*}

For any semi-algebraic subset $S \subset \R^k$, we will denote
\begin{eqnarray*}
S_{k} &=& S \cap \Weyl^{(k)},
\end{eqnarray*}
and we will  for convenience of notation write $S_{k,T}$ (respectively, $S_{k}^T$), in place of 
$(S_k)_T$ (respectively, $(S_k)^T$).
\end{notation}

Now suppose that $S$ is a closed and bounded \emph{symmetric}  semi-algebraic subset of $\R^k$, then (using Notation~\ref{not:Z-k-T}) $S_k \subset \Weyl^{(k)}$. Then,
$S_k$ along with the  tuple of  closed semi-algebraic subsets $(S_{k,s} = S_k \cap \Weyl^{(k)}_s)_{s \in \Coxeter(k)}$ 
 (cf. Notation~\ref{not:T-to-Comp})
is a semi-algebraic mirrored  space over $\Coxeter(k)$.

It follows immediately from Definition~\ref{def:U} that
\begin{proposition}
The semi-algebraic set $\mathcal{U}(\mathfrak{S}_k,S_k)$ is semi-algebraically homeomorphic to 
$S$.
\end{proposition}
\label{prop:U}
\begin{proof}
It is a simple exercise to verify that the map
\[
[(w,\x)]  \mapsto w\cdot \x
\]
is a semi-algebraic homeomorphism $\mathcal{U}(\mathfrak{S}_k,S_k) \rightarrow S$.
\end{proof}

Proposition~\ref{prop:U} in conjunction with Theorem~\ref{thm:Davis0} yields the following result
that we will use later in the paper. This is the only result from this subsection that we will need in the rest
of the paper.
  
\begin{theorem}
\label{thm:Davis}
Let $S$ be a closed and bounded symmetric semi-algebraic subset of $\R^k$. Then,
\[
\HH_*(S) \cong_{\mathfrak{S}_k}  \bigoplus_{T \subset \Coxeter(k)} \HH_*(S_{k}, S_{k}^T) \otimes  \Psi^{(k)}_{T}.
\]
\end{theorem}
\qed

\section{Outline of our method and two important examples}
\label{sec:outline+example}
\subsection{Outline of the proofs of Theorems~\ref{thm:main1} and \ref{thm:vandermonde}}
\label{subsec:outline}
We first observe that symmetric semi-algebraic subsets $S \subset \R^k$, defined in terms of equalities and inequalities
of symmetric polynomials of degree at most $d$, admits a map to $\R^d$  (by the first $d$ Newton power sum polynomials restricted to $S$), whose fibers are Vandermonde varieties. Moreover the action of $\mathfrak{S}_k$ keeps the fibers stable, and thus the action of $\mathfrak{S}_k$ on $S$ also induces an action on the Leray spectral sequence of this
map. As a result in order to prove the vanishing of certain irreducible $\mathfrak{S}_k$-modules, it suffices to prove
this vanishing for Vandermonde varieties. 
The Vandermonde  varieties are well studied and have nice topological and geometric properties. For us the most important property implicit in the work of Arnold, Giventhal and Kostov is that the intersection $Z$ of a Vandermonde variety $V$
with a Weyl chamber $\Weyl^{(k)}$ in $\R^k$ is either a point or a regular cell of the dimension of the variety. 
Moreover, the structure of the boundary of $Z$ (in case $Z$ is a regular cell) is well understood in terms of the 
combinatorics of the faces of $\Weyl^{(k)}$ with which  $Z$ has a non-empty intersection.  \\

Applying 
Theorem~\ref{thm:Davis}
to our situation we obtain that 
the cohomology groups of $V$ 
are isomorphic to
direct sums of tensor products of 
the Solomon modules
$\Psi_T^{(k)}$, 
indexed by subsets $T \subset \Coxeter(k)$,
and the cohomology groups
of the pairs $(Z,Z^T), T \subset \Coxeter(k)$, where as before 
\[ 
Z^T = \bigcup_{s \in T} Z_s.
\]

Recall now that by Proposition \ref{prop:Solomon}
only those Specht modules can appear in $\Psi_T^{(k)}$ whose number of rows is bounded by $\card(T)+1$ (and a similar restriction in terms of the number of columns). \\

One final ingredient is the observation that in the case when $Z$ has the expected dimension $k-d$, then the intersection
of $Z$ with the various faces of $\Weyl^{(k)}$, induces a structure of a regular cell complex, and the boundary of $Z$
is then semi-algebraically homeomorphic to the $(k-d-1)$-dimensional sphere, and the intersection of $Z$ with the various
$\Weyl^{(k)}_s, s\in \Coxeter(k)$, gives an acyclic covering  of the boundary of $Z$ having 
cardinality at most $k-1$.
This implies via an argument using the nerve lemma and Alexander duality that the cohomology groups $\HH^i(Z,Z^T)$ must vanish if $i$ is large compared to the cardinality of $T$ 
and also a dual statement (cf. Proposition~\ref{prop:cell-complex}).\\

Putting these together we obtain our theorem on the vanishing of certain
multiplicities for Vandermonde varieties (cf. Theorem \ref{thm:vandermonde}). Theorem \ref{thm:main1} is then a consequence of Theorem \ref{thm:vandermonde} and an argument involving (an equivariant version of) the 
Leray spectral sequence. \\

 Finally, the restriction result that we prove also allows us, via the Solomon-Davis formula alluded to above, and 
 some additional ingredients (see the outline in Section~\ref{subsec:alg:outline}) including certain
 standard algorithms from semi-algebraic geometry, to effectively compute the Betti numbers $b_i(S), 0 \leq i \leq \ell$,
 for any fixed $\ell$ with complexity which is polynomial in the number of variables and the number of polynomials.
 Here we are assuming that the degrees of the input polynomials are also bounded by a constant. \\

We will now proceed to describe two important
examples, whose analysis already exposes the central ideas behind the proofs of the main theorems. \\

We first introduce some more notation.
\begin{notation}
\label{not:Weyl}
For every $m \geq 0$, and $\w = (w_1,\ldots,w_k) \in \R_{> 0}^k$ we denote
$$
 \begin{array}{rlccc}
 p_{\w,m}^{(k)} &:\,& \R^k & \,\,\longrightarrow\,\, & \R \\[0.5ex]
     & & \x = (x_1,\ldots,x_k) & \longmapsto & \sum_{j=1}^{k} w_j x_j^m, %
\end{array}%
 $$
and for every $ d\geq 0 $, and $\w \in \R_{> 0}^k$ we denote by 
$\Phi_{\w,d}^{(k)}$ 
the continuous map defined by

$$
 \begin{array}{rlccc}%
 \Phi_{\w,d}^{(k)} &:\,& \R^k & \,\,\longrightarrow\,\, & \R^{d'} \\[0.5ex]
     & & \x = (x_1,\ldots,x_k) & \longmapsto & (p_{\w,1}^{(k)}(\x),\ldots,p_{\w,d'}^{(k)}(\x)), %
\end{array}%
 $$
 where $d' = \min(k,d)$. 

Finally, we denote by $$\Psi_{\w,d}^{(k)}:\Weyl^{(k)} \longrightarrow \R^{d'}$$ the restriction of $\Phi_{\w,d}^{(k)}$ to
$\Weyl^{(k)}$. \\

If $\w = 1^k :=  (1,\ldots,1)$, then we will denote by $p_m^{(k)}$ the polynomial 
$p_{\w,m}^{(k)}$  (the $m$-th Newton sum polynomial), and by $\Phi^{(k)}_d$ (respectively, $\Psi^{(k)}_d$) the map 
$\Phi^{(k)}_{\w,d}$ (respectively, $\Psi^{(k)}_{\w,d}$). \\

For every $\w \in \R_{\geq 0}^k$,  $d,k \geq 0, d \leq k$, and $\y \in \R^d$, we will denote by 
$$V_{\w,d,\y}^{(k)} := (\Phi^{(k)}_{\w,d})^{-1}(\y),\text{ and } Z_{\w,d,\y}^{(k)} := (\Psi^{(k)}_{\w,d})^{-1}(\y).$$ \\

If $\w = 1^k :=  (1,\ldots,1)$, then we just denote $V_{\w,d,\y}^{(k)}$ by $V_{d,\y}^{(k)}$, and 
$Z_{\w,d,\y}^{(k)}$ by $Z_{d,\y}^{(k)}$.
\end{notation}

We are now ready to discuss the promised examples.

\subsection{Examples}
\label{sec:eg}
\subsubsection{Example with $d=2$ and $k \geq 3$.}
\label{eg:V-k-2}
We first consider the case $d=2$ for $k \geq 3$, which has already being alluded to in 
Remark~\ref{rem:d=1}. 
Recall that in this case,
the Vandermonde variety $V^{(k)}_{2,\y}$  is defined by the equation 
\[
\sum_{i=1}^k X_i = y_1, \sum_{i=1}^{k} X_i^2 = y_2,
\] 
and is empty, a point,  or a semi-algebraically homeomorphic to a sphere of dimension $k-2$
(depending on 
whether
$y_1^2 - k y_2$ is $> 0, = 0$, or $< 0$, respectively).

The first two cases are trivial.
In the last case,
$Z_{2,\y}^{(k)} = V_{2,\y}^{(k)} \cap \Weyl^{(k)}$ is a closed disk of dimension $k-2$,
and has a non-empty intersection with all the faces of the Weyl chamber $\Weyl^{(k)}$. 
(See Figure~\ref{fig:image0} for the case $k=4$, where  $Z_{2,\y}^{(4)}$ is one of the triangles
on the two-dimensional sphere equal to $V_{2,\y}^{4)}$. Notice that in this case $Z_{2,\y}^{(4)}$ meets all the three 
faces of the Weyl chamber $\Weyl^{(4)}$.) \\

It follows that in this case
\begin{eqnarray}
\label{eqn:d=2:1}
\HH^i(Z_{2,\y}^{(k)},Z_{2,\y}^{(k,T)}) &\cong& \Q \mbox{ if $(i,T) = (0,\emptyset)$  or $(k-2,\Coxeter(k))$},\\
\nonumber
&=& 0 \mbox{ otherwise}.
\end{eqnarray}

The $\mathfrak{S}_k$-module structure of $V_{2,\y}^{(k)}, y_1^2 - k y_2 < 0, k\geq 3$
stated in \eqref{eqn:d=2} in Remark~\ref{rem:d=1} now follows from 
\eqref{eqn:d=2:1}, 
\eqref{eqn:eg:1}, \eqref{eqn:eg:2},and 
Theorem \ref{thm:Davis}.

\subsubsection{Example of  $V_{3,\y}^{(4)} \subset \R^4$}
\label{eg:V-4-3}
We now study the cohomology of  the symmetric Vandermonde varieties (curves)  $V_{3,\y}^{(4)} \subset \R^4$, as
$\mathfrak{S}_4$-modules, 
for various $\y = (y_1,y_2,y_3) \in \R^3$. \\

In this case the Weyl chamber $\Weyl^{(4)} \subset \R^4$ has three faces corresponding to the
compositions $(2,1,1)$, $(1,2,1)$ and $(1,1,2)$. In terms of the Coxeter elements $s_1= (1,2)$, $s_2=(2,3)$, and 
$s_3 = (3,4)$, these faces correspond to $s_1$, $s_2$, and $s_3$ respectively. 
In other words, using the notation introduced in  \eqref{eqn:implicit-lambda},
\begin{eqnarray*}
\lambda(\{s_1\}) &=& (2,1,1), \\
\lambda(\{s_2\}) &=& (1,2,1), \\
\lambda(\{s_3\}) &=& (1,1,2).
\end{eqnarray*} 
Also, note that
\begin{eqnarray*}
\lambda(\{s_1,s_2\}) &=& (3,1), \\
\lambda(\{s_1,s_3\}) &=& (2,2), \\
\lambda(\{s_2,s_3\}) &=& (1,3).
\end{eqnarray*}

We first need a preliminary calculation.  
Observe that
\begin{eqnarray*}
\Ind_{\mathfrak{S}_3}^{\mathfrak{S}_4} \Psi_{\emptyset}^{(3)} &\cong_{\mathfrak{S}_4}& \mathbb{S}^{(4)} \oplus \mathbb{S}^{(3,1)}\\
&\cong_{\mathfrak{S}_4}& \Psi_{\emptyset}^{(4)} \oplus \Psi_{\{s_1\}}^{(4)} \mbox{ (using Proposition \ref{prop:Solomon})}.
\end{eqnarray*}

From this we deduce that
\begin{eqnarray}
\label{eqn:eg:3}
\Psi_{\{s_1\}}^{(4)} &\cong_{\mathfrak{S}_4}&   \mathbb{S}^{(3,1)},
\end{eqnarray}
and using \eqref{eqn:transpose} that,
\begin{eqnarray}
\label{eqn:eg:4}
\Psi_{\Coxeter(4) - \{s_1\}}^{(4)} &\cong_{\mathfrak{S}_4}&  \mathbb{S}^{(2,1,1)}.
\end{eqnarray}

Returning to the study of topology of the curve $V_{3,\y}^{(4)}$,
there are five different cases possible depending on the configuration of the curve $V_{3,\y}^{(4)}$ inside $\Weyl^{(4)}$.
Recall (cf. Notation \ref{not:Weyl})  that we denote $Z_{3,\y}^{(k)} = V_{3,\y}^{(4)} \cap \Weyl^{(4)}$.

\begin{enumerate}[{Case} 1.]
\item
\label{itemlabel:eg:key:outer:1}
The Vandermonde variety $V^{(4)}_{2,(y_1,y_2)}$ is empty: in this case
$Z_{3,\y}^{(4)} = \emptyset$,  and $\HH^0(V_{3,\y}^{(4)}) = \HH^0(V_{3,\y}^{(4)}) = 0$.

\item
\label{itemlabel:eg:key:outer:2}
The Vandermonde variety $V^{(4)}_{2,(y_1,y_2)}$ is singular and  $V^{(4)}_{3,\y}$ is non-empty:
in this case, $Z_{3,\y}^{(4)}$ is a point which must necessarily belong to the face labeled by $(4)$ of $\Weyl^{(4)}$.
Thus, $Z_{3,\y}^{(4)}$ belongs to all non-zero  faces
of $\Weyl^{(4)}$, and
$y_2$ is a minimum value of $p_2^{(4)}$ on $V^{(4)}_{1,(y_1)}$. 
(This preceding fact follows from Theorem~\ref{thm:arnold} stated later.)

In this case  (using Notation \ref{not:Z-k-T}) 
\begin{eqnarray*}
\HH^0(Z_{3,\y}^{(4)},Z_{3,\y}^{(4,T)}) &\cong& \Q,  \mbox{ if $T = \emptyset$}, \\
\HH^0(Z_{3,\y}^{(4)},Z_{3,\y}^{(4,T)}) &=& 0,   \mbox{ otherwise}.
\end{eqnarray*}
This implies that
\begin{eqnarray*}
\HH^0(V_{3,\y}^{(4)}) &\cong_{\mathfrak{S}_4}& \Psi_{\emptyset}^{(4)}  \\
&\cong_{\mathfrak{S}_4}& 1_{\mathfrak{S}_4} \mbox{ (using \eqref{eqn:eg:1})}.
\end{eqnarray*}

It follows that $b_0(V_{3,\y}^{(4)}) = 1$ (using the Eqn. \eqref{eqn:hook}).
Clearly, $\HH^1(V_{3,\y}^{(4)}) = 0$ in this case.

\item
\label{itemlabel:eg:key:outer:3}
The Vandermonde variety $V^{(4)}_{2,(y_1,y_2)}$ is non-empty and non-singular.
Lets fix $y_1,y_2$ such that  $V^{(4)}_{2,(y_1,y_2)}$ is non-empty and non-singular.
In this case, 
$V^{(4)}_{2,(y_1,y_2)}$ is a sphere which is depicted in Figure \ref{fig:image0}. 

%%sb hides
\hide{
\begin{figure}
\includegraphics[scale=0.4]{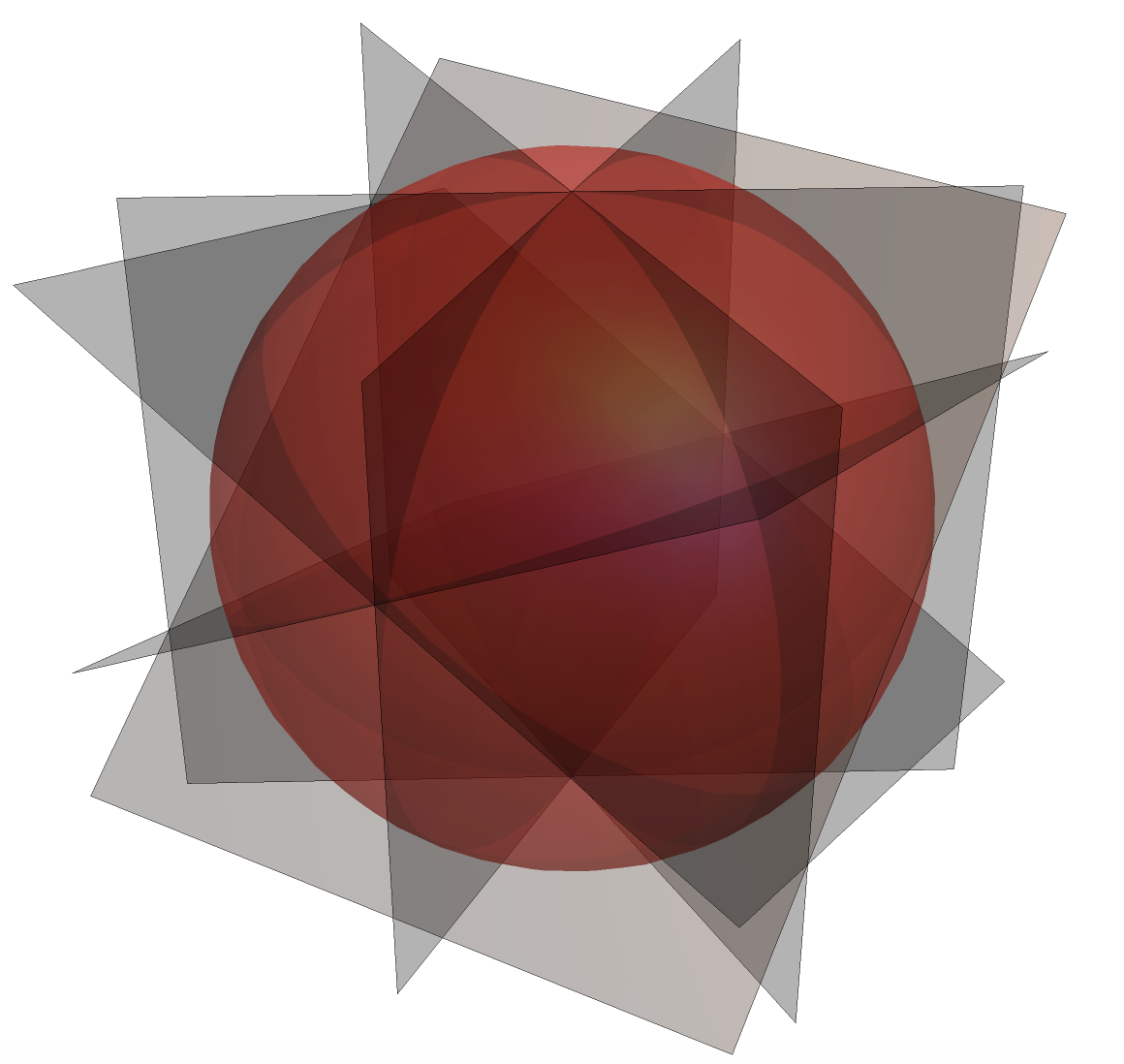}
\caption{Example of the non-singular Vandermonde variety $V_{2,(y_1,y_2)}^{(4)}$.}
\label{fig:image0}
\end{figure}
}
%%sb end of hide
The hyperplanes (shown in grey)  in Figure~\ref{fig:image0} cutting out
the $4! = 24$ triangles on the sphere are the walls of the various Weyl chambers. 
Notice that there are $14$ vertices in the arrangement of great circles on the sphere, 
$8$ of them incident on $3$ circles and the remaining $6$ incident on $2$ circles.
There are several sub-cases to consider.  The (non-empty) sub-cases are depicted in
Figures~\ref{fig:image1}, \ref{fig:image2}, \ref{fig:image3} and \ref{fig:image4} 
($V^{(4)}_{3,\y}$ is shown in blue).

\begin{figure}[htb] 
\captionsetup[subfigure]{labelfont=rm}
\begin{subfigure}{.48\textwidth}
  \centering
  \includegraphics[scale=0.22]{image0.png}
\caption{Example of the non-singular  Vandermonde variety $V_{2,(y_1,y_2)}^{(4)}$.}
\label{fig:image0}
  \end{subfigure}
\begin{subfigure}{.48\textwidth}
  \centering
  \includegraphics[scale=0.22]{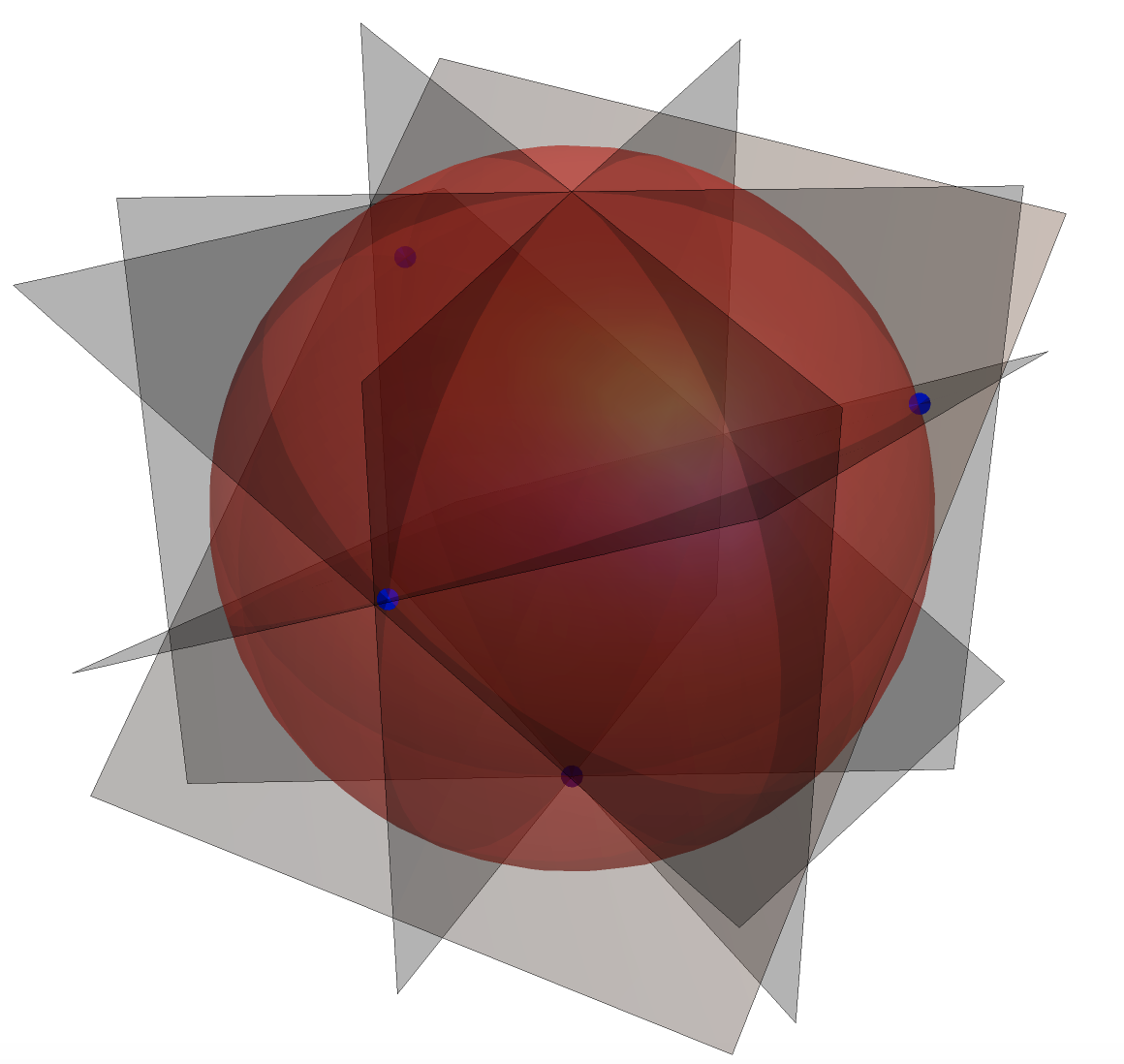}
\caption{Vandermonde variety $V^{(4)}_{3,\y}$ in Case \ref{itemlabel:eg:key:b}.}
\label{fig:image1}
\end{subfigure}
\begin{subfigure}{.48\textwidth}
   \centering
  \includegraphics[scale=0.22]{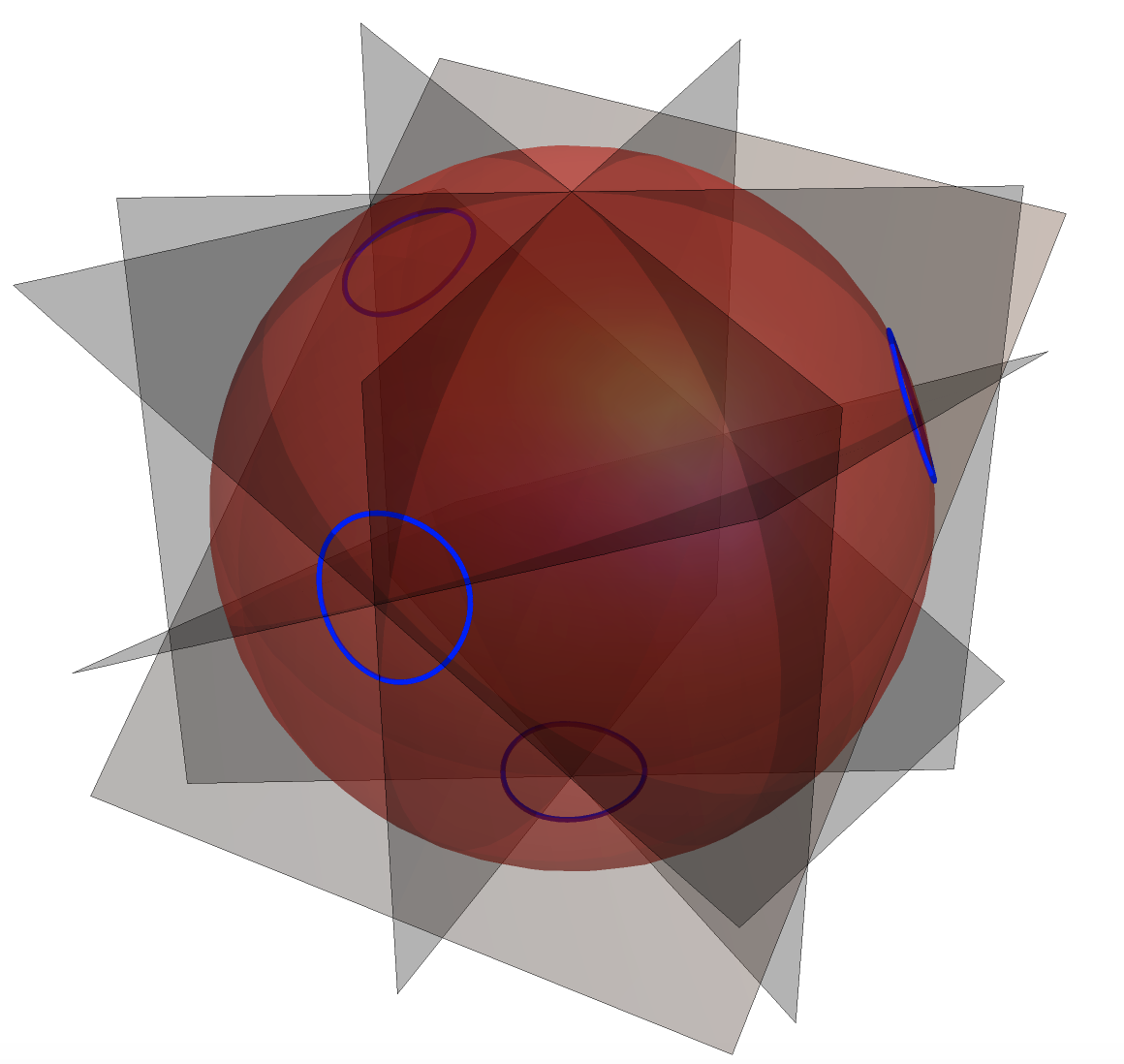}
\caption{Vandermonde variety $V^{(4)}_{3,\y}$ in Case \ref{itemlabel:eg:key:c}.}
\label{fig:image2}
  \end{subfigure}
\begin{subfigure}{.48\textwidth}
   \centering
  \includegraphics[scale=0.22]{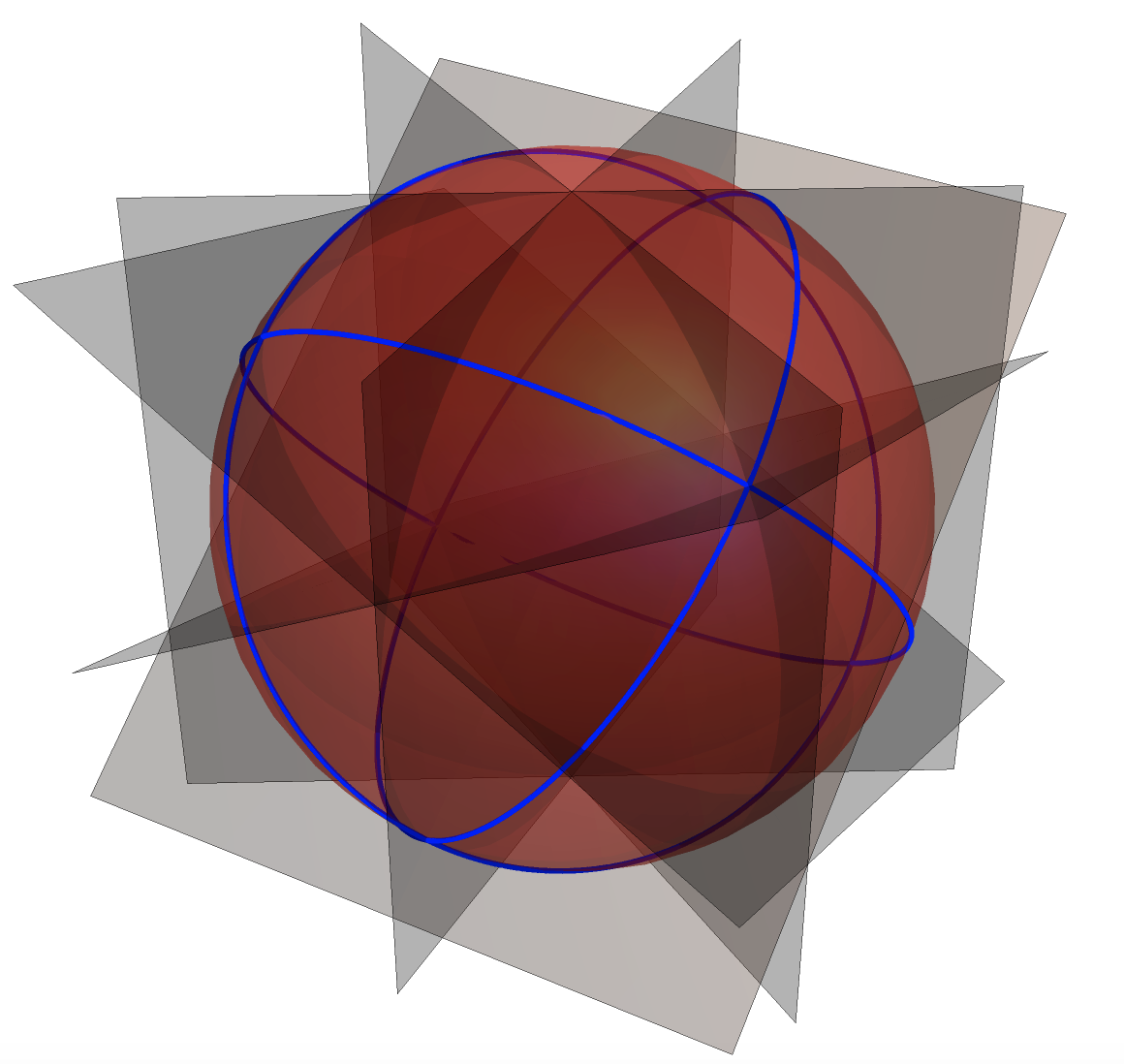}
\caption{Vandermonde variety $V^{(4)}_{3,\y}$ in Case \ref{itemlabel:eg:key:d}.}
\label{fig:image3}
 \end{subfigure}
%%  \hfill
\begin{subfigure}{.48\textwidth}
   \centering
  \includegraphics[scale=0.22]{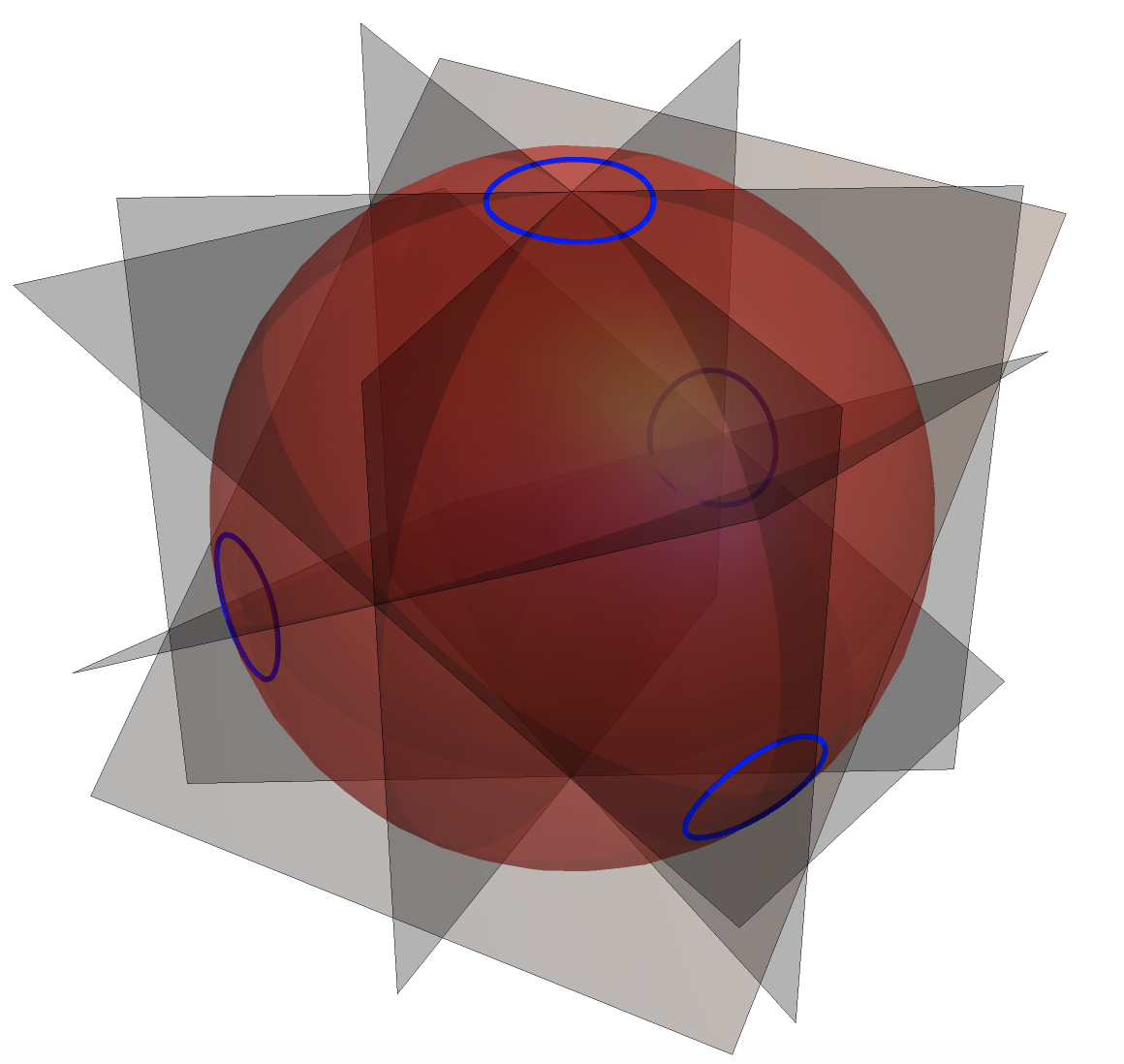}
\caption{Vandermonde variety $V^{(4)}_{3,\y}$ in Case \ref{itemlabel:eg:key:e}.}
\label{fig:image4}
\end{subfigure}
 \begin{subfigure}{.48\textwidth}
   \centering
  \includegraphics[scale=0.22]{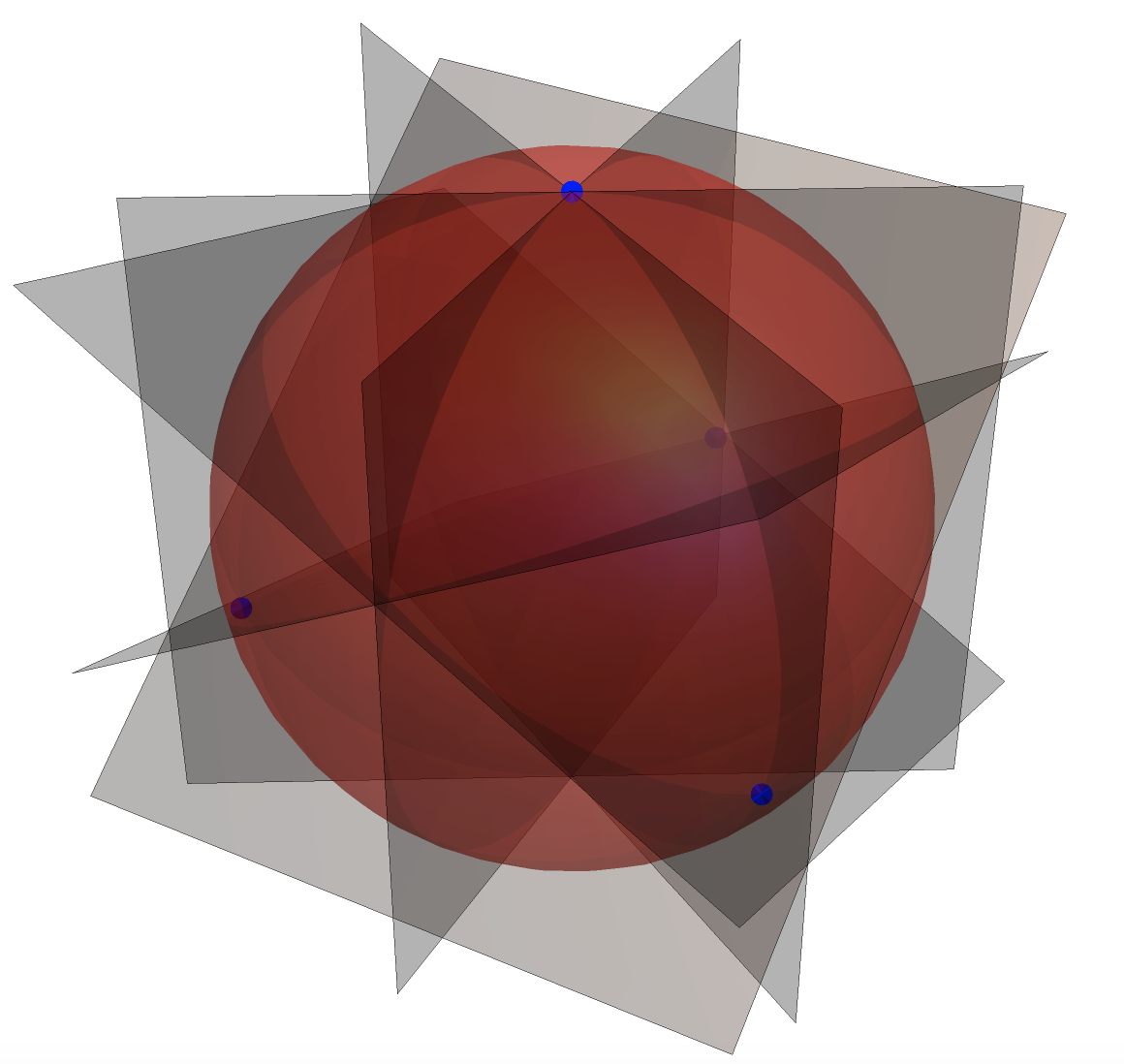}
\caption{Vandermonde variety $V^{(4)}_{3,\y}$ in Case \ref{itemlabel:eg:key:f}.}
\label{fig:image5}
\end{subfigure}
\caption{Different examples of Vandemonde varieties}
\end{figure}%%sb hides
\hide{
\begin{figure}
\includegraphics[scale=0.2]{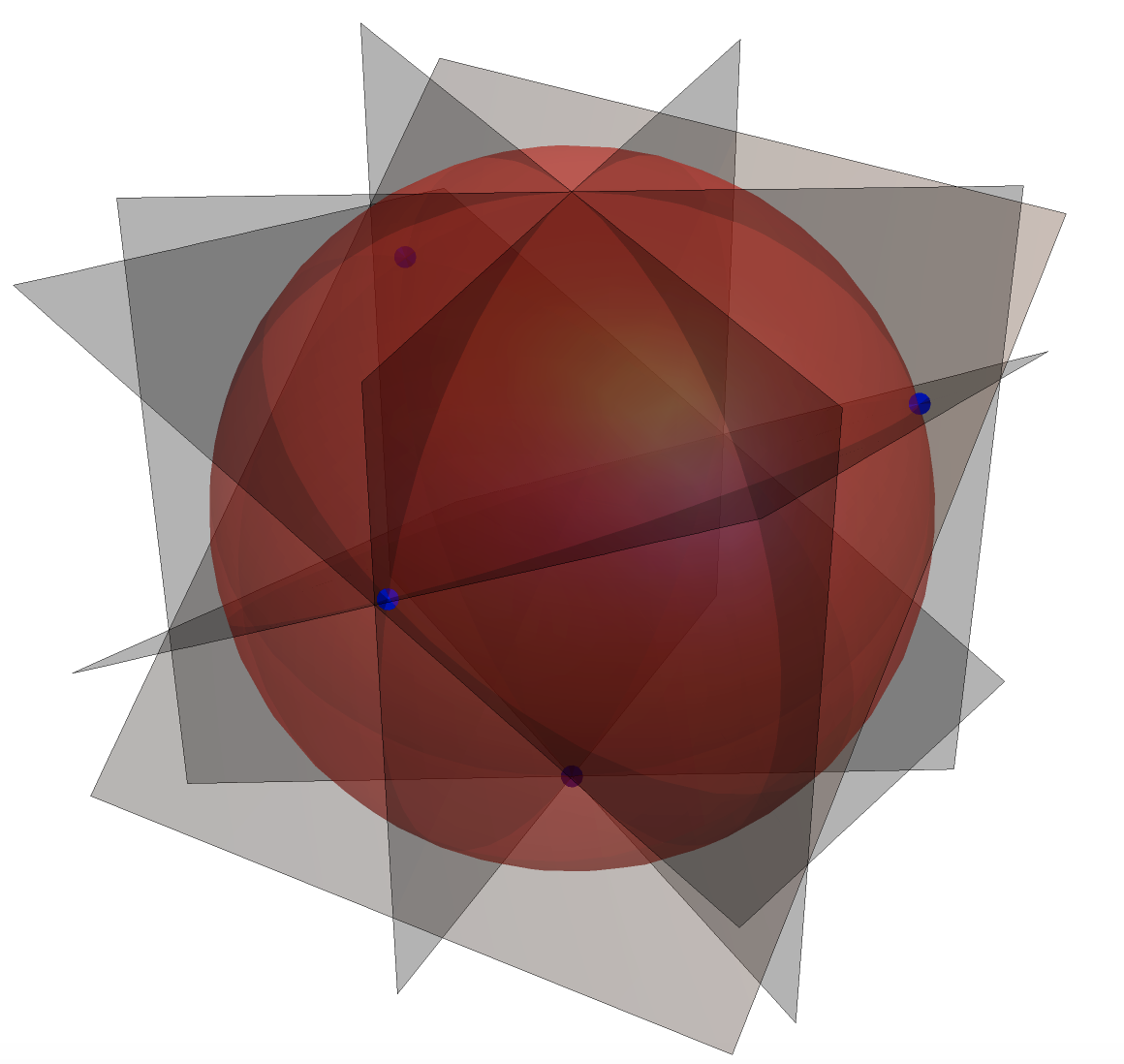}
\caption{Vandermonde variety $V^{(4)}_{3,\y}$ in Case \ref{itemlabel:eg:key:b}.}
\label{fig:image1}
\end{figure}

\begin{figure}
\includegraphics[scale=0.2]{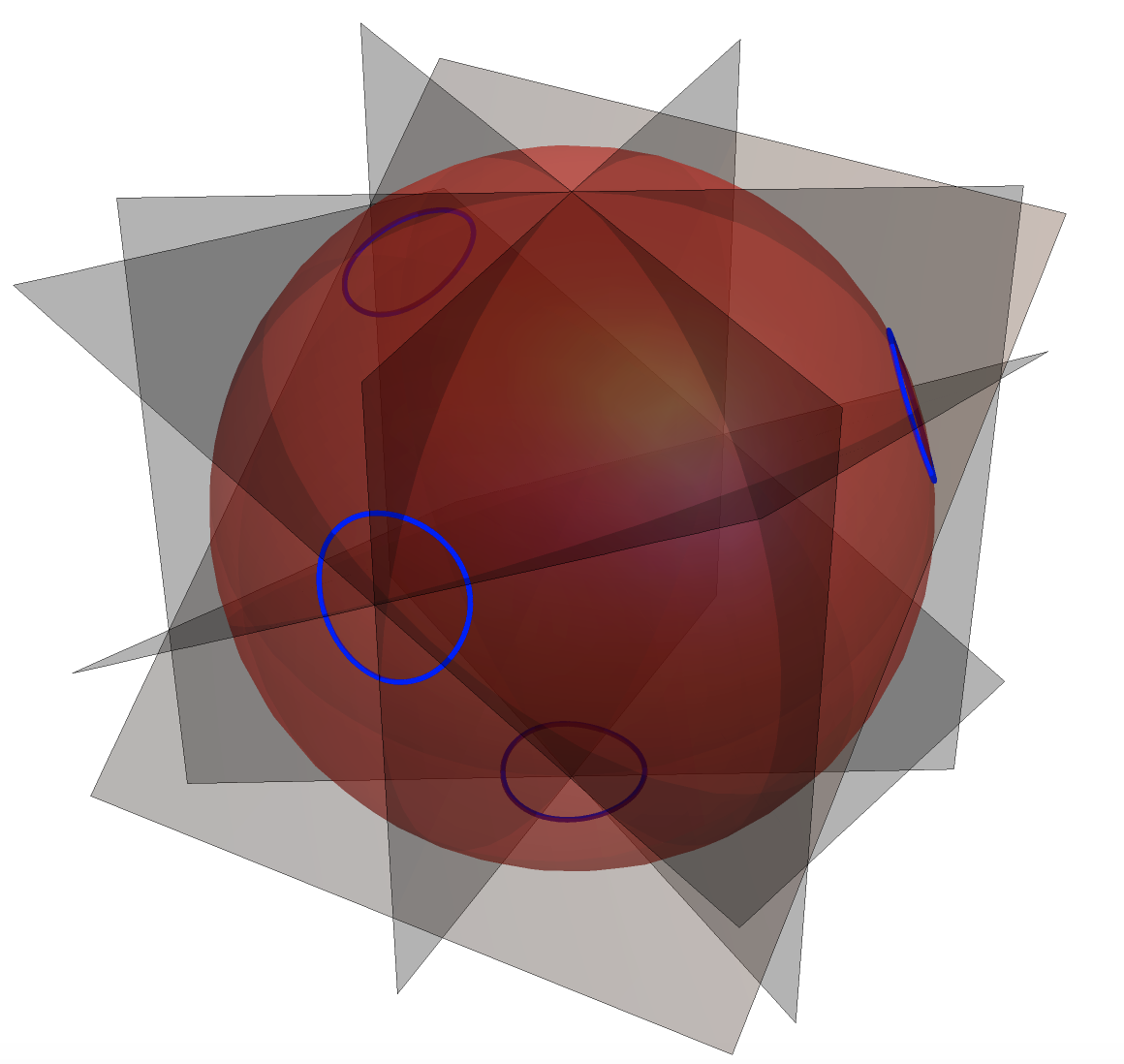}
\caption{Vandermonde variety $V^{(4)}_{3,\y}$ in Case \ref{itemlabel:eg:key:c}.}
\label{fig:image2}
\end{figure}

\begin{figure}
\includegraphics[scale=0.2]{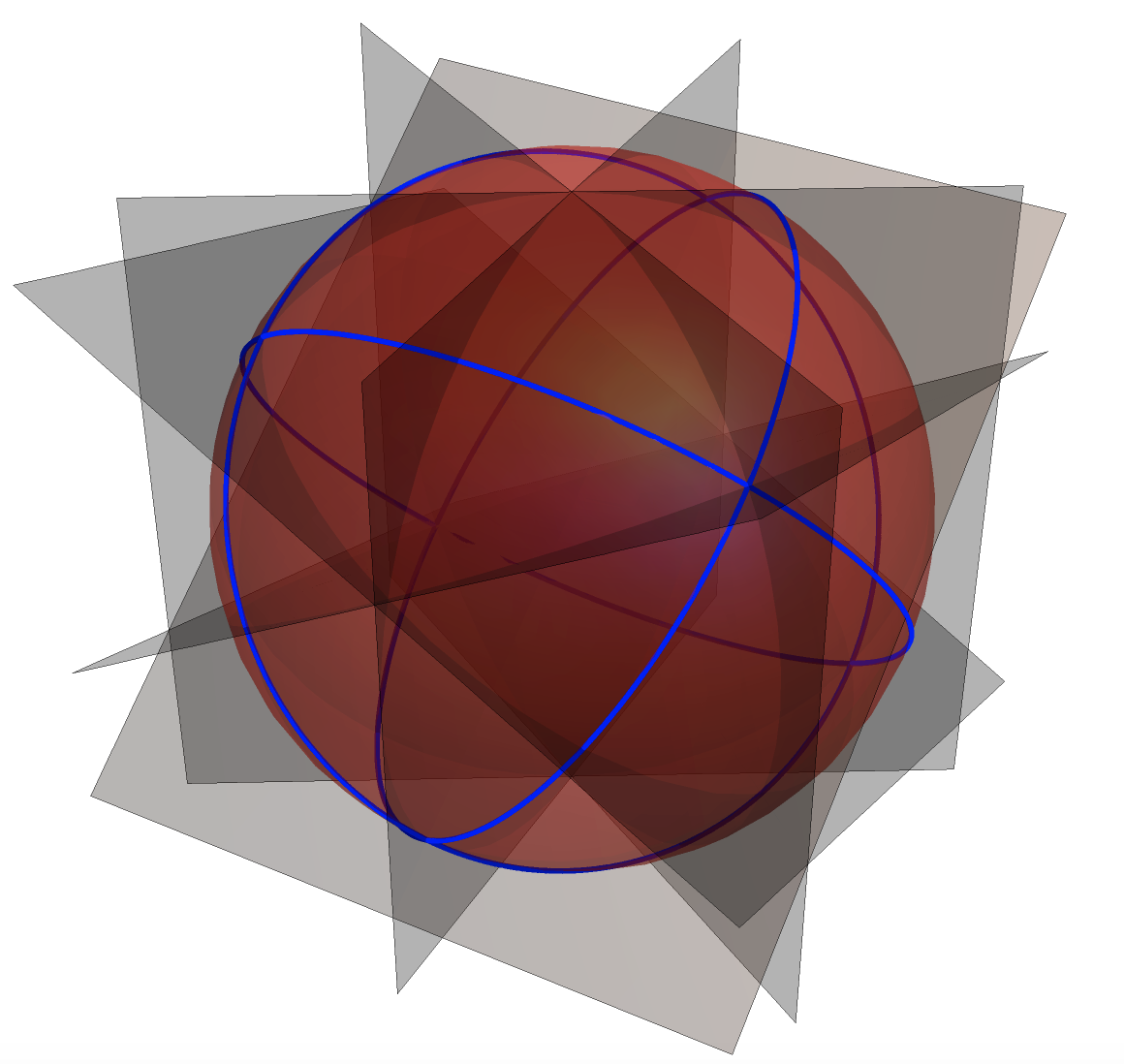}
\caption{Vandermonde variety $V^{(4)}_{3,\y}$ in Case \ref{itemlabel:eg:key:d}.}
\label{fig:image3}
\end{figure}

\begin{figure}
\includegraphics[scale=0.2]{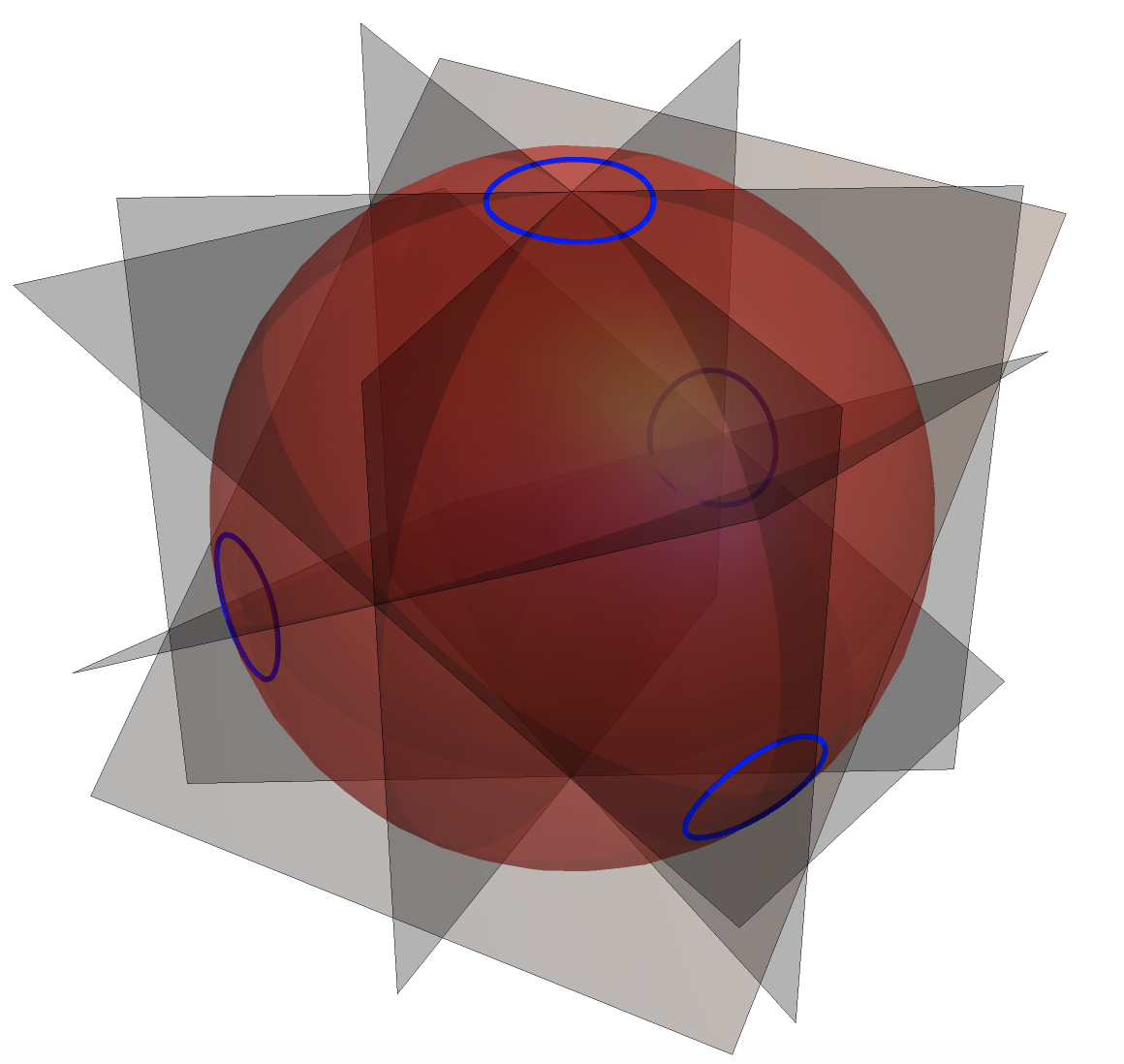}
\caption{Vandermonde variety $V^{(4)}_{3,\y}$ in Case \ref{itemlabel:eg:key:e}.}
\label{fig:image4}
\end{figure}

\begin{figure}
\includegraphics[scale=0.2]{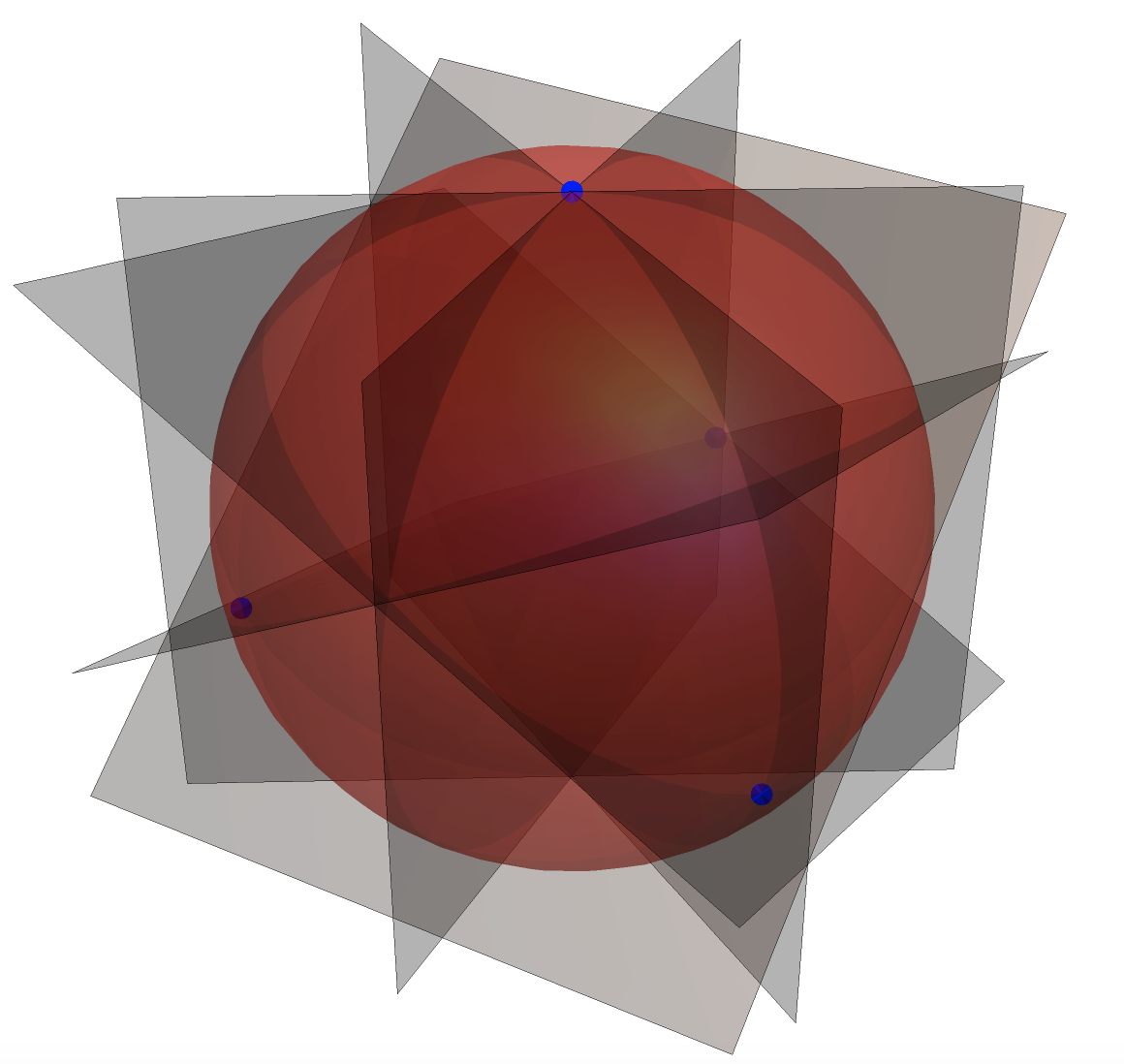}
\caption{Vandermonde variety $V^{(4)}_{3,\y}$ in Case \ref{itemlabel:eg:key:f}.}
\label{fig:image5}
\end{figure}
}
%%sb end of hide
It follows from Theorem~\ref{thm:arnold} that there exist,
\[
a(y_1,y_2), b(y_1,y_2), c(y_1,y_2) \in \R,
\] 
with
\[
 a(y_1,y_2) = \min_{\x \in V^{(4)}_{2,(y_1,y_2)} }p_3^{(4)}(\x) < b(y_1,y_2) < c(y_1,y_2) =  \max_{\x \in V^{(4)}_{2,(y_1,y_2)}} p_3^{(4)}(\x), 
 \]
giving a partition of $\R$ into points and open intervals (more precisely, 
three points and four open intervals)
such that the Vandermonde variety $V^{(4)}_{3,\y}$ can be characterized topologically by which element of the partition
$y_3$ belongs to.
 
\begin{enumerate}[{\ref{itemlabel:eg:key:outer:3}}a.]
\item
\label{itemlabel:eg:key:a}
$y_3 \in (-\infty, a(y_1,y_2))$:  In this case,  $V^{(4)}_{3,\y} = \emptyset$;

\item
\label{itemlabel:eg:key:b}
$y_3= a(y_1,y_2)$: In this case,  $V^{(4)}_{3,\y}$ is non-empty and singular, and coincides with 
$4$ of the $8$ vertices of degree $6$, and 
$Z_{3,\y}^{(4)}$ is a point which must necessarily belong to  the face labeled by $(3,1)$ (cf. Theorem~\ref{thm:arnold}).
In this case 
\[
\HH^0(Z_{3,\y,4}^{(4)}, Z_{3,\y}^{(4,T)}) = 0
\] 
if 
\[
T = \{s_2\}, \{s_3\}, \{s_2,s_3\}, \{s_1,s_2,s_3\}
\] 
(since in these cases $Z_{3,\y}^{(4)} = Z_{3,\y}^{(4,T)}$),
and 
\[
\HH^0(Z_{3,\y}^{(4)},Z_{3,\y}^{(4,T)}) \cong \Q
\]
in the case 
\[
T =\emptyset, \{s_1\}.
\]
This implies that
\begin{eqnarray*}
\HH^0(V_{3,\y}^{(4)}) &\cong_{\mathfrak{S}_4}& \Psi_{\emptyset}^{(4)} \oplus \Psi_{\{s_1\}}^{(4)} \\
&\cong_{\mathfrak{S}_4}& 1_{\mathfrak{S}_4} \oplus \mathbb{S}^{(3,1)} \mbox{ (using \eqref{eqn:eg:1} and \eqref{eqn:eg:3})}.
\end{eqnarray*}

It follows that 
\[
b_0(V_{3,\y}^{(4)}) = 1+3 = 4
\] 
(using \eqref{eqn:hook} to derive $\dim_\Q ( \mathbb{S}^{(3,1)}) = 3$).
Clearly, $\HH^1(V_{3,\y}^{(4)}) = 0$ in this case.

\item
\label{itemlabel:eg:key:c}
$y_3 \in (a(y_1,y_2),b(y_1,y_2))$:
In this case $V_{3,\y}^{(4)}$ is a non-singular curve, and 
$Z_{3,\y}^{(4)}$ intersects the faces labeled by $(1,1,2)$ and $(1,2,1)$ corresponding to
Coxeter elements $s_3$ and $s_2$ respectively.

In this case,  
\[
\HH^0(Z_{3,\y}^{(4)}, Z_{3,\y}^{(4,T)}) = 0
\]
 if 
 \[
 T = \{s_2\}, \{s_3\}, \{s_2,s_3\}, \{s_1,s_2,s_3\}
 \] 
 and 
\[
\HH^0(Z_{3,\y}^{(4)},Z_{3,\y}^{(4,T)}) \cong \Q
\] 
if
\[
T =\emptyset, \{s_1\}.
\]
This implies that
\begin{eqnarray*}
\HH^0(V_{3,\y}^{(4)}) &\cong_{\mathfrak{S}_4}& \Psi_{\emptyset}^{(4)} \oplus \Psi_{\{s_1\}}^{(4)} \\
&\cong_{\mathfrak{S}_4}& 1_{\mathfrak{S}_4} \oplus \mathbb{S}^{(3,1)} \mbox{ (using \eqref{eqn:eg:1} and \eqref{eqn:eg:3})}.
\end{eqnarray*}

In dimension one we have,

\[
\HH^1(Z_{3,\y}^{(4)}, Z_{3,\y}^{(4,T)}) = 0
\]
 if 
 \[
 T = \emptyset,\{s_1\} \{s_2\}, \{s_3\}, \{s_1,s_3\}, \{s_1,s_2\}
 \] 
 and 
\[
\HH^1(Z_{3,\y}^{(4)},X_{3,\y}^{(4,T)}) \cong \Q
\] 
if 
\[
T = \{s_2,s_3\}, \{s_1,s_2,s_3\}.
\]
This implies that
\begin{eqnarray*}
\HH^1(V_{3,\y}^{(4)}) &\cong_{\mathfrak{S}_4}& \Psi_{\{s_2,s_3\}}^{(4)} \oplus \Psi_{\{s_1,s_2,s_3\}}^{(4)} \\
&\cong_{\mathfrak{S}_4}& \mathbb{S}^{2,1,1} \oplus \mathbf{sign}_4 \mbox{ (using \eqref{eqn:eg:4} and \eqref{eqn:eg:2})}.
\end{eqnarray*}

It follows that 
\[
b_0(V_{3,\y}^{(4)}) = 1+3 = 4,
\]
and  
\[
b_1(V_{3,\y}^{(4)}) = 3+1 = 4.
\]

\item
\label{itemlabel:eg:key:d}
$y_3 = b(y_1,y_2)$:
In this case,
the Vandermonde variety $V_{3,\y}^{(4)}$ is of dimension $1$ but has singularities, and  $Z_{3,\y}^{(4)}$ intersects the faces labeled by $(2,2)$ and $(1,2,1)$ (the intersection with the face labeled $(1,2,1)$ are the singular points of  $V_{3,\y}^{(4)}$).
Thus, $Z_{3,\y}^{(4)}$ intersects the faces
labeled by 
Coxeter elements $s_1,s_2$ and $s_3$.

In this case,  
\[
\HH^0(Z_{3,\y}^{(4)}, Z_{3,\y}^{(4,T)}) = 0
\] 
if 
\[
T =\{s_1\}, \{s_2\}, \{s_3\}, \{s_1,s_3\}, \{s_1,s_2\}, \{s_2,s_3\}, \{s_1,s_2,s_3\},
\]
and
\[
\HH^0(Z_{3,\y}^{(4)},Z_{3,\y}^{(4,T)}) \cong \Q
\]
if 
\[
T =\emptyset.
\]
This implies that
\begin{eqnarray*}
\HH^0(V_{3,\y}^{(4)}) &\cong_{\mathfrak{S}_4}& \Psi_{\emptyset}^{(4)} \\
&\cong_{\mathfrak{S}_4}& 1_{\mathfrak{S}_4}.
\end{eqnarray*}

In dimension one we have,
\[
\HH^1(Z_{3,\y}^{(4)}, Z_{3,\y}^{(4,T)}) = 0,
\] 
if 
\[
T = \emptyset, \{s_1\}, \{s_2\}, \{s_3\}, \{s_1,s_3\},
\] 
and 
\[
\HH^1(Z_{3,\y}^{(4)},Z_{3,\y}^{(4,T)}) \cong \Q,
\]
if
\[
T = \{s_1,s_2\} , \{s_2,s_3\}, \{s_1,s_2,s_3\}.
\]
This implies that
\begin{eqnarray*}
\HH^1(V_{3,\y}^{(4)}) &\cong_{\mathfrak{S}_4}& \Psi_{\{s_1,s_2\}}^{(4)} \oplus  \Psi_{\{s_2,s_3\}}^{(4)} \oplus \Psi_{\{s_1,s_2,s_3\}}^{(4)} \\
&\cong_{\mathfrak{S}_4}& 2 \mathbb{S}^{2,1,1} \oplus \mathbf{sign}_4 \mbox{ (using \eqref{eqn:eg:4} and \eqref{eqn:eg:2})}.
\end{eqnarray*}

It follows that 
\[
b_0(V_{3,\y}^{(4)}) = 1,
\] 
and  
\[
b_1(V_{3,\y}^{(4)}) = 2 \cdot 3+1 = 7.
\]

This last equation can be verified directly by hand noting that $V_{3,\y}^{(4)}$ has the structure of a connected graph 
containing $6$ vertices (the $\binom{4}{2}$ singular points consisting of the orbit of  the point $Z_{3,\y}^{(4)} \cap \Weyl_{(2,2)}$), and the degree of each vertex is $4$. Thus the graph has $12$ edges, and hence 
\begin{eqnarray*}
\chi(V_{3,\y}^{(4)}) &=& -6 \\
&=& b_0(V_{3,\y}^{(4)}) - b_1(V_{3,\y}^{(4)}) \\
&=& 1 - b_1(V_{3,\y}^{(4)}),
\end{eqnarray*}
and thus,
\[
b_1(V_{3,\y}^{(4)}) = 7.
\]
\item
\label{itemlabel:eg:key:e}
$y_3 \in (b(y_1,y_2) , c(y_1,y_2))$:
In this case, 
$V_{3,\y}^{(4)}$ is a non-singular curve, and  $Z_{3,\y}^{(4)}$ intersects the faces labeled by $(2,1,1)$ and $(1,2,1)$ corresponding to
Coxeter elements $s_1$ and $s_2$ respectively.
The isotypic decomposition of $\HH^*(V_{3,\y}^{(4)}) $ in this case  is identical to the Case  \eqref{itemlabel:eg:key:c}
and is omitted.

\item
\label{itemlabel:eg:key:f}
$y_3= c(y_1,y_2)$: In this case, $V^{(4)}_{3,\y}$ is non-empty and singular, and coincides with 
other $4$ (compared to Case \eqref{itemlabel:eg:key:b}) of the $8$ vertices of degree $6$. In this case, 
$Z_{3,\y}^{(4)}$ is a point which must necessarily belong to  the face labeled by $(1,3)$.
The isotypic decomposition of $\HH^*(V_{3,\y}^{(4)}) $ in this case is identical to the Case \eqref{itemlabel:eg:key:b}
and is omitted.

\item
\label{itemlabel:eg:key:g}
$y_3  \in (c(y_1,y_2), \infty)$:
In this case,  $V^{(4)}_{3,\y}$ is again empty.
\end{enumerate}
\end{enumerate}

Notice, that the Specht module $\mathbb{S}^{(2,2)}$ does not appear with positive multiplicity in 
$\HH^*(V_{3,\y}^{(4)})$, $\y \in \R^3$ in the above analysis. Using an equivariant Leray spectral sequence
argument (cf. proof of Theorem~\ref{thm:main1}) we can deduce  from this fact the following `toy' theorem 
(which is not directly deducible from the statement of Theorem~\ref{thm:main1}):

\begin{theorem*}
If $S \subset \R^4$ is a $\mathcal{P}$-semi-algebraic set, for $\mathcal{P} \subset \R[X_1,\ldots,X_4]^{\mathfrak{S}_4}_{\leq 3}$, then
\[
m_{i,(2,2)}(S)=0.
\]
\end{theorem*}

\begin{proof}
See proof of Theorem~\ref{thm:main1} and the preceding remark.
\end{proof}

\begin{remark}
\label{rem:eg:V-4-3}
Note that it follows from the analysis in Example~\ref{eg:V-4-3} that
\begin{eqnarray*}
\max_{\y \in \R^3, \lambda \in \Par_0(V^{(4)}_{3,\y})}\length(\lambda) &=& 2, \\
\max_{\y \in \R^3, \lambda \in \Par_1(V^{(4)}_{3,\y})} \length(\lambda) &=& 4,
\end{eqnarray*}
while the Part \eqref{itemlabel:thm:vandermonde:a} of Theorem~\ref{thm:vandermonde} provides the upper bounds:
\begin{eqnarray*}
\max_{\y \in \R^3, \lambda \in \Par_0(V^{(4)}_{3,\y})} \length(\lambda) &<& 0 + 2\cdot 3 -1 = 5, \\
\max_{\y \in \R^3, \lambda \in \Par_1(V^{(4)}_{3,\y})} \length(\lambda) &< & 1 + 2\cdot 3 -1 = 6.
\end{eqnarray*}
Thus,   Example~\ref{eg:V-4-3} is in agreement with Theorem~\ref{thm:vandermonde}.
\end{remark}

We now return to the proofs of the main theorems.

\section{Proofs of Theorems~\ref{thm:main1} and \ref{thm:vandermonde} }
\label{sec:proofs1}
We first need a few preliminary results.
\subsection{Preliminary Results}
We start by recalling a standard definition.

\begin{definition}
\label{def:regular}
We say that a semi-algebraic set $S \subset \R^k$  is a semi-algebraic regular cell of dimension $p$, if
the pair $(\overline{S},S)$ is semi-algebraically homeomorphic to $(\overline{B_p(\mathbf{0},1)},B_p(\mathbf{0},1))$
where $B_p(\mathbf{0},1)$ denotes the unit ball in $\R^p$.  
\end{definition}

\begin{remark}[Monotonicity and regularity of semi-algebraic sets]
\label{rem:regular}
We will prove in Proposition~\ref{prop:agk} that the intersections of weighted
 Vandermonde varieties with the interior  of $\Weyl^{(k)}$ is a semi-algebraic regular cell of dimension $k-d$, if
 the dimension of the variety is equal to $k-d$, 
 and this property will play an important role later in the paper
 (see Lemma~\ref{lem:regular} and Proposition~\ref{prop:cell-complex}).
 To prove that a given semi-algebraic set is a semi-algebraic regular cell is often not easy.
In order to overcome this difficulty, a stronger notion, that of  a \emph{monotone cell},  was introduced in \cite{BGV2012}. The property that a semi-algebraic set is a monotone cell is much easier to check. 
We do not reproduce the definition of a monotone cell here but refer the reader to \cite[Theorem 9]{BGV2012}
for one of the several equivalent definitions which is the easiest to check for the sets $Z_{\w,d,\y}^{(k)}$.
Finally, the main result (Theorem 6) in \cite{BGV2012} states that a semi-algebraic set which is a monotone cell
is a semi-algebraic regular cell, which is what we will use in the proof of  Proposition~\ref{prop:agk}. 
\end{remark}

The following proposition which has been referred to before, 
and which describes  the topological structure of the intersection of a general
Vandermonde variety with a Weyl chamber,  is a key topological ingredient in our proofs.

\begin{proposition}
\label{prop:agk}
For every $\w \in \R_{> 0}^k$,  $d,k \geq 0, d \leq k$, and $\y \in \R^d$,
$Z_{\w,d,\y}^{(k)}$  is either empty, 
a point, or semi-algebraically homeomorphic 
a semi-algebraic regular cell of dimension $k-d$.
\end{proposition}

\begin{proof}
Suppose that $Z_{\w,d,\y}^{(k)}$  is not empty.
Let $\x \in Z_{\w,d,\y}^{(k)}$ and suppose that $\x$ is a regular point
of the intersection of the Vandermonde variety $V_{\w,d,\y}$ with the linear subspace $L_\lambda$
(i.e. the linear hull of the  face $\Weyl_{\lambda}$) for some $\lambda \in \Comp(k)$. Then,
$\x$ is a regular point of $V_{\w,d,\y}$, and $\x \in \overline{Z_{\w,d,\y}^{(k)} \cap \Weyl^{(k),o}}$. \\

We next prove that if $Z_{\w,d,\y}^{(k)} \neq \overline{Z_{\w,d,\y}^{(k)} \cap \Weyl^{(k),o}}$, then
$Z_{\w,d,\y}^{(k)}$ must be a point. 
Indeed, if $\x \in Z_{\w,d,\y}^{(k)}$, but $\x \not\in \overline{Z_{\w,d,\y}^{(k)} \cap \Weyl^{(k),o}}$,
then by the above observation and
\cite[Theorem 5]{Arnold}, 
$\x \in \Weyl_\lambda^o$, with
$\length(\lambda) < d$, and moreover $Z_{\w,d,\y}^{(k)} \cap \Weyl_\lambda^o = \{\x\}$.
Moreover, in this case $\x$ must be an isolated point of 
$Z_{\w,d,\y}^{(k)}$, since any neighborhood of $\x$ in $Z_{\w,d,\y}^{(k)}$, unless equal to just $\x$ itself,
will contain some regular point $\x'$ of the intersection of $Z_{\w,d,\y}^{(k)}$ with  $L_{\lambda'}$ with
$\lambda \prec \lambda'$, and this would imply that $\x \in \overline{Z_{\w,d,\y}^{(k)} \cap \Weyl^{(k),o}}$.
But on the other hand we know that $Z_{\w,d,\y}^{(k)}$ is contractible \cite[Theorem 1.1]{Kostov}.
This proves that in this case $Z_{\w,d,\y}^{(k)} = \{\x\}$, and hence 
if $Z_{\w,d,\y}^{(k)} \neq \overline{Z_{\w,d,\y}^{(k)} \cap \Weyl^{(k),o}}$, 
$Z_{\w,d,\y}^{(k)}$ is a point.  \\

So we might suppose that 
\begin{equation}
\label{eqn:prop:agk}
Z_{\w,d,\y}^{(k)} = \overline{Z_{\w,d,\y}^{(k)} \cap \Weyl^{(k),o}}.
\end{equation}
 
In this case $Z_{\w,d,\y}^{(k)} \cap \Weyl^{(k),o} \neq \emptyset$, and 
using \cite[Theorem 5]{Arnold}  $Z_{\w,d,\y}^{(k)} \cap \Weyl^{(k),o}$ is  non-singular of dimension
$k-d$.
Now using \cite[Corollary 2.2]{Kostov}, and \cite[Theorem 9]{BGV2012}
we deduce that 
$Z_{\w,d,\y}^{(k)} \cap \Weyl^{(k),o}$
is a monotone
cell (see \cite{BGV2012} for the definition of a monotone cell). 
This implies using  \cite[Theorem 13]{BGV2012}  that 
$Z_{\w,d,\y}^{(k)} \cap \Weyl^{(k),o}$
is a regular cell. 
In conjunction with  \eqref{eqn:prop:agk} 
this implies that
$Z_{\w,d,\y}^{(k)}$ is semi-algebraically homeomorphic to the closure of a regular cell, and the boundary of  
$Z_{\w,d,\y}^{(k)}$ is semi-algebraically homeomorphic to the sphere $\Sphere^{k-d-1}$.
\end{proof}

\begin{remark}
\label{rem:agk}
Using Proposition~\ref{prop:agk} again on the intersection of $Z_{\w,d,k}$ with the faces of $\Weyl^{(k)}$ we get that
if $Z_{\w,d,k}$ is not empty or a point, then its boundary is a regular cell complex (homeomorphic to
$\Sphere^{k-d-1}$).
\end{remark}

\begin{definition}
\label{def:Leray}
Let $X$ be a closed and bounded semi-algebraic set and $\mathcal{C}$ be a finite set of closed semi-algebraic 
subsets of $X$. We say that $\mathcal{C} = (C_i)_{i \in I}$, where $I$ is a finite set,  is a \emph{closed Leray cover}
of $X$ if $\mathcal{C}$ satisfies:
\begin{enumerate}[(a)]
\item
$X = \bigcup_{i \in I } C_i$;
\item
for  each subset $J  \subset I$, $\bigcap_{j \in J} C_j$ is empty or
semi-algebraically contractible.
\end{enumerate}
We say that  $\mathcal{C}$ is a \emph{regular} closed Leray cover if in addition  
for  each subset $J \subset I$, $\bigcap_{j \in J} C_j$ is empty or
the closure of a regular semi-algebraic  cell.
\end{definition}

\begin{notation}[Nerve complex associated to  a closed Leray cover]
\label{not:nerve}
Given a closed Leray cover $\mathcal{C} = (C_i)_{i \in I}$ with $I = \{1,\ldots,N\}$,
we will denote by $\mathcal{N}(\mathcal{C})$ the  simplicial complex whose set of $p$-dimensional simplices are given by
\[
\mathcal{N}_p(\mathcal{C}) = \{(\alpha_0,\ldots,\alpha_p) \mid 1 \leq \alpha_0 < \cdots < \alpha_p \leq  N, C_{\alpha_0} \cap \cdots \cap C_{\alpha_p} \neq \emptyset\}.
\]
\end{notation}

We need the following technical lemma in the proof of Proposition~\ref{prop:cell-complex} which
plays an important role in the proof of Theorem~\ref{thm:vandermonde}.

\begin{lemma}
\label{lem:regular}
Let $(P_i)_{i \in I}$, and $(Q_j)_{j\in J}$ be finite tuples of polynomials
in $\R[X_1,\ldots,X_k]$, and $S \subset \R^k$ a basic closed semi-algebraic set defined by
\[
\bigwedge_{i \in I} (P_i = 0) \wedge \bigwedge_{j \in J} (Q_j > 0),
\] 
such that the closure 
$\overline{S}$ of $S$ is defined by 
\[
\bigwedge_{i \in I} (P_i = 0) \wedge \bigwedge_{j \in J} (Q_j \geq 0).
\]
Moreover, suppose that the pair $(\overline{S},S)$ is semi-algebraically homeomorphic to
\[
(\overline{B_p(\mathbf{0},1)},B_p(\mathbf{0},1))
\] 
(recall that $B_p(\mathbf{0},1)$ denotes the unit ball in $\R^p$).

Then for all $J' \subset J$, and all sufficiently small $\eps > 0$, the semi-algebraic set
$S_{J',\eps}$ 
(see Figure~\ref{fig:regular-cell})
defined by 
\[
\bigwedge_{i \in I} (P_i = 0) \wedge \bigwedge_{j \in J'} (Q_j \geq \eps) \bigwedge_{j \in J - J'} (Q_j \geq 0) 
\]
is semi-algebraically contractible.
\end{lemma}

\begin{figure}
\includegraphics[scale=0.6]{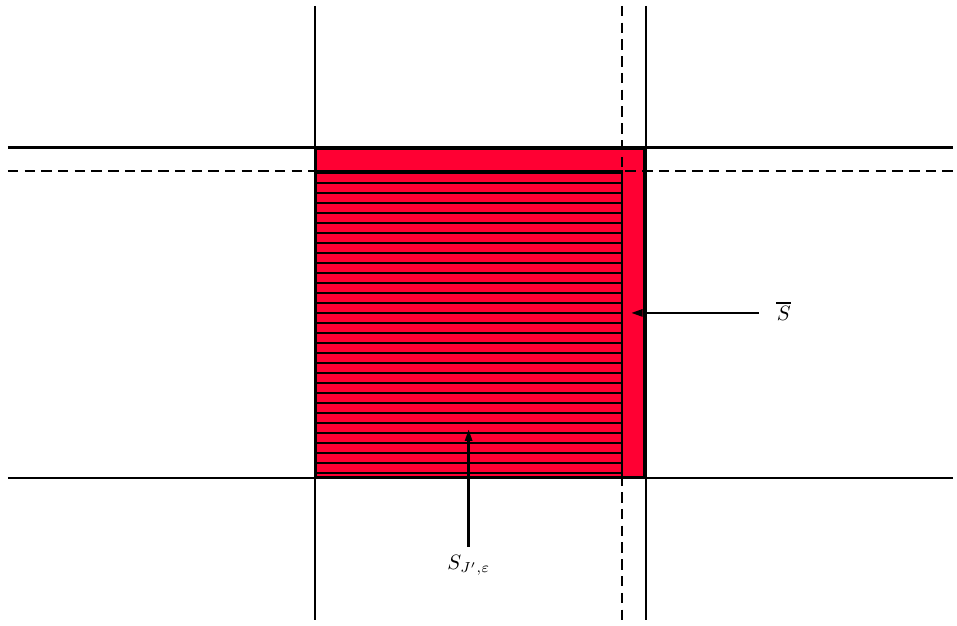}
\caption{Schematic depiction of the sets $\overline{S}$ and $S_{J',\eps}$}
\label{fig:regular-cell}
\end{figure}

\begin{proof}
Let $S', S'' $ be the semi-algebraic subsets of $\overline{S}$ defined by 
\[
\bigwedge_{i \in I} (P_i = 0) \wedge \bigwedge_{j \in J'} (Q_j > 0) \bigwedge_{j \in J - J'} (Q_j \geq 0),
 \]
 and 
\[
\bigwedge_{i \in I}(P_i=0) \wedge \bigwedge_{j' \in J'} (Q_{j'} = 0) \wedge \bigwedge_{j \in J - J'} (Q_j \geq 0),
\]
respectively.
 
Observe that 
\[
S' = \overline{S} - S'',
\]
and
\[
S'' \subset \overline{S} - S.
\]
Let $\phi: \overline{S} \times [0,1] \rightarrow \overline{S}$ be the homeomorphic image of the
standard retraction of $\overline{B_p(\mathbf{0},1)}$ to $\mathbf{0}$ (i.e.
$(\x,t) \mapsto (1-t) \x$).

Since $S''$ is contained in the boundary of $S$, we can
restrict the retraction $\phi$ to
$S' = \overline{S} - S''$
and obtain that $S'$ is also semi-algebraically contractible.
It now follows from the 
the local conic structure theorem for semi-algebraic sets \cite[Theorem 9.3.6]{BCR} 
that for all small enough $\eps >0$
that $S'$ and $S_{J',\eps}$
are semi-algebraically homotopy equivalent, and hence 
$S_{J',\eps}$ is also semi-algebraically contractible.
\end{proof}

\begin{proposition}
\label{prop:cell-complex}
Let $2 \leq d \leq k$, $\y \in \R^d$, $V = V_{d,\y}^{(k)}$, 
$\dim(V) = k-d$,
$K = V\cap \bigcup_{s \in \Coxeter(k)} \Weyl^{(k)}_s$,
$I = \{s \in \Coxeter(k) \mid V\cap \Weyl^{(k)}_s \neq \emptyset\}$.
Let $J \subset I$, and $K^J =  V \cap \bigcup_{s \in J} \Weyl^{(k)}_s$.
Then:
\begin{enumerate}[1.]
\item
\label{itemlabel:prop:cell-complex:0}
$K$ is semi-algebraically homeomorphic to the $\Sphere^{k-d-1}$.
\item
\label{itemlabel:prop:cell-complex:1}
The tuple $\mathcal{C} = (V_s = V \cap \Weyl^{(k)}_s)_{s \in I}$ is a regular closed Leray cover of $K$.

\item
\label{itemlabel:prop:cell-complex:a}
$\HH^i(K^J) = 0$ for $i \geq \card(J)$.
\item
\label{itemlabel:prop:cell-complex:b}
$\HH^i(K^J) = 0$ for  
$0 < i \leq \card(J) - d-1$.

\item
\label{itemlabel:prop:cell-complex:c}
$\HH^0(K^J) \cong \Q$ if 
$\card(J) \geq d+1$.
\end{enumerate}
\end{proposition}

\begin{proof}
Parts~\eqref{itemlabel:prop:cell-complex:0}
and \eqref{itemlabel:prop:cell-complex:1}
are immediate from Proposition~\ref{prop:agk}, since each intersection of the 
various $V_s$ are semi-algebraically homeomorphic to some $Z_{\w,d,\y}^{(p)}$ for some
$p$, $0 \leq p < k$, and $\w \in \mathbb{Z}_{>0}^p$  (using the notation from Proposition~\ref{prop:agk}), and
is thus empty, a point,  or semi-algebraically homeomorphic to a regular cell of dimension $p$. \\
 
It follows from the nerve lemma that $\HH^*(K^J) \cong \HH^*(\mathcal{N}(\mathcal{C}^J))$,
where $\mathcal{C}^J = (V_s)_{s\in J}$. Since $\mathcal{N}(\mathcal{C}^J)$ is a simplicial complex
with $\card(J)$ vertices, $\HH^i(\mathcal{N}(\mathcal{C}^J)) = 0$ for $i \geq \card(J)$. This proves Part~\eqref{itemlabel:prop:cell-complex:a}. \\

We now prove Parts~\eqref{itemlabel:prop:cell-complex:b} and \eqref{itemlabel:prop:cell-complex:c}.
We can assume that $J \neq \emptyset$ which implies that  
$K^J \neq \emptyset$, since otherwise the claim is obviously true.

For $s = (i,i+1) \in \Coxeter(k)$, let $P_s$ denote the polynomial $X_{i+1} - X_{i}$.
Then, for each $s \in  I$,
$V_s$ is  the intersection with $V$ of the semi-algebraic set defined by
\[
(P_s = 0) \wedge \bigwedge_{s' \in \Coxeter(k) - \{s\}} (P_{s'} \geq 0).
\]

For $\eps >0$,
denote by $K^J_\eps$  the union of $V_{s,\eps}, s \in J$, where
$V_{s,\eps}$ is the intersection with $K$ of the open semi-algebraic set defined by 
\[
 (-\eps <  P_{s} < \eps) \wedge \bigwedge_{s'  \in \Coxeter(k) - \{s\}} (P_{s'} > -\eps).
\]

Then, using the local conic structure theorem for semi-algebraic sets \cite[Theorem 9.3.6]{BCR}, 
for all small enough $\eps > 0$, $K^J_\eps$ is semi-algebraically homotopy equivalent to 
$K^J$ and  $K - K^J_\eps$ is closed and semi-algebraically homotopy equivalent to $K - K^J$.  \\

We now claim that for all small enough $\eps > 0$,
$(V_s - K^J_\eps)_{s \in I - J}$ is a closed Leray cover of $K - K^J_\eps$.
Let $J' \subset I-J$, and 
consider $\bigcap_{s \in J'} (V_s - K^J_\eps)$. Then, there exists $J'' \subset J$ such that $\bigcap_{s \in J'} (V_s - K^J_\eps)$ is the intersection 
with $V$ of the semi-algebraic set defined by 
\[
\bigwedge_{s \in J'} (P_s = 0) \wedge \bigwedge_{s \in J''} (P_s \geq \eps) \wedge \bigwedge_{s \in \Coxeter(k) - (J' \cup J'')} (P_s \geq 0).
\]
It follows from  Lemma~\ref{lem:regular} and the above description that for all 
$\eps >0$ small enough,
$\bigcap_{s \in J'} (V_s - K^J_\eps)$ is either empty or semi-algebraically contractible, and hence
$(V_s - K^J_\eps)_{s \in I - J}$ is a closed Leray cover of $K - K^J_\eps$.
Using the same argument involving the nerve complex as in the previous paragraph we obtain that 
\[
\HH^i(K - K^J_\eps) = 0
\]
for 
$i \geq \card(I) - \card(J)$.
However, by Alexander duality (see for example \cite[page 296]{Spanier}) we have that

\begin{equation}
\label{eqn:prop:cell-complex:1}
\tilde{\HH}^i(K^J) \cong \tilde{\HH}^i(K^J_\eps)  \cong \tilde{\HH}_{n-i-1}(K- K^J_\eps) .
\end{equation}

Let $n = k-d-1$.
It follows from Part~\eqref{itemlabel:prop:cell-complex:a}  and 
\eqref{eqn:prop:cell-complex:1} that 
$\tilde{\HH}^i(K^J) = 0$ for 
$n - i -1   \geq \card(I) - \card(J)$ or equivalently for
$i \leq n - \card(I) + \card(J) - 1$. \\

Since, $\card(I) \leq n+d$, it follows that 
$\tilde{\HH}^i(K^J) = 0$ for 
$0 \leq i \leq \card(J) - d -1 $.
Parts~\eqref{itemlabel:prop:cell-complex:b} and \eqref{itemlabel:prop:cell-complex:c} of the proposition follows.
\end{proof}

\subsection{Proofs of Theorems~\ref{thm:main1} and \ref{thm:vandermonde} }

 \begin{proof}[Proof of Theorem \ref{thm:vandermonde}]
Let $V = V_{d,\y}^{(k)}$.
We first prove Part~\eqref{itemlabel:thm:vandermonde:a}.
From 
Proposition~\ref{prop:agk} we have that $V$ is either empty, or a finite union of points, or of dimension $k-d$. 
If $V$ is empty there is nothing to prove.
Suppose that $V$ is not empty. \\

Using Theorem~\ref{thm:Davis} we have that 
\begin{equation}
\label{eqn:thm:Davis}
\HH^i(V) \cong \bigoplus_{T \subset \Coxeter(k)} \HH^i(V_{k},V_{k}^T) \otimes_\Q \Psi_{T}^{(k)}.
\end{equation}

Since we have from Proposition~\ref{prop:Solomon} that
\begin{equation*}
\mult_{\mathbb{S}^{\lambda}}(\Psi^{(k)}_{T}) =  0 \mbox{ if }
\length(\lambda) > \card(T)+1,
\end{equation*}
we might as well also assume that 
\[
\length(\lambda) \leq  \card(T)+1,
\]  
or that
\[
\card(T) \geq \length(\lambda) - 1.
\] 

It thus suffices to prove that $\HH^i(V_{k},V_{k}^T) = 0$,
for all pairs $(i,T)$ satisfying:
\begin{eqnarray*}
i &\leq& \length(\lambda)- 2d + 1, \\
\card(T) &\geq& \length(\lambda) - 1,
\end{eqnarray*}
for which it suffices to prove that 
$\HH^i(V_{k},V_{k}^T) = 0$ for  all $(i,T)$ satisfying 
\begin{equation}
\label{eqn:restriction-on-T}
i \leq \card(T) - 2d+2 \Leftrightarrow \card(T) \geq i +2d - 2.
\end{equation}

We now fix the pair $(i,T)$ satisfying \eqref{eqn:restriction-on-T},
and treat the cases $i=0$, $i = 1$, and  $i >1$ separately. \\

\noindent Case $i=0$:
In this case, if $V_k \neq \emptyset$, $\HH^0(V_k,V_k^T) \neq 0$ if and only if 
$V_k^T = \emptyset$. If $V_k \neq \emptyset$, it must meet a $d$-dimensional face of the 
$\Weyl^{(k)}$, which is incident on $k-d$ of the $k-1$ codimension one faces,
$\Weyl^{(k)}_s, s \in \Coxeter(k)$, of $\Weyl^{(k)}$. 
This implies that
\[
V_k^T = \emptyset \Rightarrow \card(T) \leq d-1.
\]

Since, for $d > 1$, $2d - 2 > d -1$,   
it follows that 
\[
\card(T) \geq i +2d - 2 = 2d -2 \Rightarrow \card(T) > d-1 \Rightarrow 
V_k^T \neq  \emptyset \Rightarrow  \HH^0(V_k,V_k^T) = 0.
\]
This completes the proof of Part~\eqref{itemlabel:thm:vandermonde:a} in the case $i=0$. \\

Now suppose that $i > 0$.
Let for $s \in \Coxeter(k)$, $V_s = V \cap \Weyl^{(k)}_s$.
We denote (following the notation in Proposition~\ref{prop:cell-complex})
\begin{eqnarray*}
I &=& \{s \in \Coxeter(k) \mid V_s \neq \emptyset \}, \\
J_T &=& T\cap I, \\
K &=& \bigcup_{s \in I} V_s, \\
K^{J_T} &=& \bigcup_{s \in J_T} V_s  \;\;= \;\;V_k^T.
\end{eqnarray*}

Using Parts~\eqref{itemlabel:prop:cell-complex:0}
and \eqref{itemlabel:prop:cell-complex:1} of Proposition \ref{prop:cell-complex},
$K$ is semi-algebraically homeomorphic  to
$\Sphere^{n}$, with $n= k-d-1$,
$\mathcal{C} = (V_s)_{s \in I}$,  is a regular closed Leray cover of $K$ (cf. Definition~\ref{def:Leray}). \\

It follows from \cite[Theorem 7]{Arnold} that the maximum and minimum of 
$p^{(k)}_{d+1}$ is obtained on $V_k$ in two distinct $d$-dimensional faces of $\Weyl^{(k)}$. Moreover,
each of these two distinct $d$-dimensional faces are incident on exactly $k-d$ codimension one faces, $\Weyl^{(k)}_s, s \in \Coxeter(k)$, of $\Weyl^{(k)}$.  
We thus have
 \begin{equation}
\label{eqn:card1}
 k-d+1 \leq \card(I) \leq k-1 = n+d.
 \end{equation}

Clearly, $\card(J_T) =  \card(T\cap I)  \leq \card(T)$.
On the other hand, 
\begin{eqnarray}
\nonumber
\card(J_T) &=&  \card(T \cap I) \\
\nonumber
&=& \card(T) + \card(I) - \card(T \cup I) \\
\nonumber
&\geq&  \card(T) + \card(I)  - \card(\Coxeter(k)) \\
\nonumber
&\geq&  \card(T) + \card(I)  - (k-1) \\
\nonumber
&\geq & \card(T) + (k-d +1) -(k-1)  \mbox{ (using inequality \eqref{eqn:card1})} \\
\label{eqn:card2}
&=& \card(T) - d +2.
\end{eqnarray}

\noindent
Case $i =1$:
We only need to consider the case 
$i = 1 \leq  \card(T) - 2d +2$.
We distinguish the following two cases: 
\begin{itemize}
\item
If $T  = \emptyset$, then 
since 
$d > 1$,
the inequality 
$i = 1 \leq \card(T) -2d +2$ 
cannot hold.

\item
If $T \neq \emptyset$, 
and 
$i=1 \leq \card(T) -2d +2$, then 
\[
\card(J_T) \geq \card(T) - d +2 \geq 2d -1 -d +2 =d+1,
\] 
and it follows from Part~\eqref{itemlabel:prop:cell-complex:c}  of Proposition~\ref{prop:cell-complex}
that
$\HH^0(V_k^T) = \HH^0(K^{J_T})  \cong \Q$.
In this case the restriction homomorphism $\HH^0(V_{k}) \rightarrow \HH^{0}(V_{k}^T)$ is 
an isomorphism
which implies that 
$\HH^1(V_{k},V_{k}^T) = 0$. 
\end{itemize}

\noindent Case $i > 1$:
In this case, we can assume that
$\dim(V) = k-d$. Otherwise, $V$ is zero-dimensional and $\HH^i(V) = 0$ for $i>0$.
From the exactness of the long exact sequence,
\[
\cdots \rightarrow \HH^{i-1}(V_{k}^T) \rightarrow \HH^i(V_{k},V_{k}^T) \rightarrow \HH^i(V_{k}) \rightarrow \cdots
\]
of the pair $(V_{k},V_{k}^T)$ and the fact that $\HH^i(V_{k}) = 0$ for $i \geq 1$, it suffices to prove that
$\HH^{i-1}(V_{k}^T) = 0$ for 
$1< i \leq \card(T) - 2d+2$ or equivalently
$\HH^{j}(V_{k}^T) = 0$ for 
$1 \leq j \leq \card(T) - 2d+1$. \\

Applying Parts~\eqref{itemlabel:prop:cell-complex:a}  and \eqref{itemlabel:prop:cell-complex:b}  of Proposition~\ref{prop:cell-complex}, noting that $K^{J_T} = V_k^T$, we obtain 
\[
\HH^j(K^{J_T}) = \HH^j(V_{k}^T) = 0
\]
for 
$0 < j \leq \card(T) - 2d +1$.
This completes the proof 
for the case $ i >1$. \\
This completes the proof of Part~\eqref{itemlabel:thm:vandermonde:a}. \\

We now prove Part~\eqref{itemlabel:thm:vandermonde:b}. 
First assume that $\dim(V) = k-d$.
Since we have from Proposition~\ref{prop:Solomon} that
\begin{equation*}
\mult_{\mathbb{S}^{\lambda}}(\Psi^{(k)}_{T}) =  0  \mbox{ if } \length(^{t}\lambda) >  k - \card(T), 
\end{equation*}
we might as well also assume that 
\[
\length(^{t}\lambda) \leq  k -\card(T),
\]  
or that
\[
\card(T) \leq k - \length(^{t}\lambda).
\] 

It thus suffices to prove that $\HH^i(V_{k},V_{k}^T) = 0$,
for all pairs $(i,T)$ satisfying:
\begin{eqnarray*}
i &\geq& k - \length(^{t}\lambda)+1, \\
\card(T) &\leq& k - \length(^{t} \lambda),
\end{eqnarray*}
for which it suffices to prove that 
$\HH^i(V_{k},V_{k}^T) = 0$ for  all $(i,T)$ satisfying 
\[
i \geq \card(T)+1.
\]

From the exactness of the long exact sequence,
\[
\cdots \rightarrow \HH^{i-1}(V_{k}^T) \rightarrow \HH^i(V_{k},V_{k}^T) \rightarrow \HH^i(V_{k}) \rightarrow \cdots
\]
of the pair $(V_{k},V_{k}^T)$ and the fact that $\HH^i(V_{k}) = 0$ for $i \geq 1$, it suffices to prove that
$\HH^{i-1}(V_{k}^T) = 0$ for 
$i \geq \card(T)+1$ or equivalently 
$\HH^{j}(V_{k}^T) = 0$ for 
$j \geq \card(T)$. \\

It follows from Part~\eqref{itemlabel:prop:cell-complex:a} of Proposition~\ref{prop:cell-complex}, 
that $\HH^{j}(V_{k}^T) =\HH^{j}(K^{J_T}) = 0$ for 
$j \geq \card(T)$. \\

If $\dim(V) = 0$, we only need to consider the case $i=0$.
In this case, 
we need to show that for $\lambda \vdash  k$ satisfying 
\[
\length(^{t}\lambda)  \geq k+1,
\] 
$m_{0,^{t}\lambda}(V)=0$.
But since $\length(^{t}\lambda)  \leq k$, this case does not occur.
This completes the proof of Part~\eqref{itemlabel:thm:vandermonde:b}.
 \end{proof}

\begin{proof}[Proof of Theorem~\ref{thm:main1}]
First observe that by the 
local conic structure theorem for semi-algebraic sets \cite[Theorem 9.3.6]{BCR},
there exists $R > 0$, such that
the inclusion 
$S \cap \overline{B_k(\mathbf{0},R)} \hookrightarrow S$ is a semi-algebraic homeomorphism.
Moreover, since $S$ and $\overline{B_k(\mathbf{0},R)}$ are both symmetric, the above inclusion is $\mathfrak{S}_k$-equivariant.
Hence
\begin{equation}
\label{eqn:proof:thm:main1:GV}
\HH^*(S \cap \overline{B_k(\mathbf{0},R)}) \cong_{\mathfrak{S}_k}  \HH^*(S).
\end{equation}
Note that $\overline{B_k(\mathbf{0},R)}$ is defined by the symmetric inequality
\[
\sum_{i=1}^2 X_i^2 - R \leq 0
\]
of degree $2$.
This, in view of the isomorphism in \eqref{eqn:proof:thm:main1:GV},
we can assume without loss of generality (after
replacing $S$ by $S \cap \overline{B_k(\mathbf{0},R)}$ and 
$\mathcal{P}$ be $\mathcal{P} \cup \{\sum_{i=1}^2 X_i^2 - R \}$)
that the given semi-algebraic set $S$ is closed and bounded.  \\

Since $S$ is a $\mathcal{P}$-semi-algebraic set, and 
$\mathcal{P} \subset \R[X_1,\ldots,X_k]^{\mathfrak{S}_k}_{\leq d}$, it follows from the fundamental
theorem of symmetric polynomials, that 
\[
S = (\Phi^{(k)}_{d})^{-1}(\Phi^{(k)}_{d}(S)).
\]
Let $f = \Phi^{(k)}_{d}|_{S}$ and observe that $f$ is a proper map. 
We have a spectral sequence (the Leray spectral sequence of the map $f$), 
converging to $\HH^{p+q}(S)$, whose $E_2$-term is given by
\[
E_2^{p,q} = \HH^p (T, R^q f_*(\Q_S)),
\] 
where $T = f(S)$, and $\Q_S$ denotes the constant sheaf on $S$. \\

Using the proper base change theorem (see for example \cite[\S 3, Theorem 6.2]{Iversen}) we obtain
that for $\y \in T$,
\begin{equation}
\label{eqn:stalk}
R^q f_*(\Q_S)_\y \cong \HH^q(V_{d,\y}^{(k)},\Q),
\end{equation}
and this gives $R^q f_*(\Q_S)$ the structure of a sheaf of $\mathfrak{S}_k$-modules.
Moreover, since the action of $\mathfrak{S}_k$  on $S$ leaves the fibers of the map $f:S \rightarrow T$ invariant,
the action of $\mathfrak{S}_k$ on $E_2^{p,q}$ is given by its action on the sheaf $R^q f_*(\Q_S)$. \\

Now, $\HH^n(S)$ is isomorphic as an $\mathfrak{S}_k$-module to a  ($\mathfrak{S}_k$-equivariant) subquotient of 
\[
\bigoplus_{p+q = n} E_2^{p,q}.
\]

Using Theorem \ref{thm:vandermonde}, we have that 
\begin{equation*}
m_{i,\lambda}(V_{d,\y}^{(k)})=0,
\mbox{ for } i \leq \length(\lambda)- 2d + 1.
\end{equation*}
This implies using \eqref{eqn:stalk} that,
\begin{eqnarray}
\nonumber
\mult_{\mathbb{S}^\lambda} (E_2^{p,n-p}) &=&  0,
\mbox{ for } n- p \leq \length(\lambda)- 2d + 1,\\
\label{eqn:main1:1}
&& \mbox{ or equivalently for $n \leq \length(\lambda) -2d +p + 1$}.
\end{eqnarray}

From the fact that $\HH^{p+q}(S)$ is a ($\mathfrak{S}_k$-equivariant) subquotient of 
$
\bigoplus_{p+q} E_2^{p,q},
$ 
and 
\eqref{eqn:main1:1},  we obtain that
\begin{eqnarray*}
m_{i,\lambda}(S)&=&  0
\mbox{ for }n \leq \length(\lambda)- 2d + 1.
\end{eqnarray*}

This proves Part \eqref{itemlabel:thm:main1:a}.\\

In order to prove Part \eqref{itemlabel:thm:main1:b}, recall first that Theorem~\ref{thm:vandermonde} implies that

\begin{equation}
\label{eqn:main1:2}
m_{i,\lambda}(V_{d,\y}^{(k)}) = 0,
\mbox{ for } i \geq k - \length(^{t}\lambda)+1.
\end{equation}

 Using \eqref{eqn:stalk} and \eqref{eqn:main1:2} we obtain that,
\begin{eqnarray}
\nonumber
\mult_{\mathbb{S}^\lambda} (E_2^{p,n-p}) &=&  0,
\mbox{ for } n- p \geq k - \length(^{t}\lambda)+1\\
\label{eqn:main1:3}
&& \mbox{ or equivalently for $n \geq  p+ k - \length(^{t}\lambda)+1$}.
\end{eqnarray}

Now observe that since $\dim(T) \leq d$, $E_2^{p,q} = 0$ for $p \geq d$. Applying this to
\eqref{eqn:main1:3}, we get that

\begin{eqnarray*}
\mult_{\mathbb{S}^\lambda} (E_2^{p,n-p}) &=&  0,
\mbox{ for $n \geq  k + d - \length(^{t}\lambda)+1$}.
\end{eqnarray*}
This completes the proof of Part \eqref{itemlabel:thm:main1:b}.
\end{proof}

\section{Proof of Theorem~\ref{thm:alg}}
\label{sec:proofs2}
In this section we prove Theorem~\ref{thm:alg} by describing an algorithm for efficiently computing 
the first $\ell+1$ Betti numbers  of any given  symmetric semi-algebraic subset  of $\R^k$ defined by symmetric polynomials
of degrees bounded by $d$,  having complexity bounded by a  polynomial in $k$ (for fixed $d$ and $\ell$). \\

We first outline our method. 
\subsection{Outline of the proof of Theorem~\ref{thm:alg}}
\label{subsec:alg:outline}
We first use a construction due to Gabrielov and Vorobjov discussed in Section~\ref{subsec:GV} below
to reduce to the situation where the given symmetric semi-algebraic set is closed and bounded. 
We then use Theorem~\ref{thm:Davis} to decompose 
the task of computing $b_i(S) = \dim_\Q \HH^i(S)$ into two parts: 
\begin{enumerate}
\item
\label{itemlabel:alg-outline:A}
computing the
dimensions of $\HH^i(S_k,S_k^T)$;
\item
\label{itemlabel:alg-outline:B}
computing the isotypic decompositions of the modules 
$\Psi^{(k)}_T$ for various subsets $T \subset \Coxeter(k)$. 
Notice that using  Theorem \ref{thm:main1}, in order to compute 
$b_i(S)$ for  $i \leq \ell$,
we need to compute isotypic decompositions of $\Psi^{(k)}_T$ 
 with 
 $\card(T) < \ell + 2d -1$.
 \end{enumerate}
 
We first describe an algorithm (cf. Algorithm \ref{alg:mult})  for computing the 
isotypic decomposition of $\Psi^{(k)}_T$, which has complexity polynomially bounded in $k$ if $\card(T)$ 
is bounded by 
$\ell + 2d -1$ 
(considering $\ell$ and $d$ to be fixed). The key ingredient for this algorithm is
Corollary~\ref{cor:Solomon-to-Specht}
which allows a recursive scheme to be used for computing the decomposition. 
The fact that we need to consider only subsets $T$ of small cardinality (using Theorem \ref{thm:main1}) is key
in keeping the complexity bounded by a polynomial.
This accomplishes task \eqref{itemlabel:alg-outline:B}. \\

We next address  task  \eqref{itemlabel:alg-outline:A}. 
We first prove that
that the cohomology groups of the pair $(S_k,S_k^T)$ are isomorphic to those of another semi-algebraic
pair $\left(\widetilde{S^{(T)}_k},\widetilde{S_k^T}\right)$ (cf. Proposition \ref{prop:alg}).
Proposition \ref{prop:alg} is the key mathematical result behind our algorithm.
The advantage of the pair 
$\left(\widetilde{S^{(T)}_k},\widetilde{S_k^T}\right)$ over the original pair  $(S_k,S_k^T)$ is that 
$\widetilde{S^{(T)}_k},\widetilde{S_k^T}$ are subsets of an $O(d+\ell)$-dimensional space
(unlike $S_k, S_k^T$ which are subsets of $\Weyl^{(k)} \subset \R^k$). Moreover, a semi-algebraic description of 
$\left(\widetilde{S^{(T)}_k},\widetilde{S_k^T}\right)$ can be computed efficiently (i.e. with polynomially bounded complexity) from that
of the pair $(S_k,S_k^T)$ using a slightly modified version of efficient quantifier elimination algorithm over reals
(cf. Algorithm~\ref{alg:tilde}). The number and the degrees of the polynomials appearing in the description
of $\left(\widetilde{S^{(T)}_k},\widetilde{S_k^T}\right)$ are bounded by a polynomial in $k$ (for fixed $d$ and $\ell$). Finally,
we compute the Betti numbers of the pair $\left(\widetilde{S^{(T)}_k},\widetilde{S_k^T}\right)$ using effective algorithms for computing semi-algebraic triangulations (cf. Algorithm \ref{alg:betti}). We exploit the fact that this is now a constant (i.e. $O(d+\ell)$) dimensional problem, and we can use algorithms which have doubly exponential complexity in the number of variables without affecting  the overall polynomial complexity of our algorithm.

\subsection{Replacing an arbitrary semi-algebraic set by a closed and bounded one}
\label{subsec:GV}
We recall a fundamental construction due to Gabrielov and Vorobjov
\cite{GV07}
which allows us to reduce to the case when the given symmetric semi-algebraic set is closed and bounded. \\

We first need some preliminaries.
We recall some basic facts about real closed fields and real
closed extensions.

\subsubsection{Real closed extensions and Puiseux series}
We will need some
properties of Puiseux series with coefficients in a real closed field. We
refer the reader to \cite{BPRbook2} for further details.

\begin{notation}
  For $\R$ a real closed field we denote by $\R \left\langle \eps
  \right\rangle$ the real closed field of algebraic Puiseux series in $\eps$
  with coefficients in $\R$. We use the notation $\R \left\langle \eps_{1},
  \ldots, \eps_{m} \right\rangle$ to denote the real closed field $\R
  \left\langle \eps_{1} \right\rangle \left\langle \eps_{2} \right\rangle
  \cdots \left\langle \eps_{m} \right\rangle$. Note that in the unique
  ordering of the field $\R \left\langle \eps_{1}, \ldots, \eps_{m}
  \right\rangle$, $0< \eps_{m} \ll \eps_{m-1} \ll \cdots \ll \eps_{1} \ll 1$.
\end{notation}

Let $\mathcal{P} \subset \R[X_1,\ldots,X_k]$,
$S$ be a $\mathcal{P}$-semi-algebraic set defined by a $\mathcal{P}$-formula $\Phi$. 
Without loss of generality we can suppose that 
\[
\Phi = \Phi_1 \vee \cdots \vee \Phi_N,
\]
where for $1 \leq i \leq N$, 
\[
\Phi_i =  \bigwedge_{P \in \mathcal{P}_{i,0}} (P=0) \wedge  \bigwedge_{P \in \mathcal{P}_{i,1}} (P>0) \wedge  \bigwedge_{P \in \mathcal{P}_{i,-1}} (P<0),
\]
where $\mathcal{P}_{i,0}, \mathcal{P}_{i,1},\mathcal{P}_{i,-1}$ is a partition of the set $\mathcal{P}$. \\

For $\eps,\delta >0$ we denote
\[
\Phi_{i,\eps,\delta} = \bigwedge_{P \in \mathcal{P}_{i,0}} ((P - \eps \leq 0 ) \wedge (P + \eps \geq 0)) \wedge  \bigwedge_{P \in \mathcal{P}_{i,1}} (P- \delta \geq 0) \wedge  \bigwedge_{P \in \mathcal{P}_{i,-1}} (P+\delta \leq 0),
\]
and
\[
\Phi_{\eps,\delta} = \bigwedge_{i=1}^N \Phi_{i,\eps,\delta}.
\]

Gabrielov and Vorobjov \cite{GV07} proved the following theorem.
\footnote{The theorem in \cite{GV07} is not stated using the language of non-archimedean extensions and 
Puiseux series but it is easy to translate it into the form stated here.}

\begin{theorem*}\cite[Theorem 1.10]{GV07}
Let $\mathcal{P} \subset \R[X_1,\ldots,X_k]$ and
$S = \RR(\Phi)$, where $\Phi$ is a $\mathcal{P}$-formula. 
For $0 \leq m \leq k$, let
\begin{equation}
\label{eqn:GV}
\widetilde{\Phi}_m = \left(\bigvee_{0 \leq j \leq m} \Phi_{\eps_j,\delta_j} \right) \wedge (\eps(X_1^2 +\cdots+X_k^2) -1 \leq 0)),
\end{equation}
and let $
S'_m
= \RR\left(\widetilde{\Phi}_m\right) \subset \R\la\eps,\eps_0,\delta_0,\cdots,\eps_m,\delta_m\ra^k$.
Then, 
\[
\HH^i(S) \cong 
\HH^i(S'_m)
\] 
for $0 \leq i < m$. 
\end{theorem*}

\begin{remark}\label{rem:GV}
Observe that 
$S'_m$
is a bounded $\widetilde{\mathcal{P}}_m$-closed semi-algebraic set,
where 
\[
\widetilde{\mathcal{P}}_m = \bigcup_{P \in \mathcal{P}}\bigcup_{0 \leq i \leq m}
\{ P \pm \eps_i, P \pm \delta_i \}  \cup \{\eps \sum_i X_i^2 - 1\}.
\]

Moreover, if $\mathcal{P} \subset \R[X_1,\ldots,X_k]^{\mathfrak{S}_k}_{\leq d}, d \geq 2$, then
\[
\widetilde{\mathcal{P}}_m \subset \R\la\eps,\eps_0,\delta_0,\ldots,\eps_m,\delta_m\ra[X_1,\ldots,X_k]^{\mathfrak{S}_k}_{\leq d},
\]
and $\card(\widetilde{\mathcal{P}}_m) = 4m \cdot \card(\mathcal{P}) +1$. \\

In our algorithmic application (cf. Algorithm~\ref{alg:betti} below)
we will replace the given semi-algebraic set  $S \subset \R^k$ 
by the closed and bounded semi-algebraic set 
$
S'_{\ell+1}
\subset \R\la\eps,\eps_0,\delta_0,\ldots,\eps_{\ell+1},\delta_{\ell+1}\ra^k$.
By the preceding theorem the first $\ell+1$ Betti numbers of $S$ and 
$S'_{\ell+1}$
are equal.
Moreover, the number of infinitesimals appearing in the definition of 
$S'_{\ell+1}$
is bounded by $O(\ell)$.
The number of infinitesimals used to make the deformation from $S$ to 
$S'_{\ell+1}$
is important 
for analyzing the complexity of our algorithms. In our algorithms, we will extend the given ring of coefficients
to a polynomial ring in these infinitesimals. As a result each arithmetic operation in this larger ring needs 
several operations to be performed in the original ring -- and this added cost enters as a multiplicative factor
in the complexity upper bounds (see proof of Proposition \ref{prop:alg:betti}).
\end{remark}

\subsection{Computing the isotypic decomposition of $\Psi^{(k)}_T$}
\label{subsec:multi}
We now describe more precisely our algorithm for computing the multiplicities of various Specht modules in the 
representations $\Psi^{(k)}_T$.

\begin{algorithm}[H]
\caption{(Computing  isotypic decomposition of $\Psi_T^{(k)}$)}
\label{alg:mult}
\begin{algorithmic}[1]
\INPUT
\Statex{
An integer $k \in \Z_{>0}$, and $T \subset \Coxeter(k)$.
 }
 \OUTPUT
 \Statex{
 \begin{enumerate}
 \item
 The set 
 $\Par(k,T) = \{ \lambda \vdash k \mid  \mult_{\mathbb{S}^\lambda}(\Psi_T^{(k)}) \neq 0\}$;
 \item
 $\mult_{\mathbb{S}^\lambda}(\Psi_T^{(k)})$ for each $\lambda \in \Par(k,T)$.
 \end{enumerate}
 }

\PROCEDURE
\If {$T = \emptyset$}
\State {Output  $\Par(k,T) = \{(k)\}$, and $\mult_{\mathbb{S}^{(k)}}(\Psi_T^{(k)})  = 1$ and terminate.} 
\Else
	\If{$k=2$}
		\State {output $\Par(k,T) = \{(1,1)\}$, and $\mult_{\mathbb{S}^{(1,1)}}(\Psi_T^{(k)})  = 1$ and terminate.}
      \EndIf
\EndIf
\For {$\lambda \vdash k, \length(\lambda) \leq \card(T) +1$ }
	\State{ $m_\lambda \gets 0$.}
\EndFor

\State{$ P_T \gets \emptyset$.}
\State {$q \gets \max\{j \mid s_j \in T\}$.}
\State{$Q' \gets T -\{s_q\}$.}
\State{$P' \gets \{s_1,\ldots,s_{q-1}\} - T$.}
\State{$P'' \gets \{s_{q+1},\ldots,s_{k-1}\}$.}

\State{Using a recursive call to Algorithm~\ref{alg:mult} with input $q$ and $Q'$, compute $\Par(q,Q')$ and $\mult_{\mathbb{S}^\mu}(\Psi_{Q'}^{(q)})$  for each $\mu \in \Par(q,Q')$.
}  \label{line:alg:mult:call1}
\For {$\mu \in \Par(q,Q')$}  \label{line:alg:mult:call2}
	\For {$\lambda \in S(\mu,k)$} (cf. Notation~\ref{not:Pieri})
		\State {$P_T \gets P_T \cup \{\lambda\}$.}
		\State { $m_\lambda \gets m_\lambda + \mult_{\mathbb{S}^\mu}(\Psi_{Q'}^{(q)})$.}
	\EndFor
\EndFor

\State {Using a recursive call to Algorithm~\ref{alg:mult} with input $k$ and $P'$, compute $\Par(k,Q')$ and 
$\mult_{\mathbb{S}^\lambda}(\Psi_{P'}^{(k)})$ for each $\lambda \in \Par(k,Q')$.}  \label{line:alg:mult:call3}

\For {$\lambda \in \Par(k,Q')$} \label{line:alg:mult:call4}
	\State { $m_\lambda \gets m_\lambda - \mult_{\mathbb{S}^\lambda}(\Psi_{Q'}^{(k)}).$}
	\If {$m_\lambda = 0$}  
	 	\State {$P_T \gets P_T \setminus \{\lambda\}$.}
	\EndIf
\EndFor
\State{Output $\Par(k,T) = P_T$, and for each $\lambda \in \Par(k,T)$, output $\mult_{\mathbb{S}^\lambda}(\Psi_T^{(k)}) = m_\lambda$.}

\end{algorithmic}
\end{algorithm}

\begin{proof}[Proof of correctness of Algorithm~\ref{alg:mult}]
The correctness of the algorithm follows from 
Corollary~\ref{cor:Solomon-to-Specht}, 
and Lemma~\ref{lem:Pieri}.
\end{proof}

\begin{proof}[Complexity Analysis of Algorithm~\ref{alg:mult}]
Let $F(k,n)$ denote the maximum of the complexity of the algorithm over all inputs ($k, T)$,  where $\card(T) = n$.
Then, $F(k,n)$ is also an upper bound on the 
cardinality of the set $\Par(k,T)$ produced in the 
output of the algorithm.
First consider the recursive call to the algorithm in Line \ref{line:alg:mult:call1}. 
The complexity of computing $\Par(q,Q')$ as well as the cardinality of the set $\Par(q,Q')$ is bounded by
$F(q,n-1) \leq F(k-1,n-1)$.
Also observe that for each $\mu$  belonging to the output $\Par(q,Q')$ of this recursive call 
$\length(\mu) \leq \card(Q')+1 \leq n$, which is a consequence of 
Proposition~\ref{prop:Solomon}.
The cardinality of the set $S(\mu,k)$ is bounded by $O(k^{\length(\mu)}) = O(k^n)$
(using Lemma~\ref{lem:Pieri-quantitative}).
The complexity of 
computing $S(\mu,k)$ is also bounded by $k^{O(n)}$.
Thus the total cost of the `\textbf{for}'  loop in Line \ref{line:alg:mult:call2} is bounded by 
$ k^{C n} F(k-1,n-1)$ for a large enough constant $C > 0$.
The cost of the recursive call  in Line \ref{line:alg:mult:call3} is bounded by $ F(k,n-1)$, and the cost of
the `\textbf{for}' loop in Line \ref{line:alg:mult:call4} is bounded by $C F(k,n-1)$
for a large enough constant $C > 0$.
Thus the function $F(k,n)$ satisfies the following inequalities for large enough constants $C, C' > 0$:
\begin{eqnarray*}
F(k,0) &\leq & C,\\
F(2,\cdot) &\leq & C ,\\
F(k,n) &\leq & k^{C n}F(k-1,n-1) + C F(k,n-1) \\
&\leq&  k^{C' n} F(k,n-1).
\end{eqnarray*}

It follows from the above inequalities that there exists some constant $C'' > 0$ such that
\[
F(k,n) \leq k^{C'' n^2} .
\]
Thus the complexity of Algorithm~\ref{alg:mult} is bounded by $k^{O(\card(T)^2)}$.
\end{proof}

We summarize the above in the following proposition.

\begin{proposition}
\label{prop:alg:mult}
Algorithm \ref{alg:mult} is correct and has complexity, measured by the number of arithmetic operations
in $\Z$,  bounded by  
$k^{O(\card(T)^2)}$.
Moreover, the cardinality of the set $\Par(k,T)$  output is also bounded by
$k^{O(\card(T)^2)}$.
\end{proposition}

\begin{proof}
Follows from the proof of correctness and the complexity analysis of Algorithm~\ref{alg:mult} given previously.
\end{proof}

\subsection{The pair $\left(\widetilde{S^{(T)}_k},\widetilde{S_k^T}\right)$ and its properties}
\label{subsec:tilde}
In this section we define the pair $\left(\widetilde{S^{(T)}_k},\widetilde{S_k^T}\right)$, and prove its key property.

\begin{notation}
For any finite set $T$ and $s \in T$, 
we denote by $\Delta_T \subset \R^T$, the standard simplex in $\R^T$. In other words, 
$\Delta_T$ is  the convex hull of the points $(e_s)_{s \in T}$, 
where $e_s$ is defined by $\pi_t(e_s) = \delta_{s,t}$ where 
for each $t \in T$,  $\pi_t:\R^T \rightarrow \R$ is the projection map on to the $t$-th coordinate. 
For $T' \subset T$, we denote by $\Delta_{T'}$, the convex hull of the points $(e_s)_{s \in T'}$, and call
$\Delta_{T'}$ the face of $\Delta_T$ corresponding to the subset  $T'$.
\end{notation}

\begin{definition}
\label{def:poset-compcat}
Let $k\in \Z_{\geq 0}$, and  $\lambda,\mu \in \Comp(k)$. We denote, $\lambda \prec \mu$, if $\Weyl_{\lambda} \subset \Weyl_{\mu}$.
It is clear that $\prec$ is a partial order on $\Comp(k)$ making $\Comp(k)$ into a poset.
\end{definition}

In the following paragraph we introduce notation two denote certain special subsets of $\Comp(k)$.
Their significance will be clear from
the proposition that follows immediately.
\begin{notation}
For $\lambda = (\lambda_1,\ldots,\lambda_\ell) \in \Comp(k)$, we denote
$\length(\lambda) = \ell$,
and
for $k,d \in \Z_{\geq 0}$,
we denote 
\begin{eqnarray*}
\CompMax(k,d)  &=& \{\lambda=(\lambda_1,\ldots,\lambda_d) \in \Comp(k) \; \mid \;
\lambda_{2i+1} =1, 0 \leq i < d/2 \},\\
\CompMin(k,d)  &=& \{\lambda=(\lambda_1,\ldots,\lambda_d) \in \Comp(k) \; \mid \;
\lambda_{2i} =1, 0 < i \leq d/2 \}.
\end{eqnarray*}
\end{notation}

We state the following  important theorem due to Arnold \cite{Arnold} 
which has been referred to in Example~\ref{eg:V-4-3}. It plays a key role in the proof of 
Proposition~\ref{prop:arnold} below. Since we refer the reader to \cite{BC-selecta} for the proof of
Proposition~\ref{prop:arnold}, we do not use Theorem~\ref{thm:arnold} subsequently in this paper.

\begin{theorem}\label{thm:arnold}\cite[Theorems 5, 6 and 7]{Arnold}
For every $\w \in \R_{\geq 0}^k$, $d,k \geq 0$, 
$d'  = \min(k,d)$,
and $\y \in \R^{d'}$
the function $p_{\w,d+1}^{(k)}$ has exactly one local maximum on
$(\Psi^{(k)}_{\w,d})^{-1}(\y)$,
which furthermore depends continuously on $\y$.

Moreover,  a point $\x \in V_{\w,\y}  \cap \Weyl^{(k)}$ is a local maximum if and only if $\x \in \Weyl_\lambda^{(k)}$ for some 
$\lambda \in \CompMax(k,d')$.
Similarly, 
a point $\x \in V_{\w,\y}  \cap \Weyl^{(k)}$ is a local minimum if and only if $\x \in \Weyl_\lambda^{(k)}$ for some 
$\lambda \in \CompMin(k,d')$.
\end{theorem}

Note that as already noted in \cite{BC-selecta} there is a slight inaccuracy  in 
\cite[Theorem 7]{Arnold} in that the word “minimum” should be replaced by the word “maximum” and vice versa. 
A correct statement and a more detailed proof can be found in 
\cite{Meguerditchian} (Proposition 8). \\

We need some more notation.

\begin{notation}
\label{not:emb}
For $\lambda = (\lambda_1,\ldots,\lambda_\ell)  \in \Comp(k)$, we denote by 
$\iota_\lambda: \Weyl^{(\ell)} \rightarrow \Weyl^{(k)}$ the embedding that takes
$(y_1,\ldots,y_\ell) \in \Weyl^{(\ell)}$ to the point
$(\underbrace{y_1,\ldots,y_1}_{\lambda_1},\ldots, \underbrace{y_\ell,\ldots,y_\ell}_{\lambda_\ell})$.
\end{notation}

\begin{notation}
\label{not:Weyl-T}
We denote by 
\[
\Weyl^{(k)}_d = \bigcup_{\lambda \in \CompMax(k,d)} \Weyl_\lambda.
\]
For $T \subset \Coxeter(k)$ and $d \geq 0$, we denote: 
\begin{eqnarray*}
\Weyl^{(k)}_{T,d}  &=& \iota_{\lambda(T)} (\Weyl^{(\length(\lambda(T)))}_d).
\end{eqnarray*}
\end{notation}

\begin{definition}
\label{def:S-et-al}
For any semi-algebraic set $S \subset \R^k$, $T \subset \Coxeter(k)$, and $d \geq 0$, we set
\begin{eqnarray*}
S_{k} &=& S \cap \Weyl^{(k)}, \\
S_{k,d} &=& S \cap \Weyl^{(k)}_d, \\
S_k^T &=& \Weyl^{(k,T)} \cap S, \\
S_{k,T} &=& \Weyl^{(k)}_T \cap S, \\
S_{k,T,d} &=& S \cap \Weyl^{(k)}_{T,d}. \\
\end{eqnarray*}
\end{definition}

\begin{proposition}\label{prop:arnold}
Let $1 < d$, and $\mathcal{P} \subset \R[X_1,\ldots,X_k]^{\mathfrak{S}_k}_{\leq d}$,
$S \subset \R^k$,  a $\mathcal{P}$-closed and bounded semi-algebraic set, and
$\w \in \R_{>0}^k$. Then the  following holds.

\begin{enumerate}[1.]
\item
\label{itemlabel:prop:arnold:b}
The map $\Psi^{(k)}_{\w,d}$ restricted to $S_{k,d}$ is a semi-algebraic homeomorphism on to its image,
and 

\item
\label{itemlabel:prop:arnold:c}
$\Psi^{(k)}_{\w,d}(S_{k,d}) = \Psi^{(k)}_{\w,d}(S_{k})$.
\end{enumerate}
\end{proposition}
\begin{proof}
Both parts 
follow from the weighted version of Part (1) of Proposition 9 in \cite{BC-selecta}.
\end{proof}

We have the following corollary of Proposition~\ref{prop:arnold} that we will need. 
\begin{corollary}
\label{cor:arnold} 
With the same hypothesis
as in Proposition~\ref{prop:arnold},
for each subset $T \subset \Coxeter(k)$,
 $\Psi^{(k)}_{d}$ restricted to $S_{k,T,d}$ is a semi-algebraic homeomorphism onto its image,
and 
\[
\Psi^{(k)}_{d}(S_{k,T}) = \Psi^{(k)}_{d}(S_{k,T,d}).
\]
\end{corollary}

\begin{proof}
Let $\ell = \length(\lambda(T))$, and $S'_{\ell} = \iota_{\lambda(T)}^{-1}(S_{k,T})$ (cf. Notation~\ref{not:emb}). Then, 
\[
S_{k,T,d} = \iota_{\lambda(T)}(S'_{\ell, d}),
\] 
and  
\[
\Psi^{(k)}_{d}|_{S_{k,T}} = \Psi^{(\ell)}_{\lambda(T),d} \circ \iota_{\lambda(T)}^{-1}.
\]
The corollary now follows from Proposition~\ref{prop:arnold}, and the fact that $\iota_{\lambda(T)}$ is a semi-algebraic homeomorphism on to its image.
\end{proof}

Now, let $1 < d$, and $\mathcal{P} \subset \R[X_1,\ldots,X_k]^{\mathfrak{S}_k}_{\leq d}$,
$S \subset \R^k$,  a $\mathcal{P}$-closed and bounded semi-algebraic set, and $T \subset \Coxeter(k)$.

\begin{definition}
\label{def:tilde}
\[
\widetilde{S^{(T)}_k} = \Psi^{(k)}_{d}(S_{k}) \times \Delta_T \subset \R^d \times \R^T,
\]
 and
\[
\widetilde{S_k^T}  = \bigcup_{T' \subset T} \Psi^{(k)}_{d}(S_{k,T'}) \times \Delta_{T'} \subset \widetilde{S^{(T)}_k}.
\]
\end{definition}

The key property of the pair $\left(\widetilde{S^{(T)}_k}, \widetilde{S_k^T}\right)$ defined above  that will be used later
is the following.

Using the definitions given above we have:
\begin{proposition}
\label{prop:alg}
\[
\HH^*\left(\widetilde{S^{(T)}_k}, \widetilde{S_k^T}\right)\cong \HH^*(S_k,S_k^T).
\]
\end{proposition}

Before proving Proposition~\ref{prop:alg} we recall the notion of the \emph{blow-up complex} of a collection of 
closed and bounded semi-algebraic subsets of $\R^N$.

\begin{definition}[Blow-up complex]
\label{def:blow-up}
Given a finite family $\mathcal{A} = (A_\alpha)_{\alpha \in I}$ of closed and bounded 
semi-algebraic subsets of $\R^N$, we
denote 
\[
\Bl(\mathcal{A}) =  \coprod_{J \subset I} A_J \times \Delta_J/\sim,
\]
where for $J \subset I$, $A_J = \bigcap_{\alpha \in J} A_\alpha$, and $\Delta_J$ is the face of the standard
simplex $\Delta_I \subset \R^I$ (i.e. $\Delta_J = \{ (x_\alpha)_{\alpha \in I} \in \Delta_I \mid  \forall (\alpha \not\in J)  x_\alpha = 0\}$, and $\sim$ is the obvious identification.
\end{definition}

It is an easy consequence of the Vietoris-Begle theorem (see for example \cite[page 344]{Spanier})
that (using the same notation as in Definition~\ref{def:blow-up}) 
the map
\[
\pi: \Bl(\mathcal{A})   \rightarrow  A = \bigcup_{\alpha} A_\alpha, \pi (x;t) = x,
\] 
induces isomorphism between the corresponding cohomology groups.

Moreover, 
if $\mathcal{B} = (B_\alpha)_{\alpha \in I}$ is another family of closed and bounded semi-algebraic sets,
such that for each $\alpha \in I$, $A_\alpha \subset B_\alpha$, then there is an obvious inclusion
$\Bl(\mathcal{A}) \hookrightarrow \Bl(\mathcal{B})$, and we have a commutative diagram,
\[
\begin{tikzcd}
\Bl(\mathcal{A}) \arrow[hookrightarrow]{r}\arrow{d}{\pi} & \Bl(\mathcal{B})\arrow{d}{\pi} \\
A=\bigcup_\alpha A_\alpha \arrow[hookrightarrow]{r} & B=\bigcup_\alpha B_\alpha
\end{tikzcd},
\]
where the horizontal arrows are inclusions. 
This gives a map between the pairs
$(\Bl(\mathcal{B}),\Bl(\mathcal{A})) \rightarrow (B,A)$.
(In particular, note that  if $B_\alpha = B$ for all $\alpha \in I$,
$\Bl(\mathcal{B}) = B \times \Delta_I$.)

\begin{lemma}
\label{lem:hocolim}
The induced homomorphism
\[
\HH^*(B,A) \rightarrow \HH^*(\Bl(\mathcal{B}),\Bl(\mathcal{A}))
\]
is an isomorphism.
\end{lemma}

\begin{proof}
The lemma is a consequence of the `five-lemma', and the fact that the 
induced homomorphisms, $\pi^*: \HH^*(A) \rightarrow \HH^*(\Bl(\mathcal{A})), 
\HH^*(B) \rightarrow \HH^*(\Bl(\mathcal{B}))$ are isomorphisms.
\end{proof}

\begin{proof}[Proof of Proposition~\ref{prop:alg}]
Let $\mathcal{A} = (S_{k, \{s\}})_{s \in T}$, and $\mathcal{B} = (S_k)_{s \in T}$.
Then, using  Lemma~\ref{lem:hocolim} and noting that $S_k^T = \bigcup_{s \in T} S_{k,\{s\}}$,
 we have that 
 \[
 \HH^*(S_k,S_k^T) \cong 
\HH^*(\Bl(\mathcal{B}) ,\Bl(\mathcal{A})).
\]

Moreover, observe that for $T'' \subset T' \subset T$, we have a commutative diagram 
\[
\begin{tikzcd}
S_{k,T'} \arrow[hookrightarrow]{r}\arrow{d}{\Psi^{(k)}_d} & S_{k,T''} \arrow{d}{\Psi^{(k)}_d} \\
\Psi^{(k)}_d(S_{k,T'}) \arrow[hookrightarrow]{r} & \Psi^{(k)}_d(S_{k,T''})
\end{tikzcd}
\]
where the horizontal arrows are inclusions.

This allows us to define a map,
$\Bl(\mathcal{B})\rightarrow \widetilde{S^{(T)}_k}$,
by 
\[
[(x;t)] \mapsto (\Psi^{(k)}_d(x);t),
\] 
where $[(x;t)]$ denotes the equivalence class of $(x;t) \in S_k \times \Delta_T$
under the equivalence relation $\sim$ in the definition
of  $\Bl(\mathcal{B})$ (cf. Definition~\ref{def:blow-up}). 
It is easy to verify that this map is well-defined and also that
it restricts to a map  $\Bl(\mathcal{A}) \rightarrow \widetilde{S_k^T}$.

Hence we have a induced map of pairs 
\begin{equation}
\label{eqn:prop:alg:map}
(\Bl(\mathcal{B}), \Bl(\mathcal{A})) \rightarrow \left(\widetilde{S^{(T)}_k}, \widetilde{S_k^T}\right).
\end{equation}
The fibers of the maps  
$\Bl(\mathcal{B})\rightarrow \widetilde{S^{(T)}_k}$,
$\Bl(\mathcal{A})\rightarrow \widetilde{S_k^T}$, defined above
are weighted Vandermonde varieties inside Weyl chambers and
are thus contractible using Proposition~\ref{prop:agk}.
Hence  the induced homomorphisms,
$\HH^*\left(\widetilde{S^{(T)}_k}\right)  \rightarrow \HH^*(\Bl(\mathcal{B}))$,
$\HH^*\left(\widetilde{S_k^T}\right) \rightarrow \HH^*(\Bl(\mathcal{A})$ are isomorphisms.

Using the `five lemma' we obtain that the homomorphism,
\[
\HH^*\left(\widetilde{S^{(T)}_k},\widetilde{S_k^T}\right) \rightarrow \HH^*(\Bl(\mathcal{B}) ,\Bl(\mathcal{A}))
\]
induced by the map in 
\eqref{eqn:prop:alg:map} 
is an isomorphism.
This proves the proposition.
\end{proof}

\subsection{Algorithm for computing a semi-algebraic description of the pair 
$\left(\widetilde{S^{(T)}_k},\widetilde{S_k^T}\right)$}
\label{subsec:alg:tilde}
We now describe an efficient algorithm which takes as input the semi-algebraic description of a symmetric
semi-algebraic subset $S \subset \R^k$, which uses symmetric polynomials of degree at most $d$, and 
produces semi-algebraic descriptions of $\widetilde{S^{(T)}_k}$ and $\widetilde{S_k^T}$.

\begin{algorithm}[H]
\caption{(Computing  semi-algebraic descriptions  of $\left(\widetilde{S^{(T)}_k},\widetilde{S_k^T}\right)$}
\label{alg:tilde}
\begin{algorithmic}[1]
\INPUT
\Statex{
\begin{enumerate}
\item Integers $k,d \geq 0$, $d \leq k$;
\item
a finite set $\mathcal{P} \subset \D[X_1,\ldots,X_k]^{\mathfrak{S}_k}_{\leq d}$;
\item
a $\mathcal{P}$-closed formula, $\Phi$ such that $\RR(\Phi) = S$;
\item
$T \subset \Coxeter(k)$.
\end{enumerate}
 }
 \OUTPUT
 \Statex{
 \begin{enumerate}
 \item An ordered domain $\D$ contained in a real closed field $\R$;
 \item
 A finite family of polynomials $\widetilde{\mathcal{Q}} \subset \D[(Y_s)_{s \in T}, Z_1,\ldots,Z_d]$;
 \item
 $\widetilde{\mathcal{Q}}$ formulas, $\widetilde{\Phi^{(T)}_k}$ and $\widetilde{\Phi^T_k}$, such that
 $\RR(\widetilde{\Phi^{(T)}_k}) = \widetilde{S^{(T)}_k}$ and $\RR(\widetilde{\Phi^T_{k}}) = \widetilde{S_k^T}$.
 \end{enumerate}
 }

\PROCEDURE

\For  {$\lambda \in \CompMax(k,d)$}  \label{line:alg:tilde:for:1}
\State{Using the algorithm from \cite[Corollary 6]{BC-selecta} applied to the family $\mathcal{P}$, the formula
$\Phi \wedge \bigwedge_{1 \leq i \leq k-1} (X_i \leq X_{i+1})$, and the linear equations defining 
the subspace $L_\lambda$ containing the face $\Weyl_\lambda$,
and the polynomial map $\Phi^{(k)}_d$,
obtain
a family of polynomials formula $\mathcal{Q}_\lambda \subset \R[Z_1,\ldots,Z_d]$, and 
$\mathcal{Q}_\lambda$-formula $\Phi_\lambda$, such that 
$\RR(\Phi_\lambda) = \Psi^{(k)}_d(S \cap \Weyl_{\lambda})$.} \label{line:alg:tilde:1}
\EndFor

\State {$\Theta \gets (\sum_{s \in T} Y_s -1 = 0) \wedge \bigwedge_{s \in T} (Y_s  \geq 0)$.}
\State {$\widetilde{\mathcal{Q}} \gets 
\{\sum_{s \in T} Y_s -1 \} \cup \bigcup_{s \in T } \{Y_s \}
\cup \bigcup_{\lambda \in \CompMax(k,d)} \mathcal{Q}_\lambda $.}
\State {$ \widetilde{\Phi^{(T)}_k} \gets \Theta \wedge  \bigvee_{\lambda \in \CompMax(k,d)} \Phi_\lambda $.}
\label{line:alg:tilde:2}

\For {$T' \subset T$}  \label{line:alg:tilde:for:2}
 	\For {$\mu \in \CompMax(\length(\lambda(T')),d)$}   \label{line:alg:tilde:for:3}
		\State{Using the algorithm from \cite[Corollary 6]{BC-selecta} applied to the family $\mathcal{P}$, the formula
		$\Phi \wedge \bigwedge_{1 \leq i \leq k-1} (X_i \leq X_{i+1})$, the linear equations defining the subspace 
		the face $\iota_{\mu} (\Weyl^{(\length(T'))}_{\lambda})$,
		and the polynomial map $\Phi^{(k)}_d$,
		obtain a family of polynomials formula $\mathcal{Q}_{T',\mu}  \subset \R[Z_1,\ldots,Z_d]$, and 
		$\mathcal{Q}_{T', \mu}$-formula $\Phi_{T',\mu}$, such that 
		$\RR(\Phi_{T',\mu} ) = \Phi^{(k)}_d(S \cap \iota_{\mu} (\Weyl^{(\length(T'))}_{\mu})$.}\label{line:alg:tilde:3}

	\EndFor

	 		%%sb pagebreak
%%%%%%%%%%%%%%%%%%%%%%%%%%%%%%%%%%%%%%%%%%%%%%%%%%%%%%
%%%%%%%%%%%%%%%%%%%%%%%%%%%%%%%%%%%%%%%%%%%%%%%%%%%%%%
\algstore{myalg}
\end{algorithmic}
\end{algorithm}
 
\begin{algorithm}[H]
\begin{algorithmic}[1]
\algrestore{myalg}
%%%%%%%%%%%%%%%%%%%%%%%%%%%%%%%%%%%%%%%%%%%%%%%%%%%%%%
%%%%%%%%%%%%%%%%%%%%%%%%%%%%%%%%%%%%%%%%%%%%%%%%%%%%%%
%%sb end of pagebreak

	\State {$\Phi_{k,T'} = \bigvee_{\mu \in \CompMax(\length(\lambda(T')),d)} \Phi_{T',\mu} \wedge (\sum_{s \in T'}Y_s -1=0) \wedge \bigwedge_{s \in T- T'} (Y_s = 0)$.}
\EndFor

\State{ \[
\widetilde{\mathcal{Q}} \gets \widetilde{\mathcal{Q}} \cup \bigcup_{T' \subset T} \bigcup_{\mu \in \CompMax(\length(\lambda(T')),d)} \mathcal{Q}_{T', \mu}.
\]
}

\State{
$\widetilde{\Phi_k^T} \gets \bigvee_{T' \subset T} \Phi_{k,T'}$.}\label{line:alg:tilde:4}
\end{algorithmic}
\end{algorithm}

\begin{proposition}
\label{prop:alg:tilde}
Algorithm \ref{alg:tilde} is correct and its complexity, measured by the number of arithmetic operations in the
domain $\D$, is bounded by 
\[
(s k d)^{O(d +\card(T))}.
\] 
Moreover, 
$\card(\widetilde{\mathcal{Q}}) \leq (s k d)^{O(d +\card(T))}$, and the degrees of the polynomials in 
$\widetilde{\mathcal{Q}}$ are bounded by $d^{O(d +\card(T))}$.
\end{proposition}

\begin{proof}
It follow from Proposition~\ref{prop:arnold} and  \cite[Corollary 6]{BC-selecta},
that the first order formulas $\Phi_\lambda, \lambda \in \CompMax(k,d)$, 
computed  in Line  \ref{line:alg:tilde:1} of Algorithm
\ref{alg:tilde} have the property that
\[
\RR\left(\bigvee_{\lambda \in \CompMax(k,d)} \Phi_\lambda\right) =  \Phi^{(k)}_d(S_k).
\]
It now follows from the definition of $\widetilde{S^{(T)}_k}$ (cf. Definition~\ref{def:tilde}),
that the formula $\widetilde{\Phi^{(T)}_k}$ computed in Line~\ref{line:alg:tilde:2} in Algorithm
\ref{alg:tilde} satisfies
\[
\RR\left(\widetilde{\Phi^{(T)}_k}\right) = \widetilde{S^{(T)}_k}.
\]

Similarly, it  follows from Corollary~\ref{cor:arnold}, and  \cite[Corollary 6]{BC-selecta},
that the first order formulas $\Phi_{T',\mu}, 
\mu \in \CompMax(\length(\lambda(T')),d)$
computed  in Line  \ref{line:alg:tilde:3} of Algorithm
\ref{alg:tilde} have the property that,
\[
\RR\left(\bigvee_{\mu \in \CompMax(\length(\lambda(T')),d)} \Phi_{T',\mu}\right) =  \Phi^{(k)}_d(S_{k,T',d}).
\]
It now follows from the definition of $\widetilde{S_{k,T}}$ (cf. Definition~\ref{def:tilde}),
that the formula $\widetilde{\Phi^T_{k}}$ computed in Line~\ref{line:alg:tilde:4} of Algorithm
\ref{alg:tilde} satisfies
\[
\RR\left(\widetilde{\Phi^T_k}\right) = \widetilde{S_{k}^T}.
\]

This completes the proof of the correctness of Algorithm~\ref{alg:tilde}. 
 The complexity upper bound is a consequence of the complexity bound in \cite[Corollary 6]{BC-selecta},
 and the following:
 \begin{enumerate}[(i)]
 \item
  the number of iterations of the `\textbf{for}' loop in Line~\ref{line:alg:tilde:for:1} is bounded by 
 \[
 \card(\CompMax(k,d)) \leq k^{O(d)};
 \]
 \item
 the number of iterations of the '\textbf{for}' loop in Line~\ref{line:alg:tilde:for:2} bounded by 
 \[
 2^{\card(T)};
 \] 
 and,
 \item
 the number of iterations of the `\textbf{for}' loop in Line~\ref{line:alg:tilde:for:3}is bounded by
 \[
 \card(\CompMax(\length(T'),d)) \leq k^{O(d)}.
 \] 
 \end{enumerate}
\end{proof}

\subsection{Algorithm for computing the 
the Betti numbers of symmetric semi-algebraic sets}
\label{subsec:betti}

We are now in a position to describe our algorithm for computing 
the first $\ell+1$ Betti numbers of
symmetric semi-algebraic sets
which will finally prove Theorem~\ref{thm:alg}.

\vspace{-0.1cm}
\begin{algorithm}[H]
\caption{(Computing the
first $\ell+1$ cohomology groups of a symmetric semi-algebraic set)}
\label{alg:betti}
\begin{algorithmic}[1]
\INPUT
\Statex{
\begin{enumerate}
\item An ordered domain $\D$ contained in a real closed field $\R$;
\item Integers $k,d,\ell \geq 0$, $\ell, d \leq k$;
\item
a finite set $\mathcal{P} \subset \D[X_1,\ldots,X_k]^{\mathfrak{S}_k}_{\leq d}$;
\item
a $\mathcal{P}$-formula $\Phi$.
\end{enumerate}
 }
 \OUTPUT
 \Statex{
 The integers 
$b_0(\RR(\Phi)),\ldots,b_\ell(\RR(\Phi))$.
}

\PROCEDURE
\State{$\Phi \gets \widetilde{\Phi}_{\ell+1}$ (cf. Eqn. ~\eqref{eqn:GV}).} \label{line:alg:betti:0}
\State{$\D \gets \D' = \D[\eps,\eps_0,\delta_0,\ldots,\eps_{\ell+1},\delta_{\ell+1}]$.} \label{line:alg:betti:0.5}
\State{$ \R \gets \R'  = \R\la\eps,\eps_0,\delta_0,\ldots,\eps_{\ell+1},\delta_{\ell+1}\ra$.}

\For  {$T \subset \Coxeter(k), \card(T) <  \ell+2d -1$} \label{line:alg:betti:for:1}
	\State{Compute using Algorithm~\ref{alg:tilde}, the family of polynomials $\widetilde{\mathcal{Q}}$ and
the formulas $\widetilde{\Phi^{(T)}_k}$ and $\widetilde{\Phi_k^T}$.} \label{line:alg:betti:1}
	\State{Compute a semi-algebraic triangulation $h_T: |K_T| \rightarrow \RR\left(\widetilde{\Phi^{(T)}_k}\right)$, 
	such that $h_T^{-1}\left(\RR\left(\widetilde{\Phi^{(T)}_k}\right)\right) = |K_T'|$, $K_T'$ is a sub-complex of $K_T$, 
	as in the proof of Theorem 5.43 \cite{BPRbook2}.} \label{line:alg:betti:2}
	\State{Compute $b_i\left(\RR\left(\widetilde{\Phi^{(T)}_k}\right), \RR\left(\widetilde{\Phi_k^T}\right)\right) = b_i(K_T,K_T')$ for $0 \leq i \leq \ell$
	(using for example the Gauss-Jordan elimination algorithm from elementary linear algebra).}\label{line:alg:betti:3}
	\State{Compute using Algorithm~\ref{alg:mult}, the set $\Par(k,T)$.}
 	\For {$\lambda  \in \Par(k,T)$}  \label{line:alg:betti:for:4}
		\State{$m_{\lambda,T} \gets \mult_{\mathbb{S}^\lambda}(\Psi^{(k)}_T)$.}
	\EndFor
\EndFor

		%%sb pagebreak
%%%%%%%%%%%%%%%%%%%%%%%%%%%%%%%%%%%%%%%%%%%%%%%%%%%%%%
%%%%%%%%%%%%%%%%%%%%%%%%%%%%%%%%%%%%%%%%%%%%%%%%%%%%%%
\algstore{myalg}
\end{algorithmic}
\end{algorithm}
 
\begin{algorithm}[H]
\begin{algorithmic}[1]
\algrestore{myalg}
%%%%%%%%%%%%%%%%%%%%%%%%%%%%%%%%%%%%%%%%%%%%%%%%%%%%%%
%%%%%%%%%%%%%%%%%%%%%%%%%%%%%%%%%%%%%%%%%%%%%%%%%%%%%%
%%sb end of pagebreak

\For {$0 \leq i \leq \ell$}
	\For {$\lambda \in \Par(k), \length(\lambda) \leq i +2d -1$} \label{line:alg:betti:for:5}
		\State{$m_{i,\lambda} \gets 0$.}
	\EndFor

	\For {$T \subset \Coxeter(k),  \card(T) < i+2d -1$}  \label{line:alg:betti:for:6}
		\For {$\lambda \in \Par(k,T)$}
			\State{$m_{i,\lambda} \gets m_{i,\lambda} + b_i\left(\RR\left(\widetilde{\Phi^{(T)}_k}\right), \RR\left(\widetilde{\Phi_k^T}\right)\right)\cdot m_{\lambda,T}$.}
		\EndFor
	\EndFor

	\State {
		\[
		b_i(\RR(\Phi)) \gets  \sum_{\lambda \in \Par(k), \length(\lambda) \leq i +2d -1} 
m_{i,\lambda}\cdot \dim_\Q \mathbb{S}^\lambda,
		\]
calculating $\dim_\Q \mathbb{S}^\lambda$ using Eqn. \eqref{eqn:hook}.}
\EndFor
\end{algorithmic}
\end{algorithm}

\begin{proposition}
\label{prop:alg:betti}
Algorithm \ref{alg:betti} is correct and has complexity,
measured by the number of arithmetic operations in the
domain $\D$,
bounded by $(s  kd )^{2^{O(d + \ell)}}$. 

If $\D = \Z$, and the bit-sizes of the coefficients of the input is bounded by $\tau$, then the bit-complexity
of  Algorithm \ref{alg:betti} is bounded by 
\[
(\tau s k d)^{2^{O(d+\ell)}}.
\]
\end{proposition}
\begin{proof}
First observe that
 the formula $\widetilde{\Phi}_{\ell+1}$ in 
Line~\ref{line:alg:betti:0} is a $\widetilde{P}_{\ell+1}$-closed formula, where
\[
\widetilde{P}_{\ell+1} \subset \D[\eps,\eps_0,\delta_0,\ldots,\eps_{\ell+1},\delta_{\ell+1}]^{\mathfrak{S}_k}_{\leq d},
\]
and  $S= \RR\left(\widetilde{\Phi}_{\ell+1}\right)$ is closed and bounded. \\

Moreover, using Remark~\ref{rem:GV}, we have that 
\begin{equation}
\label{eqn:alg:mult:GV}
b_i(\RR(\Phi)) = b_i(S), 0 \leq i \leq \ell.
\end{equation}

It follows from Proposition~\ref{prop:alg:tilde}, that  the  pair of formulas $\left(\widetilde{\Phi^{(T)}_k}, \widetilde{\Phi^T_k}\right)$ 
computed in Line~\ref{line:alg:betti:1}  of Algorithm~\ref{alg:betti} has the property that,
\[
\left(\RR\left(\widetilde{\Phi^{(T)}_k}\right),\RR\left(\widetilde{\Phi^T_k}\right)\right) = \left(\widetilde{S^{(T)}_k},\widetilde{S_k^T}\right).
\]
It follows from Proposition~\ref{prop:alg}, that
\[
\HH^*\left(\widetilde{S^{(T)}_k},\widetilde{S_k^T}\right) \cong \HH^*(S_k,S_k^T),
\]
and it follows from 
Theorem 5.43 in \cite{BPRbook2},
that the numbers $b_i\left(\widetilde{S^{(T)}_k},\widetilde{S_k^T}\right) = b_i(S_k,S_k^T)$ are computed correctly in Line~\ref{line:alg:betti:3}  of Algorithm~\ref{alg:betti} (for
$0 \leq i \leq \ell$). \\

It follows from Theorem~\ref{thm:Davis}  that,
\begin{equation}
\label{eqn:proof-alg-betti}
b_i(S) = \sum_{T \subset \Coxeter(k)} b_i(S_k,S_k^T) \cdot \dim \Psi^{(k)}_T.
\end{equation}

It follows from  \eqref{eqn:restriction-on-T} that the sum on the right hand
side of Eqn. \eqref{eqn:proof-alg-betti} needs to  be taken only over those $T \subset \Coxeter(k)$, satisfying
$\card(T) < i+2d-1$,
i.e.
\[
b_i(S) = \sum_{T \subset \Coxeter(k), \card(T) < i+2d -2} b_i(S_k,S_k^T) \cdot \dim \Psi^{(k)}_T.
\]
The correctness of the algorithm now follows from 
Proposition~\ref{prop:alg:mult} and \eqref{eqn:alg:mult:GV}. \\

In order to analyze the complexity, first notice that in  Line~\ref{line:alg:betti:0.5}, the ordered domain
$\D$ is replaced by the ordered domain $\D' = \D[\eps,\eps_0,\delta_0,\ldots,\eps_{\ell+1},\delta_{\ell+1}]$.
Each subsequent arithmetic
operation takes place in the larger domain 
\[
\D' = \D[\eps,\eps_0,\delta_0,\ldots,\eps_{\ell+1},\delta_{\ell+1}].
\] 
Since
the number of arithmetic operations in $\D$ needed for computing the sum and the product of two polynomials
in $\D'$ of degrees bounded by $D$  is at most $D^{O(\ell)}$, and the degrees of the polynomials
in $\D'$ that show up in the intermediate computations are well controlled, it suffices to bound the number 
of arithmetic operations in the new ring $\D'$. \\

The number of iterations of the `\textbf{for}' loop in Line \ref{line:alg:betti:for:1} is bounded by 
$\binom{k-1}{\ell + 2d -2} = k^{O(d+\ell)}$.
In each iteration, notice that  the semi-algebraic sets $\widetilde{S^{(T)}_k}, \widetilde{S_k^T} \subset
\R^{\card(T)} \times \R^d$, and thus the number of variables in the calls to the 
triangulation algorithm in Line~ \ref{line:alg:betti:2}
equals 
$\card(T) + d \leq (\ell +2d -1) +d  = O(\ell +d)$. 
The number of arithmetic operations in $\D'$ in each iteration is thus bounded by 
\[
(\ell s d k)^{2^{O(d +\ell)}}
\]
from the complexity bounds in Propositions~\ref{prop:alg:mult}, \ref{prop:alg:tilde}, and
the complexity 
of the triangulation algorithm. \\

Since, the degrees of the polynomials appearing in the computations are bounded by $d^{2^{O(d+\ell)}}$, it follows that
the number of arithmetic operations in $\D$ is also bounded by 
\[
(\ell s d k)^{2^{O(d +\ell)}}.
\]

It follows from Proposition~\ref{prop:alg:mult}, that the number of iterations of the 
`\textbf{for}' loop in Line \ref{line:alg:betti:for:4}  is bounded by 
$k^{O((d+\ell)^2)}$.
Also, the number of iterations of the 
`\textbf{for}' loop in Line \ref{line:alg:betti:for:5} is bounded by $k^{O(d+\ell)}$ using the trivial upper bound
on the number of partitions of $k$ of length bounded by  
$\ell + 2d -1$
and the
 number of iterations of the 
`\textbf{for}' loop in Line \ref{line:alg:betti:for:6} is bounded by 
$\binom{k-1}{\ell + 2d -2} = k^{O(d+\ell)}$.
Thus, the complexity of the whole algorithm is bounded by 
\[
(\ell s d k)^{2^{O(d +\ell)}}.
\]

The bit bound on the complexity follows in a standard manner by keeping track of the bit-lengths ofthe  integers occurring in the intermediate computations and using standard algorithms for arithmetic over integers. We omit the details.
\end{proof}

\begin{proof}[Proof  of Theorem~\ref{thm:alg}]
The theorem follows directly from Proposition~\ref{prop:alg:betti}. 
\end{proof}

\appendix
\section{}
\label{sec:Appendix}

This section is divided into two subsections.
The first subsection consists of fairly standard material on representation theory of  finite groups,
and in the second subsection we discuss the representation theory of the symmetric groups.
We chose to include this in order to make the paper reasonably self-contained. 
A standard reference for this material for this material is Serre's classic book \cite{Serre-book}. However, in Serre's book the field of scalars is taken (in most parts)  to be algebraically closed.  Since we
consider representations over $\Q$, we refer the reader to the book \cite{Procesi-book} for the basic
results listed below.

\subsection{A quick digest of representation theory of finite groups}
\label{subsec:primer}

In this paper we only consider group representations over the field $\Q$. So all vector spaces 
in the following are finite dimensional $\Q$-vector spaces and all groups are finite.

\begin{definition}[Representations of a group $G$]
\label{def:G-module}
A representation of a group $G$  is a group homomorphism $\rho:G \rightarrow \mathrm{GL}(V)$
for some vector space $V$. The representation $\rho$ induces a left action of 
the group $G$ on the vector space $V$, by $g \cdot v = \rho(g)(v), v \in V$, making $V$ into a \emph{left $G$-module}. We will use  the language of representations and modules interchangeably.

We call $\dim_\Q V$  the \emph{dimension of the representation $\rho$} .
\end{definition}

\begin{definition}[Morphism of representations]
Given two representations $\rho_1: G \rightarrow \GL(V_1), \rho_2: G \rightarrow \GL(V_2)$,
a homomorphism $T: V_1 \rightarrow V_2$, is a \emph{morphism of $G$-modules} (or equivalently, an \emph{intertwining operator}) if it satisfies,
\[
\rho_2(g) \circ T (v_1) = T \circ \rho_1(g) (v_1),
\]
for all $g \in G, v_1 \in V_1$.  

The representations $\rho_1,\rho_2$ are \emph{equivalent}  if there exists a morphism of $G$-modules
$T:V_1 \rightarrow V_2$ which is an isomorphism. 
\end{definition}

Two canonically defined examples will play an important role.

\begin{enumerate}
\item
The one dimensional representation, corresponding to the constant homomorphism
$G \rightarrow \mathrm{Id}_V$, where $V$ is a one-dimensional vector space is called the \emph{trivial 
representation of $G$} (denoted $1_{G}$).
\item
Let $A = \Q[G]$ denote the \emph{group algebra}  of $G$. Then, $A$ has a natural structure of a  left $G$-module. The corresponding representation is called the \emph{regular representation of $G$}.
The dimension of the regular representation of $G$ is clearly equal to the order of the group  $G$.
  \end{enumerate}

\begin{definition} [Direct sums and tensor products of representations $\cdot \oplus \cdot, \cdot \otimes \cdot,
\cdot \boxtimes \cdot$]
Given two representations $\rho_1:G \rightarrow V_1, \rho_2:G \rightarrow V_2$:
\begin{enumerate}
\item 
One defines
the representation $\rho_1\oplus\rho_2:G \rightarrow \GL(V_1 \oplus V_2)$ by
$(\rho_1\oplus \rho_2)(g)(v_1\oplus v_2) = \rho_1(g)(v_1) \oplus \rho_2(g)(v_2)$.
\item
Similarly the representation $\rho_1\otimes \rho_2: G \rightarrow \GL(V_1 \otimes_\Q V_2)$ is defined by
$(\rho_1 \otimes \rho_2)(g)(v_1 \otimes v_2) = \rho_1(g)(v_1) \otimes \rho_2(g)(v_2)$.
\end{enumerate}
Given two representations $\rho_1:G_1 \rightarrow \GL(V_1), \rho_2:G_2 \rightarrow \GL(V_2)$,
we define a representation $\rho_1 \boxtimes \rho_2:  G_1 \times G_2 \rightarrow \GL(V_1\otimes_\Q V_2)$
of the direct product group $G_1 \times G_2$ by
\[
(\rho_1 \boxtimes \rho_2)(g_1,g_2)(v_1 \otimes v_2) = \rho_1(g_1)(v_1) \otimes \rho_2(g_2)(v_2).
\]
\end{definition}

\begin{definition}[Irreducible representations]
\label{def:irreducible}
A representation $\rho:G \rightarrow \GL(V)$ is irreducible if $V$ does not contain a non-zero proper
sub-representation (i.e. a non-zero proper subspace $W \subset V$ which is closed under $\rho(g)$ for every $g \in G$). 
\end{definition}

\begin{lemma}[Schur's Lemma] \cite[Corollary on page 151]{Procesi-book}
\label{lem:Schur}
Let $\rho_1:G \rightarrow V_1, \rho_2:G \rightarrow V_2$ be two irreducible representations, and
$T:V_1\rightarrow V_2$ an intertwining operator. Then $T$  is either $0$ or an isomorphism. 
\end{lemma}

\begin{definition}
\label{def:isotypic}
Let $\alpha$ be an isomorphism class of irreducible representations of a finite group $G$ and let
$M$ be a finite-dimensional $G$-module. Let $M^\alpha$ denote the sum of all submodules
of $M$ isomorphic to $\alpha$. 
We call $M^\alpha$ to be the \emph{isotypic  ccomponent of $M$ of type $\alpha$}.
\end{definition}

With the same notation as in Definition~\ref{def:isotypic}:
\begin{theorem}[Isotypic decomposition] \cite[Theorem, Section 2.3]{Procesi-book}
\label{thm:isotypic}
The isotypic components give a direct sum decomposition of $M$.
\end{theorem} 

Moreover, Lemma~\ref{lem:Schur}  (Schur's Lemma) implies the following.

\begin{theorem}
\label{thm:isotypic-canonical}
Suppose that $M$ and $N$ are two finite dimensional $G$-module, $\alpha$ an isomorphism class of 
irreducible $G$-modules and $f:M \rightarrow N$ a morphism of $G$-modules. Then, there is 
a commutative diagram of $G$-module homomorphisms
\[
\xymatrix{
M \ar[r]^{f} \ar[d] & N \ar[d]\\
M^\alpha \ar[r]^{f|_{M^\alpha}} & N^\alpha
}
\]
where the vertical arrows are canonical projections.
\end{theorem}

Finally, with the same hypothesis as Definition~\ref{def:isotypic}:
\begin{proposition}\cite[Section 2.1, Corollaries 1,2]{Procesi-book}
\label{prop:definition-of-mult}
Each $M^\alpha$ is (non-canonically) isomorphic to the direct sum of $m_\alpha$ copies of 
the irreducible representation of type $\alpha$ for some $m_\alpha \geq 0$. 
\end{proposition}

\begin{definition}[Multiplicity]
\label{def:mult}
We will call the non-negative integer $m_\alpha$ 
that appears in Proposition~\ref{prop:definition-of-mult}
to be \emph{the multiplicity of $\alpha$ in $M$}, and denote 
\[
m_\alpha = \mult_\alpha(M).
\]
It follows from Theorem~\ref{thm:isotypic-canonical} that 
$\mult_\alpha(M)$ is well defined.
\end{definition}

It is obvious that a representation of a group $G$ restricts to a representation of any subgroup of $G$.
It is less obvious how to lift a representation of  a subgroup of $G$ to a representation of $G$ itself.
There is in fact a canonically defined lift which is referred to as the induced representation. 
The notion of induced representations is used in Lemma~\ref{lem:Pieri} in the paper.

The construction of the induced representation is best stated in the language of modules.
 Let $H \subset G$ be a subgroup of $G$, and let $\rho:H \rightarrow V$ be a representation of
$H$. Then $V$ is naturally a left $\Q[H]$-module and $\Q[G]$ a right $\Q[H]$-module. 

\begin{definition}[Induced representation]
\label{def:induced}
We denote by $\ind_{H}^{G} V$ the left $G$-module $\Q[G] \otimes_{\Q[H]} V$ (called the 
\emph{induced representation} of $V$ on $G$). 
\end{definition}

\subsection{Representation theory of symmetric groups}
\label{subsec:Specht}

In this paper we are concerned with the representations of the symmetric groups $\mathfrak{S}_k$
and certain subgroups of the symmetric groups. We state below the main definitions and results related
to this very classical topic.

\begin{notation}[Partitions and compositions]
  \label{not:Partition1} 
  We denote by $\Par(k)$ the set of \emph{partitions} of $k$, where each partition $\lambda \in \Par(k)$
  (also denoted $\lambda \vdash k$) 
  is a tuple  $(\lambda_{1} , \lambda_{2} , \ldots
  , \lambda_{\ell})$, with $\lambda_{1} \geq \lambda_{2} \geq \cdots \geq
  \lambda_{\ell} \geq 1$, and $\lambda_{1} + \lambda_{2} + \cdots + \lambda_{\ell} =k$. We
  call $\ell$ the length of the partition $\lambda$, and denote
  $\length(\lambda) = \ell$.

  A tuple  $(\lambda_{1} , \lambda_{2} , \ldots
  , \lambda_{\ell})$,  with  $\lambda_{1} + \lambda_{2} + \cdots + \lambda_{\ell} =k$ (but not necessarily non-increasing)
  will be called  a \emph{composition}, and we still
  call $\ell$ the length of the composition  $\lambda$, and denote
  $\length(\lambda) = \ell$. The set of of all compositions of $k$ will be denoted by $\Comp(k)$.

  \end{notation}

\begin{notation}[Transpose of a partition]
\label{not:Partition2}
For a partition $\lambda =(\lambda_1,\ldots,\lambda_\ell) \vdash k$, we will denote by $^{t}\lambda$ the \emph{transpose} of $\lambda$.
More precisely,
$^{t}\lambda = (^{t}\lambda_1, \ldots,^{t}\lambda_{\tilde{\ell}})$, where $^{t}\lambda_j = \card(\{i \mid \lambda_i \geq j \})$, and $\tilde{\ell} = \lambda_1$.
\end{notation}

\begin{definition}[Young diagrams]
\label{def:young-diagram}
Partitions are often identified with Young diagrams. We follow the English convention and associate the partition
$\lambda = (\lambda_1,\lambda_2,\ldots)$ with the Young diagram with its $i$-th row consisting of $\lambda_i$ boxes.
Thus, the Young diagram corresponding to the partition $\lambda = (3,2)$ is
\[
\yng(3,2),
\]
the Young diagram  associated to its transpose, $^{t}\lambda = (2,2,1)$,  is 
\[
\yng(2,2,1)
\]
(note  that the Young diagram of $^{t}\lambda$ is obtained by reflecting the Young diagram of $\lambda$ 
about its diagonal).
Thus, for any partition $\lambda$, $\length(\lambda)$  (respectively $\length(^t \lambda)$) equals the number of \emph{rows} 
(respectively columns) of the Young diagram of $\lambda$ (respectively $^t\lambda$).
\end{definition} 

\begin{definition}[Young's tableau]
\label{def:Young-tableau}
A \emph{tableau} on a given Young diagram corresponding to a partition $\lambda \vdash k$ is a filling 
of its squares with $1,\ldots, k$ (with no repetitions).
\end{definition} 

The representation theory of the symmetric groups, $\mathfrak{S}_k$,  is a classical subject (see for example  \cite{Ceccherini-book} for details) and 
it is well known that the irreducible representations  (Specht modules) of $\mathfrak{S}_k$ 
are indexed by partitions of $k$.

These are defined as follows.
\begin{definition}[Specht modules]
\label{def:Specht}
Let $A = \Q[\mathfrak{S}_k]$ be the group algebra of $\mathfrak{S}_k$. Then $A$ is a $\Q$-vector space of 
dimension $k!$ and left multiplication by elements of $\mathfrak{S}_k$ makes $A$ into a $\mathfrak{S}_k$-module
(usually referred to as the regular representation of the group $\mathfrak{S}_k$).

For $\lambda \vdash k$, 
fix a tableau $T$ on the Young diagram of $\lambda$.
Let 
$P_\lambda \subset \mathfrak{S}_k$ be the set of permutations that stabilizes the rows of the tableau $T$,
and similarly $Q_\lambda \subset \mathfrak{S}_k$ be the set of permutations that fixes the columns of $T$. \\

Let
\begin{eqnarray}
\nonumber
a_\lambda &=& \sum_{w \in P_\lambda} w, \\
\label{eqn:Young-symmerizer}
b_\lambda &=& \sum_{w \in Q_\lambda} \sign(w) w, \\
\nonumber
c_\lambda &=&  a_\lambda b_\lambda.
\end{eqnarray} 
Then the left ideal $A c_\lambda$ of $A$ is an irreducible $\mathfrak{S}_k$-module, and we denote it 
$\mathbb{S}^\lambda$ (the Specht module corresponding to $\lambda$). It is easy to check that
for $\lambda = (k)$,
$\mathbb{S}^{(k)}$ is isomorphic to the one-dimensional trivial representation which we will also denote by $1_{\mathfrak{S}_k}$,
and for $\lambda = (1,\ldots,1)$ (often denoted by $1^k$),
the Specht module $\mathbb{S}^{(1^k)}$ is isomorphic to the one-dimensional sign representation which we will also denote by $\mathbf{sign}_k$.
\end{definition}

\begin{definition}[Hook lengths]
Let $B(\lambda)$ denote the set of boxes in the  Young diagram  (cf. Definition \ref{def:young-diagram}) corresponding to a partition 
$\lambda \vdash k$. For a box $b \in B(\lambda)$, the length of the hook of $b$, denoted 
$h_b$ is the number of boxes strictly to the right and below $b$ plus $1$.
\end{definition}

The following classical formula (due to Frobenius)
gives the dimension of the representation $\mathbb{S}^{\lambda}$ in terms of the hook lengths of the partition $\lambda$.

\begin{eqnarray}
\label{eqn:hook}
\dim_\Q \mathbb{S}^{\lambda} &=&  \frac{k!}{\prod_{b \in B(\lambda)} h_b}.
\end{eqnarray}

\section*{Acknowledgment}
We thank the anonymous referees for many suggestions that helped us improve the paper. In particular,
we are grateful to one of the referees for pointing out an error in the proof of Proposition~\ref{prop:Solomon} (which also caused Algorithm~\ref{alg:mult} to be erroneous)  in the
previous version of this paper. 

We also thank the Institut Henri Poincar\'e in Paris  for hosting us
during the period when this work was conceived.

\bibliographystyle{amsplain}
\bibliography{master}

\end{document}